\documentclass[10pt]{amsart}
\usepackage[margin=1.2in]{geometry}
\usepackage{amssymb,amsthm,amsmath}
\usepackage[utf8]{inputenc}
\usepackage{preamble}
\title{{Non-archimedean periods for log Calabi-Yau surfaces} }
\author{Soham Karwa, Jonathan Lai}
\begin{document}

\begin{abstract}

We prove the first instance of a conjecture by Kontsevich-Soibelman \cite{Kontsevich2006} that the non-archimedean period map recovers the analytic periods in the case of log Calabi-Yau surfaces. In particular, we show that the $K$-affine structure, a natural enhancement of the singular integral affine structure on the essential skeleton, determines the isomorphism type of the log Calabi-Yau surface. 

\end{abstract}
\maketitle

\section{Introduction}

\subsection{Main result}

Let $(Y,D)$ be a pair with $Y$ a smooth projective rational surface over $\C$ and $D \in |-K_Y|$ either an irreducible rational nodal curve or a cycle of $n \geq 2$ smooth rational curves. We call such a pair a \textit{Looijenga pair}. The open surface $U = Y\backslash D$ is a \textit{log Calabi-Yau surface} which has a distinguished volume form $\Omega$ that is unique up to scaling. We fix $\Omega$ such that such that $\int_{\gamma_0}\Omega = (2\pi i)^2$ where $\gamma_0$ is the homology class of a 2-torus inside $U$.  We say the pair $(Y,D)$ is \textit{generic} if it contains no smooth rational curves of self-intersection $-2$ disjoint from $D$; similarly we say $U$ is \textit{generic} if it contains no smooth rational curves of self intersection -2. 

Let $U$ be a log Calabi-Yau surface over the field of meromorphic Laurent series, denoted by $\C((t))_\text{mer}$. This gives rise to a family of log Calabi-Yau surfaces over a punctured disk $\Delta^*\subseteq \C^*$ which is meromorphic at the origin. We denote this family by $ \pi: U\rightarrow \Delta^*$.  Let $U_t = \pi^{-1}(t)$ and $\Omega_t$ be the normalised volume form on $U_t$. Then one can define the \textit{period map}
\begin{align}
\begin{split}
    \mathcal{P}: H_2(U_t,\Z) /\Z[\gamma_0]&\longrightarrow \C((t))^\times_{\text{mer}} \\ 
	\gamma_t &\longmapsto \exp\left(2\pi i\frac{\int_{\gamma_t} \Omega_t}{\int_{\gamma_0}\Omega_t}\right). 
\end{split}
\label{1per}
\end{align}

For the rest of the paper, we set $K:= \C((t))$ to be the field of Laurent series. When endowed with the $t$-adic absolute value given by the order of vanishing at $t=0$, $K$ is a complete non-archimedean field. There is a natural inclusion $\C((t))_\text{mer}$ into the field $K$.  The family $U$ can be viewed as a variety over $K$ and we can associate a topological space, $U^\an$ called the \textit{Berkovich analytification of $U$}. This is a partial compactification of the space of real valuations on the function field of $U$ and there exists a map of topological spaces $$\pi : U^\an \longrightarrow \Sk(U)$$ onto a canonical piecewise linear subspace $\Sk(U) \subseteq U^\an$, called the \textit{essential skeleton}. Outside of a discriminant locus $\Gamma\subseteq \Sk(U)$ of codimension 2, $\pi$ is an affinoid torus fibration; this is an analogue of a smooth real torus fibration. We refer to this as a \textit{non-archimedean SYZ fibration}.  The subspace $\Sk(U)$ is homeomorphic to $\R^2$. The fibration induces an integral affine structure with singularities on $\Sk(U)$; locally the integral affine functions are absolute values of invertible analytic functions on the fibers of the fibration.  We construct a fibration such that the smooth locus $\Sk(U)^\sm$ is an integral affine manifold with only focus-focus singularities in \S\ref{logcy}. By taking a suitable quotient of $\pi_*(\Oc_{U^\an}^\times)$, Kontsevich and Soibelman \cite{Kontsevich2006} consider an enhancement of the integral affine structure which they call a $K$-affine structure. We denote this quotient by $\Aff_K$ and it fits into a short exact sequence 
\begin{align*}
	0 \rightarrow K^\times \rightarrow \Aff_K \rightarrow \check{\Lambda} \rightarrow 0
\end{align*}
where $\check{\Lambda}$ is the sheaf of integral cotangent vectors on $\Sk(U)^\sm$. The $K$-affine structure $\Aff_K$ defines an element of $\text{Ext}^1(\check{\Lambda},K^\times) \cong H^1(\Sk(U)^\sm, \Lambda\otimes K^\times)$, where $\Lambda$ is the sheaf of integral tangent vectors on $\Sk(U)^\sm$. Let $\iota : \Sk(U)^\sm \rightarrow \Sk(U)$ denote the natural inclusion. Then the \textit{non-archimedean period map} is defined as the map
\begin{align*}
	\mathcal{P}^\an : H_1(\Sk(U), \iota_*\check{\Lambda})_\text{tf}\longrightarrow K^\times 
\end{align*}
given by pairing with $[\Aff_K]$ where $(\cdot)_\text{tf}$ denotes the torsion-free subgroup. Kontsevich-Soibelman made the following general conjecture for Calabi-Yau varieties: 
\begin{conj}[\cite{Kontsevich2006} Conjecture 10]
	Given a maximally degenerate $n$-dimensional Calabi-Yau variety $U$ over $\C((t))_\text{mer}$, we have a commutative diagram 
    
    \begin{figure}[H]
	\begin{center}
		\begin{tikzcd}
H_1(\Sk(U),i_*\Lambda)_{\text{tf}}\ar[rr,"\mathcal{P}^\an"]\ar[rd]\ar[rd] & & K^\times \\ 
			&H_n(U_t,\Z)/\Z[\gamma_0] \ar[ur,"\mathcal{P}"] & 
		\end{tikzcd}
	\end{center}
	\end{figure}
    where $\gamma_0$ is the class of a real $n$-dimensional torus inside $U_t$. 
\label{perconj}
\end{conj}

\begin{rmk}
    In the statement of the conjecture in \cite{Kontsevich2006}, they also assume the existence of a real $n$-dimensional torus fibration over $\Sk(U)^\sm$ for all $t\neq 0$ sufficiently small and take $\gamma_0$ to be a fundamental class of a fibre of this fibration. 
\end{rmk}
The main theorem of this paper proves this conjecture in the case of log Calabi-Yau surfaces. We first prove a more general result over $K$. Let $(Y,D)$ be a Looijenga pair over $K$ with $D = D_1 + ... + D_m$ an ordering of the irreducible components of $D$. Let
\[D^\perp := \{L\in \Pic(Y):L \cdot D_i = 0 \text{ for all } i \}.\] 

When $U$ is defined over $\C((t))_\text{mer}$, there is an isomorphism $D^\perp \cong H_2(U_t,\Z)/\Z[\gamma_0]$. 
Then the period map can also be realised as \cite[Prop 3.12]{friedman2016geometry} \begin{align}
\begin{split}
    \mathcal{P}: D^\perp &\longrightarrow \Pic^0(D)\cong \G_m\\
    L &\longmapsto L|_D
    \end{split}
    \label{2per}
\end{align}
where the isomorphism $\Pic^0(D) \cong \G_m$ is canonical after a choice of orientation of $D$.  






\begin{thm}
    Let $U$ be a log Calabi-Yau surface over $K$ satisfying Assumption \ref{assume}. Then we have the commutative diagram 
    
      \begin{figure}[H]
	\begin{center}
		\begin{tikzcd}
			H_1(\Sk(U),i_*\Lambda)_{\text{tf}}\ar[rr,"\mathcal{P}^\an"]\ar[rd,"\sim"]\ar[rd] & & K^\times \\ 
			&D^\perp \ar[ur,"\mathcal{P}"] & 
		\end{tikzcd}
	\end{center}
	\end{figure}
	\label{perequalintro}
\end{thm}

\begin{rmk}
    The assumption \ref{assume} in Theorem \ref{perequalintro} roughly says that the log Calabi-Yau surface arises from the generic fibre of a Looijenga pair $(\Yc,\Dc)$ over $R= \C[[t]]$ with a genericity assumption on the special fibre of $(\Yc,\Dc)$. 
\end{rmk}

Our final result is in fact stronger than Theorem \ref{perequalintro}, see Theorem \ref{perequalfinal}. The proof of Theorem \ref{perequalfinal} is done in two steps. First, we prove it for the case of generic log Calabi-Yau surfaces satisfying Assumption \ref{assume} by constructing a non-archimedean SYZ fibration and the $K$-affine structure associated to it. This defines the non-archimedean period map as above and we prove it agrees with the classical period map. Then by analytic continuation, we can uniquely extend the non-archimedean period map to cover the case of all log Calabi-Yau surfaces that satisfy the weaker assumption of Theorem \ref{perequalfinal} and prove that it factors through the period map. Conjecture \ref{perequalintrocor} is an immediate corollary: 

\begin{cor}
    Let $U$ be a log Calabi-Yau surface over $\C((t))_\text{mer}$ satisfying Assumption \ref{assume}. Then we have the commutative diagram 
    
      \begin{figure}[H]
	\begin{center}
		\begin{tikzcd}
			H_1(\Sk(U),i_*\Lambda)_{\text{tf}}\ar[rr,"\mathcal{P}^\an"]\ar[rd,"\sim"]\ar[rd] & & K^\times \\ 
			&H_2(U_t,\Z)/\Z[{\gamma_0}] \ar[ur,"\mathcal{P}"] & 
		\end{tikzcd}
	\end{center}
	\end{figure}
	\label{perequalintrocor}
\end{cor}

By Carlson's theory of extensions of mixed Hodge structures \cite{carlson1985one}, we know the mixed Hodge structure  $H^2(U)$ of the open surface $U$ is classified by the period map. In particular, Theorem \ref{perequalintro} implies the following result by a simple application of the Torelli theorem for log Calabi-Yau surfaces 
\cite{friedman2016geometry,gross_hacking_keel_2015}. 

\begin{cor}
   The $K$-affine structure on $\Sk(U)$ determines the isomorphism type of $U$.
   \label{cor:theorem}
\end{cor}

 \subsection{Motivation}

The use of non-archimedean geometry techniques is motivated by mirror symmetry. Let $X \xrightarrow{\pi} \Delta^*$ be a family of $n$-dimensional Calabi-Yau varieties over the punctured disk $\Delta^*\subseteq \C^*$. We further assume the family is maximally degenerate \ie the monodromy acting on the degree $n$ cohomology of the general fiber is maximally unipotent. Then the Strominger-Yau-Zaslow (SYZ) conjecture \cite{syz} roughly states that the general fiber $X_t = \pi^{-1}(t)$, admits a fibration $\rho: X_t \rightarrow B$, whose base $B$ is a real $n$-dimensional topological manifold and the fibers are special Lagrangian tori away from a discriminant locus $\Gamma$ of codimension $\geq 2$ in $B$.
The mirror partner $\check{X}_t$ of $X_t$ can then be constructed by dualising the torus fibration over the smooth locus
and deforming the dual fibration by so-called quantum corrections.
\begin{figure}[H]
	\begin{tikzcd}
		X_t \ar[rd, "\rho"] & & \check{X}_t \ar[ld,"\check{\rho}_t"]\\
		& B& 
	\end{tikzcd}
	\caption{Strominger-Yau-Zaslow conjecture}
\end{figure}
A fundamental insight of Kontsevich and Soibelman \cite{Kontsevich2006} is that one should be able to construct an analogue of the fibration in the world of non-archimedean geometry. Unlike its archimedean version, the non-archimedean SYZ fibration has been constructed for proper varieties in \cite{nicaise_xu_yu_2019} and extended to certain log Calabi-Yau varieties of interest in Section \ref{logcy}. It is important to note that the non-archimedean version can realise the original goal of constructing a mirror over the complex numbers by non-archimedean GAGA and algebrization techniques. 

Given a (log) Calabi-Yau variety over $K$, there exists a canonical piecewise linear subspace $\Sk(X)\subseteq X^\an$, called the \textit{essential skeleton}, and a strong deformation retraction $\rho: X^\an \rightarrow \Sk(X)$.  The construction of the retraction depends on a choice of degeneration of $X$ and in \cite{nicaise_xu_yu_2019}, the authors prove there is a minimal choice (but not necessarily unique) which gives rise to an affinoid torus fibration away from a discriminant locus in $\Sk(X)$ of codimension $\geq2$. The fibration induces a singular integral affine structure on $\Sk(X)$ given by the pushforward of the sheaf of invertible analytic functions via $\rho$. With some modification, this can be shown to agree with the integral affine structure arising from classical SYZ fibrations in certain cases \cite{mazzon2023toric}. Choosing a different model results in a different integral affine structure. 

 Following ideas in \cite{Kontsevich2006}, we show that the essential skeleton $\Sk(X)$, a much simpler object than $X$, and the $K$-affine structure contain important geometric information about $X$. Corollary \ref{cor:theorem} is the first non-trivial result in this direction.  This is in the spirit of the reconstruction problem in mirror symmetry which says that one should be able to recover the original variety from the base of the SYZ fibration and certain structures on it.  

\subsection{Relations to other works}To the best of our knowledge, the approach of computing periods using non-archimedean techniques in the context of mirror symmetry has not been explored other than the suggestions of \cite{Kontsevich2006}.

There has been work in understanding the relationship between tropical geometry and periods. For example, it is known that minus the valuation of the $j$-invariant of an degenerate elliptic curve over a non-archimedean field coincides with the cycle length of the tropical elliptic curve obtained by tropicalisation \cite{tropj1,tropj2}. This can be upgraded to the case of maximally degenerate abelian varieties as in \cite{karwa}.
The leading term of period integrals of families of varieties has been also computed using tropical geometry: the case of Riemann surfaces was done in \cite{degenrstrop}, toric Calabi-Yau hypersurfaces in \cite{gamma,tropicalperiod}.

The tropicalisation and moment map record the radial coordinate of a subvariety of the torus $(\C^\times)^n$ and gives rise to the notion of an integral affine structure on its image. In a similar vein, the $K$-affine structure remembers the phase which has been studied in 
\cite{phasetropical, psl2phase}; the $K$-affine structure deserves further understanding to enhance many results in tropical geometry concerning how the tropicalisation of a variety describes important information about the variety. Corollary \ref{cor:theorem} is a first step in this direction. Understanding the $K$-affine structure of K3 surfaces and computing periods of K3s will rely heavily on the results in this paper. 

Hosono's conjecture \cite{hosono} predicts that periods on the mirror equal the pairing of certain explicit hypergeometric series with the Chern classes of vector bundles, for Batyrev mirror pairs of toric Calabi-Yau hypersurfaces. By studying the asymptotics of Hosono's conjecture in the large complex structure limit, one arrives at the mirror symmetry $\hat{\Gamma}$-conjecture \cite{iritani2023gamma}. Abouzaid-Ganatra-Iritani-Sheridan \cite{gamma} use the tropicalisation map, which acts as a limiting SYZ fibration, to compute the leading term of these period integrals for toric Calabi-Yau hypersurfaces. It would be interesting to explore how the non-archimedean method could aid in the computation. 

Ruddat and Siebert give a period formula in \cite{ruddatsiebert} for toric degenerations arising from the Gross-Siebert program. In particular, they compute the period integrals of holomorphic volume forms over $n$-dimensional cycles corresponding to tropical one-cycles in the intersection complex of the central fiber. This formula was applied in \cite{lai2023mirror} to compute the periods of mirror families of positive log Calabi-Yau surfaces arising from the mirror construction in \cite{ghk}.

There should exist a notion of the mirror $k$-affine structure, which the present work suggests can be used to construct the mirror. In the spirit of Kontsevich-Soibelman as well as the Gross-Siebert program \cite{grosssiebert}, recent developments in non-archimedean enumerative geometry, as in \cite{yu2016enumeration, yu2, keelyu}, should be a key component.

\subsection{Outline}
In \S\ref{prelim}, we review several preliminaries needed throughout the paper. Here we will introduce all the necessary theory from Berkovich analytic geometry. In \S\ref{logcy}, we introduce log Calabi-Yau surfaces in more detail and construct the degeneration and fibration needed for the proof of our main results. This section only deals with the generic case. In \S\ref{tropconst}, we define the group of tropical cycles $H_1(\Sk(U)^\sm,\check{\Lambda})$ which is the domain of the non-archimedean period map and prove a tropical correspondence result. In \S\ref{Kaff}, we define $K$-affine structures and construct an explicit one that will be used to define the non-archimedean period map for generic log Calabi-Yau surfaces.  In \S\ref{period}, we define the non-archimedean period and compute it for log Calabi-Yau surfaces. We end this section by extending the non-archimedean period map to the case of non-generic log Calabi-Yau surfaces and arguing it coincides with the period map. 

\subsection{Acknowledgements}

The authors would like to express their gratutude to Johannes Nicaise for his support and suggestions during this project. The first author would also like to thank Robert Crumplin and Dhruv Ranganathan for useful conversations. The first author was supported by the Engineering and Physical Sciences Research Council [EP/S021590/1]. The EPSRC Centre for Doctoral Training in Geometry and Number Theory (The London School of Geometry and Number Theory), University College London. The second author  was supported by EPSRC grant EP/S025839/1. 

\section{Preliminaries}
\label{prelim}
We denote by $K = \C((t))$ the field of Laurent series over $\C$; it has the structure of a complete discrete valuation field with uniformiser $t$. Let $R = \C[[t]]$ denote its ring of integers and $k = \C$, its residue field. We denote the discrete valuation by  $\val_K: K^\times \rightarrow \Z$ and define an  absolute value $|a|_K = \exp(-\val_K(a)) $.  We will always take a \textit{K-variety} to mean a normal and  geometrically irreducible $K$-scheme. 

\subsection{Models}
\label{models}

\begin{defn}
\label{modeldef}
	Let $X$ be a normal integral $K$-scheme of finite type. 
	\begin{enumerate}
		\item An $R$-model of $X$ is a proper flat $R$-scheme $\Xc$ of finite type  endowed with an isomorphism $\Xc_K \rightarrow X$;
		\item If $\Xc$ and $\Yc$ are $R$-models of $X$, then a morphism of $R$-models of $X$ is a morphism of $R$-schemes $h: \Xc \rightarrow \Yc$ such that $h_K: \Xc_K \rightarrow \Yc_K$ is an isomorphism that commutes with the isomorphisms $\Xc_K \rightarrow X$ and $\Yc_K \rightarrow X$;
		\item A \textit{simple normal crossings (snc) model} is a regular $R$-model $\Xc$ whose special fibre $\Xc_k$ is a divisor with simple normal crossings singularities. 
	\end{enumerate}
\end{defn}

We record the following definition for use later. 
\begin{defn}
	Let $\Xc$ be a Noetherian scheme and $D$ an effective divisor on $\Xc$, with prime components $D_i$ for $i \in I$. A \textit{stratum} of $D$ is a connected component of the scheme-theoretic intersection $D_J = \cap_{j \in J} D_j$ for some $ J\neq \emptyset\subseteq I$. An \textit{open stratum} is a stratum $S$ minus the union of the prime components of $D$ which do not contain $S$. Given a stratum $S$ of $ D$, let $S^\circ$ denote the open stratum of $S$. 
\end{defn}

\begin{defn}
	An \textit{snc pair} $(\Xc,\mathcal{H})$ consists of a proper regular flat scheme $\Xc$ over $R$ and $\Hc = \sum_{i \in I^\mathbf{h}} \Dc_i$ a sum of effective Cartier divisors on $\Xc$ such that 
	\begin{enumerate}[label=(\roman*)]
		\item each $\Dc_i$ has irreducible support;
		\item the divisor $\Xc_k + \Hc $ is snc. 
	\end{enumerate}
\end{defn}

Given an snc pair $(\Xc,\Hc)$, we will be solely interested in the case when the generic fibre of $\Xc\backslash \Hc$ is a log Calabi-Yau surface. Nevertheless, we make a more general definition of a model relative to an snc divisor: 
\begin{defn}
	Let $X$ be an algebraic variety over $K$. An \textit{snc log-model} of $X$ is an snc pair $(\Xc,\Hc)$ together with an isomorphism $X \cong (\Xc\backslash\Hc )_K$. 
\end{defn}


The existence of an snc log model of $X$ imposes a strict condition on $X$: 
\begin{lem}
An snc log model of $X$ exists if and only if $X$ is smooth. 	
\end{lem}

\begin{proof}
	First suppose $X$ is smooth. By Hironaka's resolution of singularities \cite{hironaka1964resolution}, there exists an snc compactification $Y$ of $X$ \ie a dense open embedding of $X$ into a smooth proper $K$-scheme $Y$ such that the boundary $D = Y\backslash X$ (with its induced reduced structure) is an snc divisor. Then by Nagata compactification \cite[ \href{https://stacks.math.columbia.edu/tag/0F3T}{Section 0F3T}]{stacks}, there is a proper and flat $R$-scheme  $\Yc$ whose generic fibre is isomorphic to $Y$. Let $\Dc$ be the scheme-theoretic closure of $D$ inside $\Yc$. If $\Yc_k + \Dc$ is snc, then we are done. Otherwise, by another application of resolution of singularities, there is a proper birational morphism $h : \Yc' \rightarrow \Yc$ such that $\Yc'$ is regular, $h_K : \Yc'_K \rightarrow \Yc_K$ is an isomorphism and $h^*\Dc + \Yc_k' $ is snc. Then the pair $(\Yc',h^*\Dc)$ is an snc log model of $X$. 
	
	Suppose an snc log model $(\Xc,\Hc)$ of $X$ exists. By definition, $\Xc_K$ is regular and  $K$ is a perfect field so $\Xc_K$ is smooth. Thus $X\subset \Xc_K$ is smooth. 
\end{proof}

Let $(\Xc,\Hc)$ be an snc pair and consider the finite set  $\{\Dc_i\}_{i \in I^\mathbf{v}}$ of irreducible components of the special fibre $\Xc_k^\red$ with reduced scheme structure. For every $i \in I^\mathbf{v}$, let $N_i$ denote the multiplicity of $\Dc_i$ in $\Xc_k$. We refer to the divisors $\Dc_i$ indexed by $i \in I^\mathbf{v}$ as \textit{vertical divisors}. Likewise, the divisors indexed by $i \in I^\mathbf{h}$ are called \textit{horizontal divisors}. This terminology matches up with the usual definitions of horizontal and vertical divisors in the literature \cite{skeltrop}. For every non-empty set $I\subseteq I^\vb\cup I^\hb$, we write $\Dc_I = \cap_{i \in I} \Dc_i$.

\subsection{Berkovich spaces}
\label{berkovich}

Let $X$ be a $K$-variety.

\begin{defn}[\cite{berkspec}]
	The \textit{Berkovich analytification of X} is set-theoretically given by pairs $(\xi_x,|\cdot|_x)$ where: 
	\begin{enumerate}[label=(\roman*)]
		\item $\xi_x$ is a scheme-theoretic point of $X$; 
		\item $|\cdot|_x: \kappa(x) \longrightarrow \R_{\geq 0}$ is an absolute value on the residue field $\kappa(x)$ extending the $t$-adic absolute value $|\cdot|_K$ on $K$.  
	\end{enumerate}
	It comes equipped with the weakest topology such that:
	\begin{enumerate}
		\item the forgetful map $\iota: X^\an \rightarrow X$ is continuous;
		\item for any Zariski open $U  \subseteq X$ and $f\in \Oc_X(U)$, the evaluation map 
 \begin{align*}
			|f| : \iota^{-1}(U) & \longrightarrow \R_{\geq 0} \\ 
			(\xi_x,|\cdot|_x) &\longmapsto |f|_x
		\end{align*}
		is continuous. 
	\end{enumerate}
\end{defn}

\begin{eg}
	When $X = \spec A$ is affine, $X^\an$ is given by the set of multiplicative seminorms on $A$ that extend the absolute value $|\cdot|_K$ on $K$. 
\end{eg}

The topological space $X^\an$ is equipped with a sheaf of analytic functions which makes $\iota: X^\an \rightarrow X$ a morphism of locally ringed spaces. The topological space $X^\an$ is Hausdorff, and it is compact if and only if $X/K$ is proper. The subset of birational points $X^\text{bir}\subseteq X^\an$ is defined as $\iota^{-1}(\eta_X)$ where $\eta_X$ is the generic point of $X$. These points have the form $(\eta_X,|\cdot|)$ where $|\cdot| : K(X)\rightarrow \R_{\geq 0}$ is an absolute value on the function field $K(X)$ of $X$ extending the absolute value on $K$.

\begin{eg}
\label{divmon}
 Let $\Xc$ be a $R$-model of $X$ and $E$ an irreducible component of $\Xc_k$. Denote the generic point of $E$ by $\xi_E$. Then $\Oc_{\Xc,\xi_E}$ is a discrete valuation ring whose valuation $v_E$ is given by the order of vanishing along $E$ and its fraction field  is $K(X)$. The normalised valuation $\frac{1}{N_E}v_E$, where $N_E$ is the multiplicity of $E$ in $\Xc_k$, determines an absolute value on $K(X)$ which extends the absolute value on $K$. This defines a birational point $x$ of $X^\an$ and is called a  \textit{divisorial point} of $X^\an$. 

More generally, given an snc model $\Xc$ of $X$, let the tuple $(E_1,...,E_r)$ denote distinct irreducible components of $\Xc_k$ such that $E =\bigcap_i E_i \neq \emptyset$. Let $\xi_E$ denote the generic point of a connected component of $E$. Then there is a regular system of local parameters $z_1,...,z_r \in \Oc_{\Xc,\xi_E}$, positive integers $N_1,...,N_r$ and a unit $u \in \Oc_{\Xc,\xi_E}^\times$ such that \begin{align*}
	t = uz_1^{N_1}...z_r^{N_r}
\end{align*}
where locally at $\xi_E$, the local equation $z_i=0$ defines $E_i$ and $N_i$ is the multiplicity of $E_i$ in $\Xc_k$. For each choice of tuple $\alpha = (\alpha_1,...,\alpha_r) \in \R_{\geq 0}^r$ such that $\sum_i \alpha_i N_i = 1$, there exists a unique valuation \cite[Proposition 2.4.4]{mustatanicaise}
\begin{align*}
	v_{E,\alpha} : \Oc_{\Xc, \xi_E}\backslash \{0\} \longrightarrow \R 
\end{align*}
such that $v_{E,\alpha}(z_i) = \alpha_i$ for all $i$. This extends the valuation on $K$ since $$v_{E,\alpha}(t) = v(uz_1^{N_1}...z_r^{N_r}) = \sum_i N_i v_{E,\alpha}(z_i) =\sum_i \alpha_iN_i = 1.
$$
The valuation $v_{E,\alpha}$ thus defines a birational point on $X^\an$ that we call a  \textit{monomial point} of $X^\an$. 
We have the inclusions 
\begin{align*}
	X^\text{div} \subseteq X^\text{mon} \subseteq X^\text{bir}\subseteq X^\an,
\end{align*}
where $X^\text{div} $ and $X^\text{mon}$ are the subsets of divisorial and monomial points respectively. By  \cite[Proposition 2.4.9]{mustatanicaise}, the space of divisorial points $X^\text{div}$ is dense in $X^\an$.
\end{eg}

\subsection{Connection to formal schemes}
\label{formal}

\begin{notation}
	If $\Xc$ is a Noetherian $R$-scheme and $C$ is a subscheme of $\Xc$, we let $\widehat{\Xc}_{/C}$ denote the formal completion of $\Xc$ along $C$. If $C$ is equal to the special fibre, we will simply write $\widehat{\Xc}$. 
\end{notation}

\begin{defn}
	We define the following subring $R\{x_1,...,x_n\} \subseteq R[[x_1,..,x_n]]$:
	\begin{align*}
		R\{x_1,...,x_n\} := \{ f = \sum_{I \in \Z_{\geq0}^n}c_I x^I: c_I \in R, |c_I|_K \rightarrow 0 \text{ as } |I| \rightarrow \infty\}
	\end{align*}
	where $x^I = x_1^{I_1}...x_n^{I_n}$.
\end{defn}
Suppose $\Xc$ is a scheme of finite type over $R$ and consider the formal completion $\widehat{\Xc}$ of an $R$-model $\Xc$ of $X$  along its special fibre and its (analytic) generic fibre $\widehat{\Xc}_\eta$ in the category of Berkovich analytic spaces \cite[\S1]{vanish}. This is a compact analytic domain of $X^\an$. We have the following characterisation of the points in the generic fibre:  a point $x = (x,|\cdot|_x) \in X^\an$ belongs to $\widehat{\Xc}_\eta$ if and only if the morphism 

\begin{align*}
	\spec \mathscr{H}(x) \longrightarrow X
\end{align*}
extends to a morphism 
\begin{align*}
	\spec\mathscr{H}(x)^\circ  \longrightarrow \Xc
\end{align*}
where $\mathscr{H}(x)$ is the completion of the residue field $\kappa(x)$ with respect to $|\cdot|_x$ and $\mathscr{H}(x)^\circ$ denotes its valuation ring. If this occurs, the image of the closed point of $\spec \mathscr{H}(x)^\circ$ is called the \textit{centre} (or \textit{reduction}) of $x$ on $\Xc$ and we say that $x$ admits a centre on $\Xc$.

 If $X$ is proper over $K$ then $X^\an = \widehat{\Xc}_\eta$ by the valuative criterion of properness. 
There is an anti-continuous reduction map \begin{align*}
	\red_\Xc : \widehat{\Xc}_\eta \longrightarrow \Xc_k
\end{align*}  
that sends a point $x$ to the centre of $x$ on $\Xc$. 

\begin{eg}
	When $\Xc = \spec(\mathcal{A})$ is affine, then 
	\begin{align*}
		\widehat{\Xc}_\eta = \{ x \in X^\an \ : \ |f(x)| \leq 1 \text{ for all } f \in \mathcal{A} \}. 
	\end{align*}
\end{eg}

\subsection{Skeletons}
\label{skeletons}



\begin{defn}
We define the following simplicial cone complex of an snc pair $(\Xc,\Hc)$: 
	 \begin{align*}
	 	C_{(\Xc,\Hc)} := \left\{\sum_{i \in I^\textbf{v} \cup I^\textbf{h}} a_i \langle {\Dc}_i \rangle \ | \ a_i\geq 0, \bigcap_{i: a_i>0}{\Dc}_i \neq \emptyset \right\} \subset \R_{\geq 0}^{I^\textbf{v}\cup I^\textbf{h}}
	 \end{align*}
		and the polyhedron: 
		\begin{align*}
			\Sigma_{(\Xc,\Hc)} := C_{(\Xc,\Hc)} \cap \left\{\sum_{i \in I^\textbf{v}} N_i a_i = 1\right\}.
		\end{align*}
		The polyhedron $\Sigma_{(\Xc,\Hc)}$ will be called the \textit{skeleton} of the pair $(\Xc,\Hc)$. Note that when $\Hc = 0$, we recover the  definition above of the dual intersection complex of $\Xc$. In this case, we will simply write $\Sigma_\Xc:= D(\Xc_k)$. 
\end{defn}

\begin{rmk}
	The complexes $C_{(\Xc,\Hc)} $ and $\Sigma_{(\Xc,\Hc)} $ are called Clemens cones and polytopes respectively  in \cite{Kontsevich2006, yu2016enumeration}. 
\end{rmk}

We present an alternative combinatorial description when $\Xc_k$ is reduced.  This readily follows from the definition above. 
\begin{lem}
	Let $(\Xc,\Hc)$ be an snc pair and suppose that $N_i=1$ for all $i$. Given $I\subset I^\vb$, define $$\Delta_I := \left\{ \sum_{i \in I} a_i = 1\right\} \subset \R^{I^\vb \cup I^\hb}_{\geq 0}. $$ 
	Then the skeleton $\Sigma_{(\Xc,\Hc)}$ of the pair $(\Xc,\Hc)$ is the simplicial subcomplex of $\R^{I^\vb \cup I^\hb}$ such that:
	\label{description}
	\begin{enumerate}[label=(\roman*)]
		\item $\Delta_{I}$ is a face of  $\Sigma_{(\Xc,\Hc)}$ if and only if $\Dc_I \neq \emptyset$ for $I\subset I^\vb$ 
		\item $\Delta_{I}\times \R^{|J|-1}_{\geq 0}$ is a face of $\Sigma_{(\Xc,\Hc)}$ if and only if $\Dc_{I} \cap \Dc_{J}\neq \emptyset $ for $I\subset I^\vb$  and $J\subset I^\hb$ with $J\neq \emptyset$. Note we allow $I = \emptyset$ in this case. 
	\end{enumerate}
    \label{simplicial}
\end{lem}



\begin{defn}
	Given a stratum $C$ of a model $\Xc$, which corresponds to the stratum $\tau_C$ in $\Sigma_{(\Xc,\Hc)}$, let $\Star(\tau_C)$ denote the \textit{open star} of $\tau_C$ inside $\Sigma_{(\Xc,\Hc)}$. 
\end{defn}
Example \ref{divmon} admits another formulation in the context of snc log pairs.

\begin{prop}\cite[Proposition 2.4.6]{mustatanicaise}
	Let $(\Yc,\Dc)$ be a snc log model for $U = Y\backslash D$ where $\Yc_K \cong Y$ and $\Dc_K\cong D$.  Take a point $x = (x_1,...,x_n) \in \Sigma_{(\Yc,\Dc)}\subseteq \R^{I^\vb\cup I^\hb}$ and let $I = \{ i\in I^\vb\cup I^\hb  : x_i \neq 0\}$. Let $\xi_x$ be a generic point of the intersection $\bigcap_{i \in I} {\Dc}_i$. Then there exists a unique real valuation \begin{align}
		v_x : \Oc_{{\Yc},\xi_y} \backslash\{0\}\longrightarrow \R
	\end{align} such that $v_x(f_i) = x_i$ for all $i \in I^\vb \cup I^\hb$, where $f_i$ is a local equation for ${\Dc}_i$ in $\Oc_{{\Yc},\xi_x}$. The valuation $v_x$ extends the valuation on $K$. 
	\label{valsnclogpair}
\end{prop}

The valuation constructed in Proposition \ref{valsnclogpair} defines a birational point of $Y$ and thus also defines a birational point of $U$.  

\begin{prop}
	Let $(\Yc,\Dc)$ be an snc log model of $U$. Then there is a canonical embedding $\iota_{(\Yc,\Dc)}: \Sigma_{(\Yc,\Dc)}\hookrightarrow U^\an$ with a continuous retraction map $\rho_{(\Yc,\Dc)}: U^\an \rightarrow \Sigma_{(\Yc,\Dc)}$. 
	\label{retract}
\end{prop}
Note that Proposition \ref{retract} also holds in the case $D=0$ and recovers the original theory of retractions. Again, in this case we will write $\iota_\Yc$ and $\rho_\Yc$ respectively. 
\begin{proof}
The embedding is given by Proposition \ref{valsnclogpair}. The map $\rho_{(\Yc,\Dc)}: U^\an \rightarrow \Sigma_{(\Yc,\Dc)}$ is defined as follows. We first formally complete the $R$-scheme $\Yc$ along its special fibre to get a formal scheme $\widehat{\Yc}$. Then by \S\ref{formal}, and since $\Yc$ is proper, we have the reduction map \begin{align*}
		\text{red}_{\Yc} : Y^\an \longrightarrow \Yc_k.
	\end{align*}
	We can restrict $\red_\Yc$ to $U^\an \subset Y^\an$ to get a map $\red: U^\an \longrightarrow \Yc_k$.  Given a point $x \in U^\an$, 
	let $E_1,...,E_n$ be the irreducible components of $\Yc_k$ and $E_{n+1},...,E_m$ the irreducible components of $\Dc$ passing through the point $\text{red}(x)$. Let $\xi_x$ be the generic point of the connected component in the  intersection $\cap_i E_i$ containing $\text{red}(x)$. Then $\rho_{(\Yc,\Dc)}(x)$ is the point in $C_{(\Yc,\Dc)}$ given by $(v_x(f_i))\in \R^{I^\vb \cup I^\hb}$ where $f_i$ is a local equation for $E_i$  in  $\Oc_{\Yc,\text{red}(x)}$. Since $v_x(t) = 1 = \sum_{i\in I^\vb} N_iv_x(f_i)$, the point $\rho_{\Yc,\Dc}(x)$ lies in $\Sigma_{(\Yc,\Dc)}$. By construction, the retraction map is continuous and the right inverse to the embedding $\iota_{(\Yc,\Dc)}$.
\end{proof}

\begin{eg}
	Let $\Xc$ be a snc model of $X$ and $Y$ a stratum of $\Xc_k$. Then the (analytic) generic fibre in the sense of Berkovich, which we denote $]Y[$, of the formal completion $\widehat{\Xc}_{/Y}$  along $Y$ is described as 
	\begin{align*}
		]Y[ = \{x \in X^\an \ : \ \red_\Xc(x) \in Y\}.
	\end{align*}
	Let $\tau_Y \subset \Sigma_{\Xc}$ denote the stratum in the skeleton $\Sigma_\Xc$ corresponding to $Y$. Then 
	\begin{align*}
		\rho_\Xc^{-1}(\Star(\tau_Y)) = ]Y[\subseteq X^\an
	\end{align*}
	and there is a retraction map \begin{align*}
		\rho_Y :]Y[ \rightarrow \Star(\tau_Y)
	\end{align*}
	constructed as in the proof of Proposition \ref{retract}. The retraction $\rho_Y$ coincides with the map $\rho_\Xc$; in particular, the retraction over $\Star(\tau_Y)$ only depends on the formal scheme $\widehat{\Xc}_{/Y}$. 
	\label{strataformal}
\end{eg}


\begin{rmk}
The choice of snc log model is not unique in the following sense. Let $C_{(\Xc,\Hc)}'$ be a finite rational polyhedral subdivision of $C_{(\Xc,\Hc)}$ into simplicial cones whose integer points can be generated by a subset of a basis of $\Z^{I^\vb \cup I^\hb}$. Then by \cite[Chapter II + Chapter IV]{toroidal}, there is an snc pair $(\Xc',\Hc')$ and a proper morphism $\Xc' \rightarrow \Xc $ which induces an isomorphism $(\Xc'\backslash \Hc')_\eta \rightarrow (\Xc\backslash\Hc)_\eta $ such that  $$C_{(\Xc',\Hc')} \cong C_{(\Xc,\Hc)}',S_{(\Xc',\Hc')} \cong S_{(\Xc,\Hc)}' $$ where $S'_{(\Xc,\Hc)}$ is the subdivision of $S_{(\Xc,\Hc)}$ induced by the subdivision $C'_{(\Xc,\Hc)}$. 

Conversely, suppose we blow up an snc pair $(\Xc,\Hc)$ at a connected component of an intersection of irreducible components of $\Xc_k$. Denote the new pair by $(\Xc',\Hc')$ with $\Hc'$ the strict transform of $\Hc$. Then $C_{(\Xc',\Hc')}$ and $S_{(\Xc',\Hc')}$ are finite rational polyhedral subvisions of $C_{(\Xc,\Hc)}$ and $S_{(\Xc,\Hc)}$ respectively. 
\label{subdivision}
\end{rmk}

\subsection{Essential skeleton of a pair}
\label{essskel}

In the introduction, we briefly discussed the essential skeleton of a $K$-variety with respect to a global non-vanishing volume form $\omega$. Here, we describe a relative notion following \cite[\S5.1]{brownmazzon} in a more general setting. 
\begin{defn}
	An \textit{snc pair over $K$} is a pair $(X,\Delta_X)$ where $X$ is a proper and regular $K$-scheme and $\Delta_X = \sum_i \Delta_{X,i}$ an effective $\Q$-Cartier divisor such that $\Delta_{X,i}$ is a prime divisor and the support of $\Delta_X$ is snc. 
\end{defn}

Let $\omega \in K_X^{\otimes m}$ be a regular $m$-pluricanonical form on $X$ with  poles of order at most $m$ along $\Delta_{X,i}$ for each $i$. Thus $\omega$ is a section of $ \Oc_X(m(K_X+\Delta_X))$ and these forms are called $\Delta_X$-logarithmic $m$-pluricanonical forms.

If $\Xc$ is a reduced snc model of $X$, the form $\omega$ is a rational section of $\omega_{\Xc/R}^{\otimes m}\otimes \Oc_\Xc(m\overline{\Delta}_X) $ where $\overline{\Delta}_X = \sum_i \overline{\Delta}_{X,i}$. In particular, it  defines a divisor $\divi_\Xc(\omega)$ on $\Xc$. 

Alternatively, viewing $\omega$ as a rational section of $\omega^{\otimes m}_{X/K}$, 
we can then  associate a weight function $\wt_\omega : X^\an \rightarrow \R $ as in \cite{mustatanicaise, brownmazzon}. The weight function on divisorial points behaves as described in the introduction. Given an snc pair $(\Xc,\Hc)$ and  $x$ a divisorial point of $\Sigma_{(\Xc,\Hc)}$ (under the embedding $\iota_{(\Xc,\Hc)}$) with corresponding valuation $v_x$, we have  $\wt_\omega(x)=v_x(f) + m$ where $f \in K(X)^\times$ and $\text{div}_\Xc(\omega) = \text{div}(f)$ locally around $\text{red}_\Xc(x)$.  
\begin{defn}\cite[\S7]{mustatanicaise}
	The \textit{Kontsevich-Soibelman skeleton} $\Sk(X,\Delta_X,\omega)$ is the closure in Bir$(X)$ of the set $\text{Div}(X)$ of divisorial points where the weight function reaches its minimal value $\wt_\omega(X,\Delta_X) = \inf\{\wt_\omega(x) : x \in \text{Div}(X)\} \in \R\cup \{-\infty\}$. 
\end{defn}

\begin{lem} Let $\omega$ be a $\Delta_X$-logarithmic pluricanonical form. Then  for any snc pair $(\Xc,\Hc)$ of $ X$, we have the inclusion $\Sk(X,\Delta_X,\omega) \subseteq \Sigma_{(\Xc,\Hc)}$.
\end{lem} 

\begin{proof}
	This is a simpler case of Proposition 4.1.6 in \cite{brownmazzon}. 
\end{proof}


\begin{defn}(Essential skeleton)
	The \textit{essential skeleton} $\Sk(X,\Delta_X)$ is the union of all Kontsevich-Soibelman skeleta $\Sk(X,\Delta_X,\omega)$ where $\omega$ ranges over all regular $\Delta_X$-logarithmic pluricanonical forms. 
\end{defn}


\begin{prop}
	Let $(\Xc,\Dc)$ be a reduced snc pair and $(X,\Delta_X)$ an snc pair over $K$ such that $\Xc \times_R K = X$ and $\Dc = \overline{\Delta}_X $ where $\overline{\Delta}_X$ is the scheme-theoretic closure of $\Delta_X$ in $\Xc$. Suppose that $K_X + \Delta_X$ and $K_\Xc + \overline{\Delta}_X$  are semiample.  Then, under the embedding $\iota_{(\Xc,\Dc)}$ from Proposition \ref{retract}, $\Sigma_{(\Xc,\Dc)} = \Sk(X,\Delta_X) \subseteq X^\an $. 
	\label{essential}
\end{prop}

\begin{proof}
	This is again a simpler case of Proposition 5.1.7 in \cite{brownmazzon}.
\end{proof}

\begin{rmk}
	If $(\Xc,\Dc)$ is a reduced snc pair as in Proposition \ref{essential} and $K_X + \Delta_X  \sim 0$, then $\Lc := K_\Xc + \overline{\Delta}_X $ being semiample is equivalent to $K_\Xc + \overline{\Delta}_X \sim 0$.  Indeed, since $\Lc(X) = K^\times$, we have that $\Lc$ must be a multiple of $\Xc_k$ and hence principal. 
\end{rmk}

\subsection{Integral affine structures}
\label{integralaffine}
We recall the basic notions concerning integral affine structures. 

\begin{defn}
	An \textit{integral affine function} on an open subset of $\R^n$ is a continuous real-valued function locally of the form $f(x_1,\dots,x_n) = a_1x_1 +\dots + a_n x_n + b$ with $a_i \in \Z$ and $b \in \R$. We denote the sheaf of integral affine functions on $\R^n$ by $\Aff_{\R^n}$. 
\end{defn}

\begin{defn}
	Let $B$ be a topological manifold. An $n$-dimensional \textit{integral affine manifold} is a ringed space $(B,\Aff_B)$ which is locally isomorphic to $(\R^n, \Aff_{\R^n})$. \end{defn}

We say $B$ is an \textit{integral affine manifold with singularities} if there exists a subset $\Gamma \subseteq B$ which is locally a finite union of closed subspaces of codimension greater than or equal to 2 and $B\backslash \Gamma$ is an integral affine manifold.

\begin{eg}
	For an snc pair $(\Xc,\Hc)$, the embeddings of $C_{(\Xc,\Hc)}$ and $\Sigma_{(\Xc,\Hc)}$ into $\R^{I^\vb \cup I^\hb}$ naturally induce piecewise integral affine structures on $C_{(\Xc,\Hc)}$ and $\Sigma_{(\Xc,\Hc)}$ respectively, via the lattice embedding $\Z^{I^\vb\cup I^\hb} \hookrightarrow \R^{I^\vb \cup I^\hb}$. 
\end{eg}


	If $B$ is an integral affine manifold, then is it also smooth since the transition functions between local models of the ringed space are smooth. Thus given an integral affine manifold $B$, there is a subsheaf $\Lambda \subset \Lambda_\R = TB$ of the tangent sheaf consisting of integral tangent vectors. Similarly, there is a sheaf of integral cotangent vectors $\check{\Lambda} := \mathcal{H}\text{om}(\Lambda, \Z)$ on $B$. By definition, an integral affine structure $\Aff_B$ fits into the short exact sequence 
	\begin{align*}
	0 \rightarrow \R \rightarrow \Aff_B \rightarrow \check{\Lambda} \rightarrow 0 .	
	\end{align*}

\begin{rmk}
	Given an integral affine structure on $B$, there is an associated $\GL_n(\Z)\ltimes \R^n$-torsor on $B$ whose fibre over a point $x$ consists of all possible integral affine coordinate systems at $x$. 
\end{rmk}


\subsection{Affinoid torus fibrations}
\label{affinoid}
Let $T$ be a split multiplicative $n$-dimensional torus over $K$ with character lattice $M$ and cocharacter lattice $N = M^\vee$. 

\begin{defn}
	Define the \textit{tropicalisation map} of $T$ as 
	\begin{align*}
		\trop : T^\an &\longrightarrow N_\R \\
		|\cdot| &\longmapsto (m \mapsto -\log|m|).
	\end{align*}
\end{defn}
\noindent This map is continuous and its fibres are (not necessarily strict) $K$-affinoid tori. 

\begin{eg}
Let $\Xc$ be a regular toric model of $T$ \ie a regular toric scheme over $R$ such that $\Xc_K =T$. By \cite[\S7]{gubler}, such a model is described by a regular fan $\Sigma \subset N_\R \times \R_{\geq 0}$ whose cones intersect $N_\R\times \{0\}$ only at the origin of $N_\R$. Then by the proof of Proposition \ref{retract}, we have a retraction map 
	\begin{align*}
		\rho_\Xc : ]\Xc_k[ \longrightarrow \Sigma_{\Xc} 
	\end{align*}
	and a canonical embedding $\Sigma_{\Xc} \hookrightarrow ]\Xc_k[$. Let $\Sigma_1 = \Sigma \cap (N_\R \times \{1\}) $. Then there is a natural identification of $\Sigma_1$ with $\Sigma_{\Xc}$ given by sending a vertex of $\Sigma_1$ to the primitive generator of the corresponding ray of $\Sigma_\Xc$ and extending linearly \cite[\S7.9]{gubler}. Furthermore, we have $]\Xc_k[ = \trop^{-1}(|\Sigma_1|)$ and $\rho_\Xc$ is the restriction of $\trop$ over $]\Xc_k[$ \cite[Proposition 1.5.2]{mazzon2023toric}.
	\label{toricfibration}
\end{eg}

When $X_\Sigma$ is an $n$-dimensional $T$-toric variety given by a fan $\Delta $, Kajiwara \cite{kaj} and Payne \cite{payne} independently constructed a tropicalisation map 
\begin{align*}
	\trop_\Sigma : X^\an \longrightarrow N_\R(\Delta)
\end{align*}
where $N_\R(\Delta)$ is a partial compactification of $N_\R \cong \R^n$ which is uniquely determined by $\Delta$. The space $N_\R(\Delta)$ has the structure of an extended cone complex. Restricting the map $\trop_\Delta$ to the multiplicative torus, we recover the usual tropicalisation map, but the base $N_\R$ is endowed with the fan structure arising from $\Delta$. The base also has an integral affine structure arising from the natural lattice embedding $\Z^n \subseteq \R^n$. 
\begin{defn}
	Let $X$ be an  algebraic variety over $K$ and $B$ a topological space. We say that a continuous map $\rho: X^\an \longrightarrow B$ is an \textit{affinoid torus fibration} at $b\in B$ if there exists an open neighbourhood $U$ of $b$ such that the restriction of $\rho^{-1}(U)$ fits into the commutative diagram 
	\begin{figure}[H]
	\begin{center}
		\begin{tikzcd}
		\rho^{-1}(U) \ar[r,"\sim"]\ar[d,"\rho"] &\trop^{-1}(V) \ar[d,"\trop"]\\
		U \ar[r,"\sim"] & V
	\end{tikzcd}
	\end{center}
	\end{figure}
	\noindent where the top horizontal arrow is an isomorphism of Berkovich analytic spaces, the bottom horizontal arrow is a homeomorphism and $V$ is an open subset of $N_\R \cong \R^n$.
	\label{affinoidtorus}
\end{defn}

\begin{rmk}
	Affinoid torus fibrations are the analogue of smooth real torus fibrations over $\C$. As suggested in \cite[Remark 3.3]{payne}, the tropicalisation map is a non-archimedean analogue of the moment map. 
	
\end{rmk}

Given a continuous map $\rho: X^\an \rightarrow B$, let $B^\sm$ denote the locus of points where $\rho$ is an affinoid torus fibration. Then there is an induced integral affine structure on $B^\sm$ given by the pullback of the integral affine structure on $\R^n$ via the charts given in Definition \ref{affinoidtorus}. There will be a useful alternate characterisation of the integral affine structure on $B^\sm$, due to Kontsevich-Soibelman \cite{Kontsevich2006}.  Let $U$ be a connected open subset of $B^\sm$ and $h$ an invertible function on $\rho^{-1}(U)$. Then by the maximum modulus principle, the modulus $|h|$ is constant on the fibres and descends to a continuous function on $B^\sm$. After sheafifying, we then have the sheaf 
\begin{align*}
	\Aff_{B^\sm,\Z}(U) = \{-\log|h| : h\in \Oc_{X^\an}^\times (\rho^{-1}(U))\}.
\end{align*}
In \S \ref{Kaff}, we will review this construction in more detail.

\begin{defn}
	Given an snc log model $(\Xc,\Hc)$ of $X$, 
	 we say that $\rho_{(\Xc,\Hc)}$ is a \textit{non-archimedean SYZ fibration} associated to $(\Xc,\Hc)$ if the discriminant locus $\Gamma := \Sigma_{(\Xc,\Hc)}\backslash \Sigma_{(\Xc,\Hc)}^\sm$ has codimension greater than or equal to 2 in $\Sigma_{(\Xc,\Hc)}$. 
	\label{nasyz}
\end{defn}


\section{Log Calabi-Yau surfaces}

\label{logcy}

\subsection{The geometry of log Calabi-Yau surfaces}
\label{lcy}

We recall the following definition given in the introduction:
\begin{defn}
	A \textit{Looijenga pair} over a field $F$ of characteristic 0 is a pair $(Y,D)$ where $Y$ is a smooth  projective surface over $F$ and $D \in |-K_Y|$ is a reduced nodal curve with at least one singular point.

\end{defn}

We will write $D= D_1+ ... + D_n$ where $D_1,...,D_n$ are the irreducible components of $D$. By our definition, $D$ is either an irreducible nodal curve or a cycle of rational curves, meaning that $D_i$ only intersects $D_{i\pm1}$ transversally where $i$ is taken modulo $n$. Let $\ell(D)$ denote the number of irreducible components of $D$. 

Let $U = Y\backslash D$. There is a global non-vanishing 2-form $\omega \in H^0(Y,K_Y)$ on $Y$ which has simple poles along the divisor $D$; this form is unique up to scaling. The $K$-variety $U$ is a \textit{log Calabi-Yau surface} with maximal boundary. For brevity, we will just refer to $U=Y\setminus D$ as a log Calabi-Yau surface.

\begin{eg}
	Let $Y$ be a smooth toric surface over $F$ and $D$ the toric boundary of $Y$. Then $(Y,D)$ is a Looijenga pair, which we will refer to as a \textit{toric Looijenga pair}. We can blow up any smooth point of $D$ to produce a new surface $\tilde{Y}$ and take the strict transform $\tilde{D}$ of $D$. The new pair $(\tilde{Y},\tilde{D})$ is also a Looijenga pair. 
    \label{simpleeg}
\end{eg}

\begin{rmk}
The algebraic closure of $K$ is given by the field of Puiseux series $\overline{K} = \C\{\{t\} \} = \bigcup_{n \geq 1 }\C((t^\frac{1}{n}))$. Throughout this paper, we will blow up points of $D$ which may be closed \ie $K'$-rational points where $K'$ is a finite extension of $K$. For this reason, we will make a finite field extension of $K$ of the form $K_n = \C((t^{\frac{1}{n}}))$ such that the points we are blowing up become rational. Since $K_n$ is isomorphic to $K$ as discrete valuation fields, given a Looijenga pair $(Y,D)$ we will fix such a finite extension of $K$ and abuse notation to denote these finite extensions also by $K$ and its ring of integers by $R$. By the birational geometry of surfaces over an algebraically closed field of characteristic 0,  we can further assume that  $Y$ is rational. 
\label{algfield}

 \end{rmk}
\begin{defn}
	A \textit{Looijenga pair over $R$} is a pair $(\Yc,\Dc)$ where $\Yc$ is a smooth projective rational surface over $R$ and $\Dc \in |-K_{\Yc/R}|$ is nodal. We say that $(\Yc,\Dc)$ has \textit{good reduction} if both the generic and special fibre are Looijenga pairs over $K$ and $k $ respectively. 
	\label{defLooijenga}
\end{defn}

\begin{conv}
	A Looijenga pair over $K$ will be denoted by $(Y,D) $ and a Looijenga pair over $R$ will be denoted by $(\Yc,\Dc)$. 
\end{conv}

In fact,  Example \ref{simpleeg}
entirely captures  the geometry of Looijenga pairs. To see this, we make the following definition:

\begin{defn}
	For a Looijenga pair $(Y,D)$ over $K$, a blowup at a node of $D$ is a \textit{toric blowup of $Y$}. A \textit{non-toric blowup} is a blowup of a smooth point of $D$. 
\end{defn}

\begin{defn}
	Let $(Y,D)$ be a Looijenga pair over $K$. An irreducible curve $E$ on $Y$ is an \textit{interior exceptional curve} if $E \cong \proj^1$, $E^2 = -1$ and $E \neq D_i$ for any $i$. 
\end{defn}

\begin{defn}
	Let $(Y,D)$ be a Looijenga pair over $K$.  A \textit{toric model} for $(Y,D)$ consists of a toric Looijenga pair $(\overline{Y},\overline{D})$ and a birational morphism $\pi : Y\rightarrow \bar{Y}$ such that $D\rightarrow\bar{D}$ is an isomorphism.
	\end{defn}
	
	In the above definition of toric model, we account for the case where smooth points that are blown up on $D$ are infinitely near \ie having blown up a point $p \in \overline{D}_i$, we then blow up the intersection of the new exceptional curve $E$ with the strict transform of $\overline{D}_i$ and so on. This leads to the following definition of generalised exceptional curve: 
	
	\begin{defn}
		A \textit{generalised exceptional curve} on $Y$ is a divisor of the form $E$ or $C_1 + ... + C_{k} + E$, where $k\geq 1$, $C_i \cong \proj^1$ is a smooth curve of self-intersection $-2$ disjoint from $D$ for $i \leq k-1$, $E$ is an interior exceptional curve with the relations $C_i\cdot C_j = 1$ if $j= i\pm 1 $ and $0$ otherwise, $E\cdot C_{k} = 1$  and $E\cdot C_i = 0 $ for $i \neq k$. 
	\end{defn}

	





\begin{prop}\cite[Proposition 1.3]{ghk}
	Let $(Y,D)$ be a Looijenga pair over $K$. After a finite extension  of $K$, there exists a Looijenga pair $(\tilde{Y},\tilde{D})$, a toric blowup of $({Y},{D})$, such that the pair $(\tilde{Y},\tilde{D})$ admits a toric model $(\bar{Y},\bar{D})$ \ie we have the following diagram
	\label{toricmodel}
	\begin{figure}[H]
	\centering
\begin{tikzcd}
&{(\tilde{Y},\tilde{D})} \arrow[ld,swap, "\text{non-toric blowups}"]\arrow[rd, "\text{toric blowups}"] &                         \\
         {(\overline{Y},\overline{D})} &    & {(Y,D)}                                                                                                           
\end{tikzcd}
	\end{figure}
\end{prop}

In light of Remark \ref{algfield}, the above proposition also holds over a finite extension of $K = \C((t))$. Indeed, a minimal smooth projective rational surface over an algebraically closed field of characteristic 0 is either $\proj^2$ or a ruled surface. When we drop the assumption that the field is algebraically closed, but assume it is still perfect, every minimal rational surface is isomorphic to  either a del Pezzo surface with Picard rank 1 or a conic bundle with Picard rank 2 by \cite[Theorem 3.3.1]{manin}.  Given a Looijenga pair $(Y,D)$, we can make a large enough field extension $K_n = \C((t^{\frac{1}{n}}))$ such that the minimal model of $Y$ over $K_n$ is either $\proj^2$ or a ruled surface. Then the proof in \cite{ghk} proceeds as usual. 

		The constructions that appear in this paper are independent of the choice of compactification of $U$ and thus will not be affected by toric blowups. For this reason, we make the following assumption. 
        
        \begin{assumption}
            We assume the necessary toric blowups have been performed such that $(Y,D)$ has a toric model. 
esf        \end{assumption}


\begin{defn}
	Let $(Y,D)$ be a Looijenga pair over $K$.   \begin{enumerate}
		\item A curve $C \subset Y$ is \textit{interior} if no irreducible component of $C$ is contained in $D$;
		\item An \textit{internal} $(-2)$-\textit{curve} is a smooth rational curve with self-intersection $-2$ which is disjoint from $D$; 
		\item A pair $(Y,D)$ is \textit{generic} if it has no internal $(-2)$-curves. 
	\end{enumerate}
\end{defn}

By Proposition 4.1 in \cite{gross_hacking_keel_2015}, any Looijenga pair is deformation equivalent to a generic pair. By adjunction, any smooth rational curve with negative self-intersection is either a $(-1)$-curve meeting $D$ transversally at a single point or a $(-2)$-curve. 

\begin{eg}
Let $(Y,D)$ be a Looijenga pair given by a finite number of non-toric blow ups of $(\overline{Y},\overline{D})$. Then $(Y,D)$ is generic if and only if the centres of the non-toric blowups on $(\overline{Y},\overline{D})$ are distinct. 
Furthermore, given a generalised exceptional curve $C_1 + ... + C_{k-1} + E$, the curves $C_i$ are internal $(-2)$-curves.

\end{eg}

\begin{defn}
	Let $(Y,D)$ be a generic Looijenga pair with toric model $(\overline{Y},\overline{D})$.  
	\begin{enumerate}
		\item An \textit{exceptional configuration} for $(Y,D)$ is a collection $\{E_{ij}\} \subset \Pic(Y)$ of classes of exceptional divisors arising from non-toric blowups of the toric model. Here, $E_{ij}$ are the exceptional divisors of the toric model meeting the divisor $D_i$. We shall refer to them as \textit{non-toric exceptional curves}.  
		\item Let $(Y,D)$ and $(Y',D')$ be two generic Looijenga pairs with the same toric model $(\overline{Y},\overline{D})$ with respective exceptional configurations $\{E_{ij}\}$ and $\{E'_{ij}\}$. We say that $(Y,D)$ and $(Y',D')$ have the same \textit{combinatorial type} if for each toric divisor $\overline{D}_i\subset \overline{D}$, the number of exceptional divisors of the toric model for $Y$ and $Y'$ with centre on $\overline{D}_i$ is the same. 
	\end{enumerate}
\end{defn}

\begin{eg}
	Let $({Y},{D})$ be a toric Looijenga pair over $k$ and consider the base change of the pair to $R$ which we denote $(\Yc,\Dc)$. Let $\mu_i \in \Dc_i\times_k K $ be a smooth $K$-rational point  and denote by  $(\Yc',\Dc')$ the Looijenga pair over $R$ given by the blowup at the closure of $\mu_i $ in $\Yc$.  There are two cases to consider: 
\begin{enumerate}[label = (\roman*)]
	\item the specialisation of $\mu_i$ is a nodal point of the divisor $D$:  the special fibre will be another toric Looijenga pair given by a toric blowup of $(Y,D)$;
	\item the specialisation of $\mu_i$ is a smooth point of $D_i \subset D$: the special fibre will have the same combinatorial type as the generic fibre. This case will be the most important for our purposes. 
\end{enumerate}

\label{valuations}
\end{eg}

\begin{rmk}
	Unfortunately, since the special fibre in (i) is a toric Looijenga pair, the methods of this thesis will not be able to deal with this case. This explains the assumption we make in (\ref{assume}). 
\end{rmk}
\begin{defn}
	Let $(Y,D)$ be a Looijenga pair with $\ell(D)\geq 3$. An \textit{orientation} of $D$ is an orientation of the dual graph of $D$.  
\end{defn}

An orientation of $D$ determines and is determined by a labelling of the components up to cyclic permutation. In particular, the labelling above of the components of $D= \sum_i D_i$ determines an orientation. 

\begin{rmk}
	There is another interpretation of the orientation that we will use later. Let $\nu: \tilde{D} \longrightarrow D $ be the normalisation of $D$. Then given the singular points of $D$, which we denote by $D^{\text{sing}}$, we can consider the possible labellings of the set $\nu^{-1}(D^\text{sing})$ by $0$ and $\infty$ subject to the following conditions: 
	\begin{enumerate}[label=(\roman*)]
		\item if $q_1,q_2\in \nu^{-1}(D^\text{sing})$ and $\nu(q_1) = \nu(q_2)$ then exactly one of the $q_i$ is labelled by $0$ and the other $q_j$ is labelled by $\infty$;
		\item if $\tilde{D}_i = \nu^{-1}(D_i)$ for some $i$ and $q_1\neq q_2 \in \tilde{D}_i$ with $\nu(q_1),\nu(q_2) \in D^\text{sing}$, then exactly one of the $q_i$ is labelled by $0$ and the other $q_j$ is labelled by $\infty$.
	\end{enumerate}
	Thus fixing a singular point and a labelling of the normalisation of this point determines an orientation of $D$. 
\end{rmk}

Let $D_i^\circ$ denote the  smooth locus of $D_i \subset D$. 

\begin{lem}\cite[Lemma 1.6]{friedman2016geometry}
\label{gm}
	A choice of orientation of $D$ fixes a canonical isomorphism $\psi: \Pic^0(D)\cong \G_m$. Via this isomorphism, if $p,q \in D_i^\circ$ and we choose an affine coordinate on $D_i$ such that $p$ corresponds to $1 $ and $q$ corresponds to $\lambda$, then \begin{align*}
		\psi(\Oc_D(p-q)) = \lambda^{-1}.
	\end{align*}
\end{lem}

\begin{lem}\cite[Lemma 1.9]{friedman2016geometry}
	Let $\mathbf{d} = \sum_{i,j}(p_{ij}-q_{ij})$ be a divisor of multidegree $(0,\dots ,0)$ with $p_{ij},q_{ij} \in D_i^\circ$. Let $z_i$ be an affine coordinate on $D_i$ for each $i$. Then $\mathbf{d}$ is a principal divisor if and only if 
	\begin{align*}
		\prod_{i,j} \frac{z_i(p_{ij})}{z_{i}(q_{ij})} = 1.
	\end{align*}
	\label{principal}
\end{lem}

\subsection{Algebraic periods}
\label{analper}

In the introduction, we introduced the algebraic period map for a Looijenga pair $(Y,D)$.  In this section, we will recall the definition and briefly explain how the map in fact computes period integrals. 

\begin{defn}
	Given a Looijenga pair $(Y,D)$ over a field $F$ of characteristic 0,  we define 
	\begin{align*}
		D^\perp := \{ \alpha \in \Pic(Y) \ | \ \alpha \cdot [D_i] = 0 \text{ for all }i\}.
	\end{align*}
\end{defn}

\begin{defn}
\label{perioddefn}
	Given an orientation of $D$, define the \textit{algebraic period map} as  \begin{align*}
		\mathcal{P}_{(Y,D)} : D^\perp &\longrightarrow \Pic^0(D)\cong_\psi \G_m\\
	L &\longmapsto \psi(L|_D)
	\end{align*}
	where $\psi$ is the canonical isomorphism from Lemma \ref{gm}. 
\end{defn}





We briefly justify why $\mathcal{P}_{(Y,D)}$ is the same as the classical notion of period integral. Let $(Y,D)$ be a Looijenga pair over $\C$ with $\Omega$ a volume form on $Y \backslash D$ with simple poles along each irreducible component of $D$.  We choose $\Omega$ such that $$\int_\Gamma \Omega = \frac{1}{(2\pi i)^2},$$where $\Gamma$ is the homology class of a real 2-dimensional torus in $U$.  Given $\gamma \in D^\perp$, there exists a class $\tilde{\gamma} = \sum_k\pm [C_k]$ representing $\gamma$, where $C_k$ is a smooth curve in $Y$ meeting $D$ transversally at the smooth locus of $D$. Then a cycle $\gamma'$ can be constructed, which is homologous to $\tilde{\gamma}$ and is used to define the period map 
\begin{align*}
	\varphi_{(Y,D)} : D^\perp & \longrightarrow \G_m \\ 
	\gamma &\longmapsto \exp\left(\frac{1}{2\pi i} \int_{\gamma'} \Omega\right).
\end{align*}
Then Proposition 3.12 in \cite{friedman2016geometry} shows that these integrals are equal to the periods as defined in Definition \ref{perioddefn}. By \cite{carlson1985one}, the mixed Hodge structure of $U$ is classified by such integrals.
\subsection{Construction of the model and skeleton}
\label{modelgeneral}
In \S\ref{prelim}, we defined an snc log model of a variety and constructed the skeleton associated to it. In this section, given a generic Looijenga pair $(Y,D)$ over $K$ which arises from the generic fibre of a Looijenga pair $(\Yc,\Dc)$ over $R$ with good reduction (Definition \ref{defLooijenga}), we will describe several birational modifications. We will apply these to $(\Yc,\Dc)$ to explicitly construct an snc log model of $U = Y\backslash D$ that will be used in the definition of the non-archimedean period map. The constructions we describe are of a very special type in the sense that the homeomorphism type of the associated skeleton is fixed. We will also give a detailed description of the simplicial structure of the associated skeleton in this section.


\subsection{Birational modifications}
In light of Example \ref{valuations}, we make the following assumption. 

\begin{assumption}
	Given a Looijenga pair $(Y,D)$ over $K$, we will assume that the pair can be realised as a generic fibre of a Looijenga pair $(\Yc,\Dc)$ over $R$ with good reduction where $\Dc$ is the scheme-theoretic closure of $D$ inside $\Yc$.  Furthermore, we will assume the special fibre $(\Yc_k,\Dc_k)$ is a generic Looijenga pair over $k$ with the same combinatorial type as the generic fibre $(\Yc_K, \Dc_K)$. 
	\label{assume}
	
\end{assumption}

In concrete terms, the assumption guarantees that the centres of the non-toric blowups on the generic fibre are distinct and specialise to distinct smooth points in the boundary $\Dc_k$ of the special fibre $\Yc_k$.

\begin{setup}
	Let $(Y,D )$ be a generic Looijenga pair over $K$, with toric model $(\overline{Y},\overline{D} )$ satisfying Assumption \ref{assume}. By the assumption, there exists a Looijenga pair $(\Yc,\Dc)$ over $R$ with good reduction such that $(\Yc,\Dc)\times_R K = (Y,D)$. We further suppose there are $k_{i}$ non-toric blowups of distinct $K$-rational points with centre  on $\overline{D}_i$, which gives rise to the pair $(Y,D)$. Let $\Dc_i$ denote the scheme-theoretic closure of $D_i$ in $\Yc$. 
	\label{setup}
\end{setup}


We now discuss the birational modifications we will be using in this thesis to construct snc log models of $U = Y\backslash D$.   The first construction is the following: 
\begin{construct}
If the centre of a non-toric blowup is on $\overline{D}_i$, we  blow up the subvariety $\Dc_i\times_R k $ inside $\Yc$. Denote the blowup by $\Yc'$. More generally, given a Looijenga pair $(\tilde{\Yc},\tilde{\Dc})$ over $R$, let $\tilde{\Dc}_i$ denote the  scheme-theoretic closure of $D_i$ in $\tilde{\Yc}$. We then blow up $\tilde{\Dc}_i\times_R k$ in $\tilde{\Yc}$. 
\label{blow}
\end{construct}

We can alternatively describe this modification as a deformation to the normal cone of the boundary divisor $D_i$.  

The next result describes the exceptional divisor in $\Yc'$ produced by Construction \ref{blow}. We will deal with the more general case of  Construction \ref{blow} in Lemma \ref{exceptional2}.



\begin{lem}
Suppose we are in Setup \ref{setup}. 
	Let the self-intersection number of the divisor $D_i$ in $Y$ be $n_i$. Let $\Yc'$ denote the blowup of the subvariety $\Dc_i\times_R k$ in $\Yc$. Then the new irreducible component $Y_1^i$ produced in the special fibre is isomorphic to $\F_{n_i} = \proj(\Oc\oplus\Oc(n_i))$. 
	\label{operation1}
\end{lem}

\begin{proof}
	Let $C \subset \Yc_k$ denote the 1-dimensional subvariety given by $\Dc_i\times_R k $. The curve $C$ is a complete intersection of the special fibre $\Yc_k$ and the horizontal divisor $\Dc_i$. Thus the conormal bundle $\mathcal{N}_{C/\Yc}^\vee$ is equal to $\Oc_C(-\Yc_k) \oplus \Oc_C(-\Dc_i)$. Since $\Yc_k$ is principal and $\Dc_i \cdot C = n_i$, this establishes the result. 

\end{proof}

We next review the notion of a type I modification from \cite{kulikov} which we shall refer to just as a \textit{flop}. Further details can be found in \cite[\S4]{kulikov}. We illustrate the construction  in Figure \ref{typeI}. Let $\tilde{\Yc}$ be an snc model of $Y$.   Suppose we have two irreducible components $X, X'$  in the special fibre of $\tilde{\Yc}$ and let $C \subset X\cap X'$ be a double curve. Let $E$ be a curve in $X$ with $E\cong \proj^1$ and  $E^2 =-1$; further suppose that $E$ does not contain any triple points (meaning a point that is the intersection of three irreducible components), intersects $C$ at a single point $P$ \ie $(E\cdot C)_X = 1$ and does not intersect any other double curves. We blow up $E$ in $\tilde{\Yc}$ to produce an exceptional divisor $V$ isomorphic to  $\proj^1 \times \proj^1$  whose rulings are given by $E$ and $E'$ and which has multiplicity 1 in the special fibre.    We will prove this in Lemma \ref{exceptional2} in the case of interest to us. The normal bundle of $V$ in $\tilde{\Yc}'$ is $\Oc_V(V)$ and its restriction to the fibre $E$ has degree $-1$ since $\Oc_V(V) = \Oc_V(-E-E')$.  Hence we can blow down $V$  to $E'$ along $E$. Thus we have \textit{flopped} the curve $E$ on $X$ to the curve $E'$ on $X'$. 
We will give explicit equations for the contraction of the ruling in \S\ref{technical}

\begin{figure}[H]
	
\centering

\tikzset{every picture/.style={line width=0.75pt}} 

\begin{tikzpicture}[x=0.75pt,y=0.75pt,yscale=-1,xscale=1]

\draw [color={rgb, 255:red, 189; green, 16; blue, 224 }  ,draw opacity=1 ]   (128,212.5) -- (100.79,232.52) ;
\draw    (100,190) -- (22.79,245.52) ;
\draw    (156,235) -- (78.79,290.52) ;
\draw    (100,190) -- (156,235) ;
\draw    (99.79,116.52) -- (100,190) ;
\draw    (154.79,137.52) -- (156,235) ;
\draw [color={rgb, 255:red, 189; green, 16; blue, 224 }  ,draw opacity=1 ]   (298,97.5) -- (270.79,117.52) ;
\draw    (270,75) -- (192.79,130.52) ;
\draw    (326,120) -- (248.79,175.52) ;
\draw    (270,75) -- (326,120) ;
\draw    (269.79,1.52) -- (270,75) ;
\draw    (324.79,22.52) -- (326,120) ;
\draw [color={rgb, 255:red, 74; green, 144; blue, 226 }  ,draw opacity=1 ]   (269.79,82.52) -- (270.79,117.52) ;
\draw [color={rgb, 255:red, 189; green, 16; blue, 224 }  ,draw opacity=1 ]   (297,62.5) -- (269.79,82.52) ;
\draw [color={rgb, 255:red, 74; green, 144; blue, 226 }  ,draw opacity=1 ]   (297,62.5) -- (298,97.5) ;
\draw    (419,188) -- (341.79,243.52) ;
\draw    (475,233) -- (397.79,288.52) ;
\draw    (419,188) -- (475,233) ;
\draw    (418.79,114.52) -- (419,188) ;
\draw    (473.79,135.52) -- (475,233) ;
\draw [color={rgb, 255:red, 74; green, 144; blue, 226 }  ,draw opacity=1 ]   (446,175.5) -- (447,210.5) ;
\draw    (218.43,149.29) -- (170.96,179.67) ;
\draw [shift={(168.43,181.29)}, rotate = 327.38] [fill={rgb, 255:red, 0; green, 0; blue, 0 }  ][line width=0.08]  [draw opacity=0] (8.93,-4.29) -- (0,0) -- (8.93,4.29) -- cycle    ;
\draw    (369.85,178.75) -- (320.79,149.52) ;
\draw [shift={(372.43,180.29)}, rotate = 210.78] [fill={rgb, 255:red, 0; green, 0; blue, 0 }  ][line width=0.08]  [draw opacity=0] (8.93,-4.29) -- (0,0) -- (8.93,4.29) -- cycle    ;
\draw  [dash pattern={on 4.5pt off 4.5pt}]  (188.43,234.29) -- (330.43,234.29) ;
\draw [shift={(333.43,234.29)}, rotate = 180] [fill={rgb, 255:red, 0; green, 0; blue, 0 }  ][line width=0.08]  [draw opacity=0] (8.93,-4.29) -- (0,0) -- (8.93,4.29) -- cycle    ;
\draw [shift={(185.43,234.29)}, rotate = 0] [fill={rgb, 255:red, 0; green, 0; blue, 0 }  ][line width=0.08]  [draw opacity=0] (8.93,-4.29) -- (0,0) -- (8.93,4.29) -- cycle    ;

\draw (116.39,225.91) node [anchor=north west][inner sep=0.75pt]    {$E$};
\draw (286.39,110.91) node [anchor=north west][inner sep=0.75pt]    {$E$};
\draw (448,178.9) node [anchor=north west][inner sep=0.75pt]    {$E'$};
\draw (74,240.4) node [anchor=north west][inner sep=0.75pt]    {$X$};
\draw (121,144.4) node [anchor=north west][inner sep=0.75pt]    {$X'$};
\draw (240,125.4) node [anchor=north west][inner sep=0.75pt]    {$X$};
\draw (293,26.4) node [anchor=north west][inner sep=0.75pt]    {$X'$};
\draw (397,237.4) node [anchor=north west][inner sep=0.75pt]    {$X$};
\draw (439,139.4) node [anchor=north west][inner sep=0.75pt]    {$X'$};
\draw (299,65.9) node [anchor=north west][inner sep=0.75pt]    {$E'$};
\draw (97,198.4) node [anchor=north west][inner sep=0.75pt]    {$C$};
\draw (304.79,88.92) node [anchor=north west][inner sep=0.75pt]    {$C$};
\draw (416,198.4) node [anchor=north west][inner sep=0.75pt]    {$C$};
\draw (128,201.4) node [anchor=north west][inner sep=0.75pt]  [font=\scriptsize]  {$P$};
\draw (440,213.4) node [anchor=north west][inner sep=0.75pt]  [font=\scriptsize]  {$P$};
\draw (271.79,85.92) node [anchor=north west][inner sep=0.75pt]    {$V$};

\end{tikzpicture}

\label{typeI}
\caption{Type I modification from \cite{kulikov}}
\end{figure}
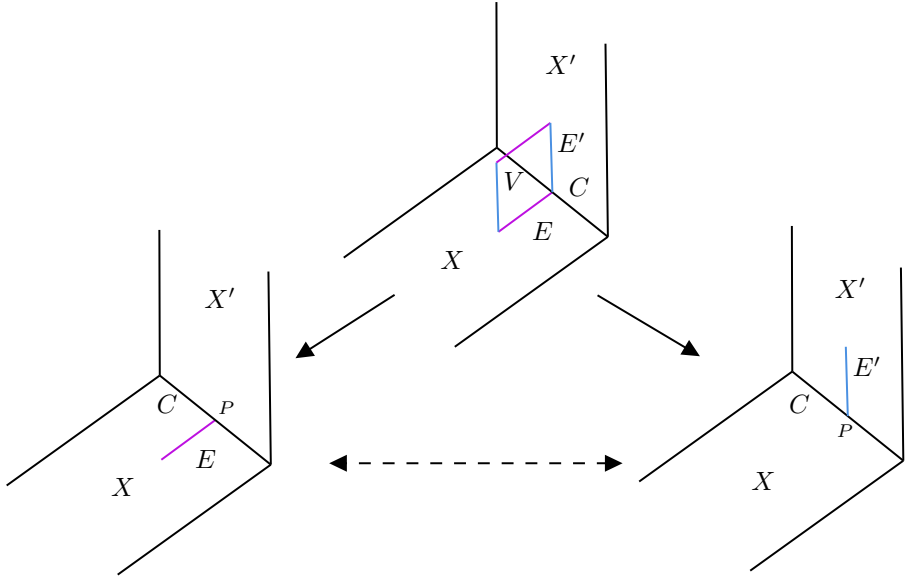

\begin{setup}
	Let $(Y,D)$ be a toric Looijenga pair where $Y$ is a ruled surface with $D_2$ and $D_4$ fibres of $Y$ under the ruling. Consider the non-toric blowup of $Y$ at a smooth point ${\mu}$ on either $D_1$ or $D_3$ and denote the exceptional curve by $E$. Let $C$ denote the strict transform of the fibre of $Y$ which passes through ${\mu}$. 
	\label{mirrorC}
	






\end{setup}

The next two constructions make use of the flop we described above. 
\begin{construct} 

\begin{enumerate}[label = (\roman*)]
	\item Suppose we have applied Construction \ref{blow} to the pair $(\Yc,\Dc)$.  Let $E$ be a non-toric exceptional curve in $Y$ whose centre is on $\overline{D}_i$. Then $E$ corresponds to an exceptional non-toric curve $E_k$ in $\Yc_k$ by assumption. We flop $E_k$ to the newly produced irreducible component $Y'$.  Denote the new Looijenga pair by $(\Yc',\Dc')$ where $\Dc'$ is the strict transform of $\Dc$. We let $\Dc_i'$ denote the strict transform  of $\Dc_i$ in $\Yc'$. 
	\item By Lemma \ref{operation1}, the new irreducible component is isomorphic to a Looijenga pair $(Y',D')$ over $k$ whose toric model is a ruled toric surface. Moreover, the centres of the non-toric blowups on $Y'$ are along an irreducible divisor in $D'$ and in particular we are in Setup \ref{mirrorC}. We first apply Construction \ref{blow} \ie blow up $\Dc_i'\times_R k $ in $\Yc'$. Then,   given a non-toric exceptional curve $E \subset Y'$, let $C$ be the strict transform of the fibre through the centre of $E$ as in Setup \ref{mirrorC}. We flop the curve $C$ to the new irreducible component $Y''$. Under this modification, $E$ becomes a fibre of $Y'$. 
\end{enumerate}

\label{flopblow}
\end{construct}

After applying Construction \ref{flopblow}, we produce a new pair $(\tilde{\Yc},\tilde{\Dc})$. We can apply Construction \ref{blow} to $\tilde{\Dc}_i$, the strict transform of $\Dc_i$, and proceed to flop the curves $C$ as in Construction \ref{flopblow} (ii) to this new irreducible component.   We thus consider the following setup: 


\begin{setup}

	Suppose $\pi : \tilde{\Yc} \rightarrow \Yc$ is a composition of blowups and flops described in Constructions \ref{blow} and \ref{flopblow} with $\tilde{\Yc}$  reduced.  Then $(\tilde{\Yc},\tilde{\Dc})$ is an snc log model for $U = Y\backslash D$ where $\tilde{\Dc}$ is the strict transform of $\Dc$ under $\pi$. Let $\tilde{\Dc}_i$ denote the strict transform of $\Dc_i$.
	\label{setupmodify}
\end{setup}

\begin{rmk}
	The modifications in Constructions \ref{blow} and  \ref{flopblow} are of a very special type in the sense that the birational morphism $f: (\tilde{\Yc},\tilde{\Dc}) \rightarrow (\Yc, \Dc)$ is log crepant \ie $f^*(K_{{\Yc}} + {\Dc}) = K_{\tilde{\Yc}} + \tilde{\Dc}$. Since $K_\Yc + 
\Dc \sim 0$, we will see in Proposition \ref{r2} that this fixes the homeomorphism type of the skeleton.
\end{rmk}

\begin{prop}
	Given a non-toric exceptional curve $E$ in $Y$, let $\mathcal{E}$ denote the scheme-theoretic closure of $E$ in $\Yc$. Then suppose we are in Setup \ref{setupmodify} and  $C$ is either
	\begin{enumerate}[label = (\roman*)]
		\item the intersection of the special fibre and  $\tilde{\Dc}_i$ for some $i$ with $k_i > 0$; 
		\item a non-toric exceptional curve in an irreducible component $X$ of $\tilde{\Yc}_k$ given by the intersection of the special fibre and the strict transform of $\mathcal{E}$. 
	\end{enumerate}
	Then the exceptional divisor of the blowup of $C$ inside $\tilde{\Yc}$ is isomorphic to
	\begin{enumerate}[label = (\roman*)]
		\item $\F_{|C\cdot \tilde{\Dc}_i|}$
		\item $\proj^1 \times \proj^1$
	\end{enumerate}
	respectively. 	\label{exceptional2}
\end{prop}

We note that (i) is a more general case of  Construction \ref{operation1}. 

\begin{proof}
	In case (i), the curve $C$ is a complete intersection of $\tilde{\Dc}_i$ and $\tilde{\Yc}_k$ and thus the conormal bundle is 
	$$\mathcal{N}_{C/\tilde{\Yc}}^\vee = \Oc_C(-\tilde{\Dc}_i) \oplus \Oc_C(-\tilde{\Yc}_k) = \Oc(-C\cdot \tilde{\Dc}_i)  \oplus \Oc(-C\cdot \tilde{\Yc}_k) = \Oc(-C\cdot \tilde{\Dc}_i) \oplus \Oc,$$ since $\tilde{\Yc}_k$ is a principal divisor.
	
	 In case (ii), the non-toric exceptional curve $C$ in $X$  is a complete intersection of $X$ and the horizontal divisor $\Ec$. Thus the conormal bundle $$\mathcal{N}_{C/\tilde{\Yc}}^\vee = \Oc(-(C\cdot X))\oplus \Oc(-(C\cdot \Ec))= \Oc(-1)\oplus \Oc(-1),$$ and the exceptional divisor of the blowup of $C$ inside $\tilde{\Yc}$ is isomorphic to $\proj^1 \times \proj^1$.  

\end{proof}

\subsection{The associated skeleton}
We now turn our attention to the skeleton associated to the snc pairs constructed above.   It is important to note there are a number of other modifications we could have performed.

If we only apply the modifications described above, then we can describe the skeleton completely. Lemma \ref{simplicial} gives a description of its simplicial structure. 

\begin{prop}
Suppose we are in Setup \ref{setupmodify}. Then the skeleton $\Sigma_{(\tilde{\Yc},\tilde{\Dc})}$ is homeomorphic to $\R^2$.
\label{r2} 	
\end{prop}
\begin{proof}
We have $K_{\tilde{\Yc}} + \tilde{\Dc} \sim 0$ since the flop we perform in Construction \ref{flopblow} has no effect on $K_{\tilde{\Yc}}$. If $\pi : \Yc' \rightarrow \Yc$ corresponds to blowing up $\Dc_i\times_Rk$ in $\Yc$ with exceptional divisor $X$,  then we have 
$K_{\Yc'} = \pi^*(K_\Yc) + X  = \Dc' -X + X = \Dc'$ where $\Dc'$ is the strict transform of $\Dc$. Repeating this argument for $\Dc_i'$ gives $K_{\tilde{\Yc}} + \tilde{\Dc} \sim 0$.

Since $K_Y + D \sim 0$ and $K_{\tilde{\Yc}} + \tilde{\Dc} \sim 0$,   by Proposition \ref{essential}, the skeleton $\Sigma_{(\tilde{\Yc},\tilde{\Dc})}$ is equal to the essential skeleton $\Sk(Y,D)$. Similarly $\Sigma_{(\Yc,\Dc)}$ is equal to $\Sk(Y,D)$ and thus it is enough to describe the homeomorphism type of $\Sigma_{(\Yc,\Dc)}$. 

We have that  $\Sigma_{(\Yc,\Dc)}$ is the cone of the dual graph of the boundary $D$ which can be `flattened'  to give a homeomorphism to $\R^2$. More precisely, for each node $p_{i,i+1} := D_i \cap D_{i+1}$ of $D$, we have a 2-dimensional cone $\sigma_{i,i+1}$ in $\Sigma_{(\Yc,\Dc)}$ generated by $v_i$ and $v_{i+1}$.  We then glue $\sigma_{i,i+1}$ to $\sigma_{i-1,i}$ along the ray $\rho_i := \R_{\geq 0}v_i$ to obtain the piecewise linear manifold $\Sigma_{(\Yc,\Dc)}$ which is homeomorphic to $\R^2$ and a decomposition $\{\sigma_{i,i+1}\}\cup \{\rho_i\} \cup \{0\}$ for $i$ running through the ordered set indexing the irreducible components of the boundary $D$.
	\end{proof}

\subsection{Construction of the model}

\begin{prop}
	Suppose we are in Setup \ref{setup}. Then there exists a reduced snc log model $(\tilde{\Yc},\tilde{\Dc})$ of $Y\backslash D$ such that one irreducible component of $\tilde{\Yc}_k$ is toric and all other irreducible components are a single non-toric blowup of a toric Looijenga pair $(X,\Delta_X)$.	\label{model}
\end{prop}

\begin{proof}

	We give an explicit construction of $\tilde{\Yc}$ here. Write  $b_i = D_i^2$ and suppose there are $k_i$ non-toric blowups along $D_i$ in $Y$. 
	
	
For each divisor $D_i$ with $k_i > 0$, we first apply Construction \ref{blow} \ie we blow up $\Dc_i\times_Rk$ in $\Yc$. We then apply Construction \ref{flopblow} (i) to flop each exceptional curve whose centre is a smooth point on $\overline{\Dc}_i\times_Rk$ to the newly produced irreducible component $Y_1^i$.  Apply Construction \ref{flopblow} (ii) to all but one of the curves which are given by the strict transform of a fibre through the centre of a non-toric blowup. Denote this snc log model of $U = Y\backslash D$ by $(\Yc',\Dc' )$ where $\Dc'$ is the strict transform of $\Dc$.   The new irreducible component produced by applying Construction \ref{blow} within  \ref{flopblow} (ii)  is denoted $Y_2^i$. 

Next we can run through Construction \ref{flopblow} (ii) again.  Indeed, let $\Dc'_i$ denote the strict transform of $\Dc_i$ and apply Construction \ref{blow} to $\Dc_i'\times_R k$. Denote the new component produced by $Y_{3}^i$. We then flop all but one of the curves which are given by the strict transform of a fibre through the centre of a non-toric blowup. 
More generally, we denote the new irreducible component produced by applying Construction \ref{blow} to the strict transform of $\Dc_i$ for the $j^\text{th}$ time by $Y_j^i$.  After applying Construction \ref{blow} for the $j^\text{th}$ time to the strict transform of $\Dc_i$, we apply Construction \ref{flopblow} (ii) and flop all but one of the curves which are given by the strict transform of a fibre through the centre of a non-toric blowup to the new irreducible component $Y_j^i$. 

This ensures we leave behind one non-toric exceptional curve on $Y_j^i$.  We can repeat this  until each irreducible component of the special fibre not equal to the strict transform of $Y$ is isomorphic to a single non-toric blowup of a toric Looijenga pair $(X,\Delta_X)$.  Recall $X$ is a ruled toric surface as described in Proposition \ref{exceptional2}. 


We can describe the irreducible components $Y^i_{j}$ by the classification of surfaces. Indeed, we have that in the special fibre, for each $i$, there will be a chain of irreducible components $$Y^i_{1} = \text{Bl}_{p_1}\F_{|b_i+k_i-1|},\ Y^i_{2} = \text{Bl}_{p_2}\F_{|b_i+k_i-2|},\ \ldots \ ,\  Y^i_{k_i} = \text{Bl}_{p_{k_i}}\F_{|b_i|}, $$
where $Y^i_{1}$ intersects $Y_0$ non-trivially, $Y^i_{j}$ only interests $Y^i_{j\pm1}$ and $p_i$ is a smooth point on the boundary of $\F_{|b_i+k_i-i|}$.  The resulting pair is an snc log model of $U = Y\backslash D$.

\end{proof}


\begin{eg}

 Assume we are in Setup \ref{setup} with $(Y,D)$ a non-toric del Pezzo surface of degree 5 whose combinatorial type is illustrated on the left hand side of the figure below. Then the special fibre of the model arising from Proposition \ref{model} is described on the right hand side of the figure below. 
	\begin{figure}[H]

\begin{center}
\tikzset{every picture/.style={line width=0.75pt}} 

\begin{tikzpicture}[x=0.75pt,y=0.75pt,yscale=-1,xscale=1]

\draw    (105,94) -- (205,194) ;
\draw    (206,160) -- (103,260) ;
\draw    (119,93) -- (118,261) ;
\draw [color={rgb, 255:red, 208; green, 2; blue, 27 }  ,draw opacity=1 ][fill={rgb, 255:red, 208; green, 2; blue, 27 }  ,fill opacity=1 ]   (184,128) -- (139,174) ;
\draw [color={rgb, 255:red, 208; green, 2; blue, 27 }  ,draw opacity=1 ][fill={rgb, 255:red, 208; green, 2; blue, 27 }  ,fill opacity=1 ]   (177.5,121) -- (132.5,167) ;
\draw [color={rgb, 255:red, 74; green, 144; blue, 226 }  ,draw opacity=1 ][fill={rgb, 255:red, 208; green, 2; blue, 27 }  ,fill opacity=1 ]   (190,236) -- (145.5,192) ;
\draw [color={rgb, 255:red, 184; green, 233; blue, 134 }  ,draw opacity=1 ][fill={rgb, 255:red, 208; green, 2; blue, 27 }  ,fill opacity=1 ]   (148,183) -- (86,183.67) ;
\draw    (286.1,136.29) -- (355.05,205.23) ;
\draw    (295.76,135.6) -- (295.07,251.42) ;
\draw [color={rgb, 255:red, 208; green, 2; blue, 27 }  ,draw opacity=1 ][fill={rgb, 255:red, 208; green, 2; blue, 27 }  ,fill opacity=1 ]   (390.9,110.09) -- (359.87,141.8) ;
\draw [color={rgb, 255:red, 208; green, 2; blue, 27 }  ,draw opacity=1 ][fill={rgb, 255:red, 208; green, 2; blue, 27 }  ,fill opacity=1 ]   (351.6,139.04) -- (320.58,170.76) ;
\draw [color={rgb, 255:red, 74; green, 144; blue, 226 }  ,draw opacity=1 ][fill={rgb, 255:red, 208; green, 2; blue, 27 }  ,fill opacity=1 ]   (358.49,240.39) -- (327.81,210.05) ;
\draw [color={rgb, 255:red, 184; green, 233; blue, 134 }  ,draw opacity=1 ][fill={rgb, 255:red, 208; green, 2; blue, 27 }  ,fill opacity=1 ]   (295.46,194.66) -- (253.7,194.66) ;
\draw   (330.03,112.05) -- (378.01,160.1) -- (343.93,194.12) -- (295.96,146.07) -- cycle ;
\draw   (364.11,78.03) -- (412.09,126.09) -- (378.01,160.1) -- (330.03,112.05) -- cycle ;
\draw    (357.8,180.41) -- (286.56,252.57) ;
\draw   (295.41,243.79) -- (343.78,194.35) -- (377,226.86) -- (328.63,276.3) -- cycle ;
\draw   (296.28,146.15) -- (295.41,243.79) -- (245.3,243.34) -- (246.17,145.71) -- cycle ;

\draw (208,199.4) node [anchor=north west][inner sep=0.75pt]    {$D_{1}$};
\draw (123,80.4) node [anchor=north west][inner sep=0.75pt]    {$D_{2}$};
\draw (83,260.4) node [anchor=north west][inner sep=0.75pt]    {$D_{3}$};
\draw (298.41,189.91) node [anchor=north west][inner sep=0.75pt]  [font=\tiny]  {$Y_{0} =\mathbb{P}^{2}$};
\draw (219.9,129.39) node [anchor=north west][inner sep=0.75pt]  [font=\footnotesize]  {$Y_{1}^{2} =\text{Bl}\mathbb{F}_{0}$};
\draw (393.68,146.58) node [anchor=north west][inner sep=0.75pt]  [font=\footnotesize]  {$Y_{2}^{1} =\text{Bl}\mathbb{F}_{1}$};
\draw (352.37,253.66) node [anchor=north west][inner sep=0.75pt]  [font=\footnotesize]  {$Y_{1}^{3} \ =\text{Bl}\mathbb{F}_{0}$};
\draw (359.8,183.81) node [anchor=north west][inner sep=0.75pt]  [font=\footnotesize]  {$Y_{1}^{1} =\text{Bl}\mathbb{F}_{0}$};
\end{tikzpicture}
\end{center}
	\caption{Model of non-toric dP5 using Proposition \ref{model}}

\end{figure}
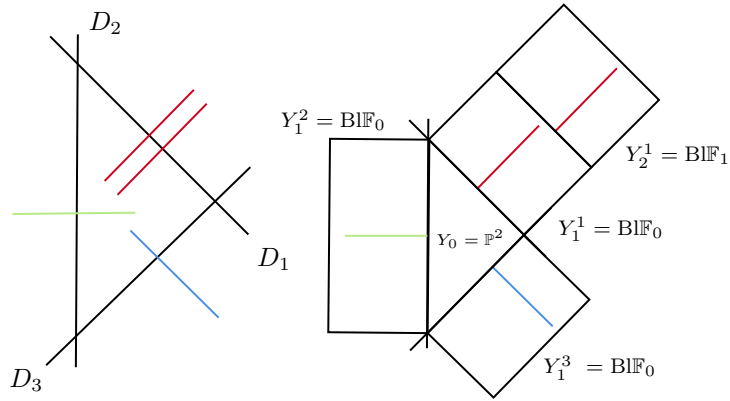
The skeleton of the new pair $(\tilde{
 \Yc},\tilde{\Dc})$ is illustrated in the figure below; the red crosses depict singularities in the integral affine structure which we discuss in \S\ref{Zaffine}.
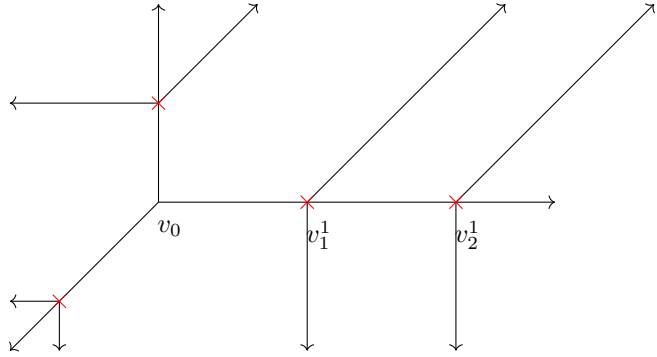
\begin{figure}[H]
    \centering
\begin{tikzpicture}[x=0.75pt,y=0.75pt,yscale=1,xscale=1]
\draw[->]    (0,0) -- (-75,-75) ;
\draw[->] (0,0) --(0,100);
\draw[->] (0,0) -- (200,0);
\draw[->]    (75,0) -- (75,-75) ;
\draw[->] (75,0) --(175,100);
\draw[->]    (150,0) -- (150,-75) ;
\draw[->] (150,0) --(250,100);
\draw[->] (0,50) --(50,100);
\draw[->] (0,50) --(-75,50);
\draw[->] (-50,-50) --(-50,-75);
\draw[->] (-50,-50) --(-75,-50);
\node[cross,color=red, fill=red, inner sep=0pt,minimum size=5pt,label=below right:{}] (b) at (75,0) {};
\node[label=below right:{$v^1_{1}$}] (b) at (65,0) {};
\node[label=below right:{$v^1_{2}$}] (b) at (140,0) {};
\node[label=below right:{$v_{0}$}] (b) at (-10,0) {};
\node[cross,color=red, fill=black, inner sep=0pt,minimum size=5pt,label=below right:{}] (b) at (150,0) {};
\node[cross,color=red, fill=black, inner sep=0pt,minimum size=5pt,label=below right:{}] (b) at (-50,-50) {};
\node[cross,color=red, fill=black, inner sep=0pt,minimum size=5pt,label=below right:{}] (b) at (0,50) {};
\end{tikzpicture}
\end{figure}
\end{eg}

\begin{rmk}
Let $(\tilde{\Yc},\tilde{\Dc})$ be the snc log model of $U = Y\backslash D$ constructed in Proposition \ref{model}. Then there is a 1-1 correspondence between non-toric irreducible components of $\tilde{\Yc}_k$ and non-toric exceptional curves in $(Y,D)$. 	Let $E_j^i$ denote the non-toric exceptional curves with centre on $\overline{D}_i$ for $j=1,..., k_i$ and $\mathcal{E}_j^i$ the scheme-theoretic closure of $E_{j}^i$ in $\Yc$. Furthermore, suppose the strict transform of $\mathcal{E}_j^i$, when restricted to the special fibre, cuts out the non-toric exceptional curve on $Y_j^i$.   Then the correspondence is given by $
	Y_j^i\leftrightarrow E_j^i. $

\end{rmk}

\subsection{Construction of the fibration}
\label{constructfib}

In this section, we construct the fibration for a log Calabi-Yau variety in any dimension which has a Zariski dense algebraic torus. 

\begin{defn}
    An $n$-dimensional Looijenga pair $(Y,D)$ over $K$ is a pair where $Y$ is a smooth projective variety over $K$ and $D \in |-K_Y|$ such that $D$ has a 0-stratum. We will also assume there is a birational map $Y \dashrightarrow Y_\mathbf{t}$ where $Y_\mathbf{t}$ is a toric variety; this is a toric model of $Y$. 
\end{defn}

\begin{rmk}
    It is no longer true that a log Calabi-Yau variety has a toric model in higher dimensions; see \cite{hackingkeel} for an example. 
\end{rmk}

Let $(Y,D)$ be an $n$-dimensional Looijenga pair over $K$ with $(\Yc,\Dc)$ a reduced snc log model of $U = Y \backslash D$. By Proposition \ref{retract}, there is a retraction map $U^\an \longrightarrow \Sigma_{(\Yc,\Dc)}$. In this section, we prove that this map is a non-archimedean SYZ fibration (\ref{nasyz}) \ie an affinoid torus fibration away from $\Gamma$, the codimension greater than or equal to 2 strata of $\Sigma_{(\Yc,\Dc)}$.

Let $C$ be a 1-dimensional stratum of $\Yc_k$; note $C \cong \proj^1$.  Let $Y_1,...,Y_r$ denote the irreducible component of $\Yc+\Dc$ containing $C$. We set 
\begin{align*}
b_i &= \deg\Oc_C(-Y_i) = -(C\cdot Y_i). 
\end{align*}

Let $\Fc = \left(\Yc_k\cup \Dc\right)\backslash \left(\cup_iY_i\right)$ and $C^\circ = C \backslash \Fc$ . The space $C\backslash C^\circ$ consists of precisely two distinct points, which we denote by $c_0$ and $c_\infty$. Let $Y_0$ and $Y_\infty$ denote the respective prime divisors containing $c_0$ and $c_\infty$. Let $N_0, N_\infty$ and $N_i$ denote the respective multiplicities in $\Yc_k + \Dc$.  Since $\Yc_k$ is principal, we have the following equality:
\begin{align}
	N_0 + N_\infty = \sum_{i=1}^r b_i N_i.
	\label{equ}
\end{align}
  The following proposition appears in \cite[Proposition 5.4]{nicaise_xu_yu_2019} for the case when you have a vertical divisor. The proof method is the same.  We remark after the proposition how the positivity assumptions can always be satisfied.

\begin{prop}
Assume that $b_i> 0$ for every $i$. Then there exists a strictly semistable toric $R$-scheme $\Xc$ and a stratum $C'$ of $\Xc_k$ such that $\widehat{\Yc}_{/C} \cong \widehat{\Xc}_{/C'}$. 
\label{toricalong}
\end{prop}
\begin{proof}
	We will construct the scheme $\Xc$ directly. Let 
	\begin{align*}
		u_0 &= (1,0,0,0,..., 0, N_0) \in \Z^{r+1},\\
		u_1 &= (0,1,0,0,..., 0, N_1) \in \Z^{r+1}, \\
        \vdots\\
		u_{r-1} &= (0,0,0,0,...,1,N_{r-1}) \in \Z^{r+1}, \\ 
		u_\infty &= (-1,b_1,...,b_{r-1},N_\infty) \in \Z^{r+1}.
	\end{align*}
	
	Let $R_i$ be the ray spanned by $u_i$ in $\R^r\times \R_{\geq 0}$ and consider the cones $\sigma_0$ and $\sigma_\infty$ spanned by $\{R_1,...,R_r, R_0\}$ and $\{R_1,...,R_r,R_\infty\}$ respectively. Let $\Sigma$ be the fan in $\R^r\times \R_{\geq 0}$ with maximal cones $\sigma_0$ and $\sigma_\infty$.
	
	 Let $\Sigma_h $ denote the slice of $\Sigma$ at height $h \geq 0$. By \cite[\S7.7]{gubler}, the fan $\Sigma$ determines a toric $R$-scheme $\Xc$ whose generic fibre is isomorphic to the toric variety whose fan is $\Sigma_0$. 
	 By \cite[Proposition 7.11]{gubler}, the special fibre $\Xc_k$ is reduced. Moreover, since $\sigma_0 $ and $\sigma_\infty$ are simple cones, we have that $\Xc$ is regular over $R$ by \cite[\S4.3]{toroidal}. Together, this shows that $\Xc$ is strictly semistable. 
	
	The special fibre $\Xc_k$ is a simple normal crossings divisor and by  \cite[Lemma 7.10]{gubler}, the rays of the fan $\Sigma$ correspond to the prime divisors in $\Xc_k$ with multiplicities given by the last coordinate of the primitive generators of the rays. 
Thus we have the relation  \begin{align*}
		\Xc_k = N_0F_0 + N_\infty F_\infty + \sum_{i=1}^r N_iF_r,
	\end{align*}
	where $F_i$ are the irreducible components of $\Xc_k$ with multiplicity $N_i$. Let $C':= \cap_{i=1}^rF_i$ and let $d_0,d_\infty$ be the respective intersections of $C'$ with $F_0$ and $F_\infty$. By \cite[\S5.1]{fultontoric}, we have $C'\cdot F_j = -b_j$ for $j \in \{1,...,r\}$. 
	Consider the following line bundles \begin{align*}
	\Lc_i &= \Oc_{\Yc}(-Y_i - b_iY_\infty), \\ 
	\Lc_0 &= \Oc_{\Yc}(Y_\infty - Y_0),\\ 
	\Lc_\infty &= \Lc_0^{-1}.
\end{align*}
By construction, the restriction of each of these line bundles to $C\cong \proj^1$ is degree zero. Thus we can choose non-zero global sections $s_i, s_0,s_\infty$ from the line bundles $ \Lc_i|_C, \Lc_0|_C$ and $\Lc_\infty|_C$ respectively.  By the positivity assumption on the $b_i$, the conormal bundle $\mathcal{N}_{C/\Yc}^\vee = \oplus_{i=1}^r \Oc(-Y_i)$ is a direct sum of ample line bundles on $C$. In particular, the first cohomology group $H^1(\mathcal{N}_{C/\Yc}^\vee)$ vanishes and we have surjective maps 
\begin{align*}
	H^0((\Yc/C)_{n+1}, \Lc_i) \twoheadrightarrow H^0((\Yc/C)_{n}, \Lc_i)
\end{align*}
where $(\Yc/C)_{n}$ is the $n^\text{th}$ order infinitesimal thickening of $C$ in $\Yc$. The section $s_i$ lifts to all finite order thickenings by induction and thus also lifts to a global section of $\Oc_{\widehat{\Yc}/C}(-Y_i-b_iY_\infty)$. We will continue to denote the element by $s_i$. Similarly, we can lift $s_0$ and $s_\infty$ to the formal completion. Note $s_\infty = \frac{1}{s_0}$ and both $s_0$ and $s_\infty$ are nowhere vanishing global sections of $\Lc_0$ and $\Lc_\infty$ on $\widehat{\Yc}_{/C}$ respectively. 

We now construct the isomorphism $ f : \widehat{\Yc}_{/C} \rightarrow \widehat{\Xc}_{/C'}$ by describing it on local toric charts and showing they agree on the intersections. To do this, consider the following open formal subschemes:
\begin{align*}
	\mathscr{Y}_0 = \widehat{\Yc}_{/C}\backslash \{c_\infty\}, \	 \mathscr{Y}_\infty = \widehat{\Yc}_{/C}\backslash \{c_0\}, \ \mathscr{X}_0 = \widehat{\Xc}_{/C'}\backslash \{d_\infty\}, \ \mathscr{X}_\infty = \widehat{\Xc}_{/C'}\backslash \{d_0\}.
\end{align*}
By construction, since $s_i$ is a global section of $\Oc_{\widehat{\Yc}/C}(-Y_i-b_iY_\infty)$, we have that $s_i$ is a global equation for $Y_i$ in $\mathscr{Y}_0$. Similarly, $s_0$ is a global equation for $X_0$ in $\mathscr{Y}_0$. On $\mathscr{Y}_\infty$, we have that $s_is_\infty^{-b_i}$ defines $Y_i$ respectively whilst $s_\infty$ defines $X_\infty$. The element $\omega = ts_0^{-N_0}\prod_{i=1}^rs_i^{-N_i} \in \Oc_{\widehat{\Yc}_{/C}}$ defines an invertible regular function on $\widehat{\Yc}_{/C}$ by (\ref{equ}). \\
Let $\{u_0^\vee, u_1^\vee,...,u_r^\vee\}$ be the dual basis to $\{u_0,u_1,...,u_r\}$. Then we have
\begin{align}
	\mathscr{X}_0 = \Spf\ R\{\chi^{u_0^\vee}\}[[\chi^{u_1^\vee},...,\chi^{u_r^\vee}]]/\left(t - \chi^{N_0u_0^\vee + N_1u_1^\vee  + 
    \ldots + N_ru_r^\vee }\right).
	\label{formal1}
\end{align}
Choose integers $\alpha_{0},\alpha_1,...,\alpha_r$ such that $\alpha N_0 + \alpha_1 N_1 +\ldots + \alpha_r N_r = 1$ and define the morphism $f_0 : \mathscr{Y}_0 \rightarrow \mathscr{X}_0$ of formal $R$-schemes by the following morphism of topological $R$-algebras: 
\begin{align*}
	\Oc(\mathscr{X}_0) &\longrightarrow \Oc(\mathscr{Y}_0)\\
	\chi^{u_i^\vee} &\longrightarrow \omega^{\alpha_i}s_i.
\end{align*}
By employing the same argument as in \cite[Proposition 5.4]{nicaise_xu_yu_2019}, $f_0$ is an isomorphism.

We next consider the pair of formal toric charts $\mathscr{Y}_\infty$ and $ \mathscr{X}_\infty$. The lattice vectors $\{-u_0^\vee, u_1^\vee + b_1u_0^\vee,\ldots, u_r^\vee +b_ru_0^\vee\}$ are a dual basis to $\{u_\infty, u_1,..., u_r\}$ and we have 
\begin{align}
		\mathscr{X}_\infty = \Spf\ R\{\chi^{-u_0^\vee}\}[[\chi^{u_1^\vee +b_1u_0^\vee},\ldots, \chi^{u_r^\vee +b_ru_0^\vee}]]/\left(t - \chi^{N_0u_0^\vee + N_1u_1^\vee  +\ldots+ N_ru_r^\vee }\right).
	\label{formal2}
\end{align}
Define the morphism $f_\infty : \Oc(\mathscr{X}_\infty) \longrightarrow \Oc(\mathscr{Y}_\infty)$ as 
\begin{align*}
	\chi^{-u_0^\vee} &\longmapsto \omega^{-\alpha_0}s_\infty\\ 
	\chi^{u_i^\vee +b_iu_0^\vee} &\longmapsto \omega^{\alpha_i + b_i \alpha_0}s_is_\infty^{-b_i}.\\
\end{align*}By the same reasoning, $f_\infty$ is an isomorphism. On the intersections $\mathscr{Y}_{0,\infty} := \mathscr{Y}_0 \cap \mathscr{Y}_\infty$ and $\mathscr{X}_{0,\infty} := \mathscr{X}_0 \cap \mathscr{X}_\infty$, the maps $f_0$ and $f_\infty$ agree and thus they both glue to give an isomorphism of formal schemes 
\begin{align*}
	f: \widehat{\Yc}_{/C} \longrightarrow \widehat{\Xc}_{/C'}.
\end{align*}
\end{proof}

\begin{rmk}

In the statement of Proposition \ref{toricalong}, we made the assumption that $b_i >0$ for all $i$ to ensure that the algebraic sections $s_j$ lifted to formal  sections. By performing toric blowups at 0-dimensional strata of $\Yc$, we can always satisfy this positivity assumption. This may destroy the property that $\Yc_k$  is reduced but the special fibre of the model will still remain snc in the total space. 
\label{makepositive}

\end{rmk}



\begin{rmk}
	In the proof of Proposition \ref{toricalong}, we showed that the elements $s_0,s_1,...,s_r$ and $s_\infty$ are coordinates, satisfying some relations, on the preimage of $\Star(\tau_C)$ under the retraction map $\rho_{(\Yc,\Dc)}$. Indeed, by Example \ref{strataformal}, by
 taking analytic generic fibres of the open formal schemes (\ref{formal1}) and (\ref{formal2}), we see that the pair $(s_0,s_1,...,s_{r-1})$ form a pair of affinoid coordinates on $\rho^{-1}(\Star(\tau_C)) = ]C[ \subseteq U^\an $. 
\end{rmk}

 We can now prove that the associated retraction map is a non-archimedean SYZ fibration. 

\begin{thm}
	Let $(Y,D)$ be an $n$-dimensional Looijenga pair  and $(\Yc,\Dc)$ a reduced snc log model of $U = Y \backslash D$. Let $\Gamma$ be the set of codimension greater than or equal to 2 strata in $\Sigma_{(\Yc,\Dc)}$. Then the retraction map $\rho_{(\Yc,\Dc)}: U^\an \rightarrow \Sigma_{(\Yc,\Dc)}$ is an $n$-dimensional affinoid torus fibration over $\Sigma_{(\Yc,\Dc)}\backslash \Gamma$. In particular, $\rho_{(\Yc,\Dc)}$ is a non-archimedean SYZ fibration. 
	\label{nasyzaffinoid}
\end{thm}

\begin{proof}
	We first show that $\rho_{(\Yc,\Dc)}$ is an affinoid torus fibration over the top dimensional open faces of $\Sigma_{(\Yc,\Dc)}$. Let $\mathring{\sigma}$ be such an open face - this corresponds to a 0-dimensional stratum $y$ of the special fibre $\Yc_k$ \ie $\mathring{\sigma} = \Star(\tau_y)$. By Example \ref{strataformal}, we know $\rho_{(\Yc,\Dc)}^{-1}(\Star(\tau_y))$ is isomorphic to the analytic generic fibre of the formal scheme $\Spf\ \widehat{\Oc}_{\Yc,y}$. Since $\Yc_k$ is reduced and $(\Yc,\Dc)$ an snc pair, we have \begin{align*}
		\widehat{\Oc}_{\Yc,y} \cong R\llbracket y_0,...,y_n\rrbracket/(y_0...y_n -t).
	\end{align*}
	By Example \ref{toricfibration},  we can identify the restriction of $\rho_{(\Yc,\Dc)}$ over $\Star(\tau_y)$ with the restriction of the tropicalisation map 
	\begin{align*}
		\rho_{\G_m^{n,\an}} : \G_m^{n,\an} \longrightarrow \R^n
	\end{align*}
	over the n-dimensional open simplex 
	\begin{align*}
		\{(v_0,...,v_n) \in \R^n_{>0}| \ v_0 + ...+v_n=1\}.
	\end{align*}
	It follows that over $\Star(\tau_y)$, $\rho_{(\Yc,\Dc)}$ is an $n$-dimensional affinoid torus fibration. 
	
	It remains to extend the fibration over the  1-dimensional strata in $\Sigma_{(\Yc,\Dc)}$.  These each correspond to a 1-dimensional stratum $C$ of the special fibre $\Yc_k$. Let $\mathring{\tau}_C$ denote an open 1-dimensional stratum in $\Sigma_{(\Yc,\Dc)}$.  Using the notation of Proposition \ref{toricalong}, we can perform toric blowups at 0-dimensional strata to ensure that $b_i> 0$ for every $i$ and we can further assume $N_0,N_\infty >0$. By \cite[Proposition 3.1.7]{mustatanicaise}, this does not affect $\rho_{(\Yc,\Dc)}$. The effect on the skeleton $\Sigma_{(\Yc,\Dc)}$ is a sequence of star subdivisions of the top dimensional faces corresponding to the zero-dimensional strata that are blown up \cite[3.1.9]{mustatanicaise}. This does not affect  $\mathring{\tau}_C$ and, since $\rho_{(\Yc,\Dc)}$ is an $n$-dimensional affinoid torus fibration over top dimensional faces, we can assume  the positivity assumption is satisfied.

	Consider the subset  $\Star(\mathring{\tau}_C)$ - this is equal to the union of the open faces corresponding to the strata $c_0$, $c_\infty$ and $C$ in $\Yc_k$.   Let $\Xc$ be the toric $R$-scheme associated to the fan $\Sigma$ constructed in the proof of  Proposition \ref{toricalong} with the closure of a  1-dimensional torus orbit $C'$ corresponding to $\sigma_0 \cap \sigma_\infty$ and $\widehat{\Yc}_{/C}\cong \widehat{\Xc}_{/C'}$. 



	Now   let $V = \Sigma_1$ be the height 1 slice of $\Sigma$. Then the restriction of $\rho_{(\Yc,\Dc)}$ over $\Star(\mathring{\tau}_C)$ can be identified with the restriction of the $n$-dimensional affinoid torus fibration $\rho_{\G_m^{2,\an}}$ over $V$. Thus the fibration extends to an affinoid torus fibration over all codimension 1 strata of $\Sigma_{(\Yc,\Dc)}$. 
	
\end{proof}


%
%

\begin{defn}
	Let $(Y,D)$ be a 2 dimensional generic Looijenga pair satisfying Assumption \ref{assume}. When $(\Yc,\Dc)$ is the snc log model of $U$ constructed in Proposition \ref{model}, we denote the skeleton $\Sigma_{(\Yc,\Dc)}$ by $\Sk(U)$. 
\end{defn}
\subsection{Integral affine structure}
\label{Zaffine}
In this section, we return to the case of generic 2 dimensional Looijenga pairs satisfying Assumption \ref{assume}.  We will denote the snc log model of $U= Y \backslash D$ from Proposition \ref{model} by $(\Yc,\Dc)$. We use Proposition \ref{toricalong} to give an explicit description of the local toric charts. Whilst we will only carry out this computation for the model constructed in Proposition \ref{model}, the computation is similar for the general models considered in \S \ref{modelgeneral}.

We start by recalling how the integral affine structure arises from the affinoid torus fibration. Following  the notation of \S\ref{constructfib}, let $C = Y_w \cap Y_v$ be a 1-dimensional stratum of $\Yc_k$ and denote by $\tau_C$  the corresponding 1-dimensional stratum in $\Sk(U)$. The open star of $\tau_C$ consists of two maximal faces corresponding to the 0-dimensional stratum $c_0 = C \cap Y_0$ and $c_\infty = C\cap Y_\infty$ joined along ${\tau}_C$. 
Let $\Xc$ denote the local toric model over $R$ associated to the fan $\Sigma$ with $C'$ the corresponding 1-dimensional stratum constructed in the proof of Proposition \ref{toricalong} such that $\widehat{\Yc}_{/C} \cong \widehat{\Xc}_{/C'}$. 
Then by Example \ref{toricfibration} there is a commutative diagram 
\begin{figure}[H]
\begin{center}
	\begin{tikzcd}
{]C[} \ar[r,"\sim"]\ar[d,"\rho_{(\Yc,\Dc)}"] & {]C'[} \ar[d, "\trop"]\\ 
\Star(\tau_C)\ar[r,"\sim", "\phi"'] & \Sigma_1 
\end{tikzcd}
\end{center}
\end{figure}
\noindent where the upper arrow is an isomorphism of $K$-analytic spaces and the lower one a homeomorphism. The $K$-analytic spaces $]C[$ and $]C'[$ are the analytic generic fibres of $\widehat{\Yc}_{/C}$ and $\widehat{\Xc}_{/C'}$ respectively. By definition, the integral affine structure on $\Star(\tau_C)$ is the pullback of the integral affine structure on $\Sigma_1$.

 Let $v_X$ be a vertex of $\Sigma_{(\Yc,\Dc)}$ corresponding to the irreducible component $X$ in the special fibre $\Yc_k$ with boundary $\Delta_X =  \sum_{i=1}^k \Delta_{X,i}$. Let $(\Xc,\Delta_\Xc) = (X,\Delta_X) \times_k R$ denote the constant family with fibre $(X,\Delta_X$). By the proof of Proposition \ref{nasyzaffinoid}, the retraction map $(X\backslash \Delta_X)^\an \rightarrow \Sigma_{(\Xc,\Delta_{\Xc})}$ is also an affinoid torus fibration away from the vertex $O \in \Sigma_{\Xc,\Delta_\Xc}$. Thus the skeleton $\Sigma_{(\Xc,\Delta_{\Xc})}$ is an integral affine manifold away from $O$.

\begin{lem}
There is a canonical identification \begin{align*}
 		\varphi_v: \Star(v_X)  
 	 \xrightarrow{\ \ \sim \ \ } \Sigma_{(\Xc,\Delta_\Xc)}
 	\end{align*}
	of singular integral affine manifolds. 
 	\label{Kk}
 \end{lem}
 
 \begin{proof}
 	By construction, the polyhedral decompositions of $\Star(v)$ and $\Sigma_{(\Xc,\Delta_\Xc)}$ are the same. By the proof of Proposition \ref{toricalong} and the discussion above, the integral affine structures around $v$ and $O$ are both given by the intersection numbers of each of the 1-dimensional strata $C$ in the boundary of $X$ with the prime divisor not containing $C$ - by construction, these numbers match up and are given by the self intersection numbers in $X$. Thus $\Star(v_X)$ and $\Sigma_{(\Xc,\Delta_\Xc)}$ are canonically identified. 
 \end{proof}

By Lemma  \ref{Kk}, we have that the induced singular integral affine structure around $v_X$ induced by the non-archimedean SYZ fibration coincides with the singular integral affine structure on $\Sigma_{(\Xc,\Delta_\Xc)}$.  By \cite[Lemma 1.3]{ghk}, the integral affine structure on $\Sigma_{(\Xc,\Delta_{\Xc})}\backslash \{O\}$  extends across the origin if and only if $X$ is toric and $\Delta_X$ is the toric boundary. Thus we have the following corollary:

 \begin{cor}
 	Let $Y_0$ denote the toric irreducible component of $\Yc_k$. Then the integral affine structure extends across $v_{Y_0}$. We will denote the vertex $v_{Y_0}$ by $O$. 

 	\label{extendtoric}
 \end{cor}


Let $X$ denote an irreducible component of the special fibre $\Yc_k = X + \sum_{i=1}^rX_i$ and $v_X$ the corresponding vertex in $\Sk(U)$. Let $\Delta_X = \sum_{i=1}^r X_i \cap X + \Dc \cap X$.  By the description of $(\Yc,\Dc)$ in Proposition \ref{model}, we know $X$ is either toric or isomorphic to a non-toric blowup of $\F_n$. In the toric case, the integral structure will extend across the vertex by Corollary \ref{extendtoric}. In the latter case, we know $\Delta_X = \sum_{j=1}^4 C_j$. 
Let $\tau_{C_i}$ denote the corresponding 1-dimensional stratum in $\Sk(U)$. Since $\Sk(U)\subset \R^{I^\vb \cup I^\hb}$, let $v_{C_{i}}\in \Z^{I^\vb \cup I^\hb}$ denote the primitive generator of the ray defining $\tau_{C_i}$.

\begin{prop}
	The monodromy $T^\text{int}_{\rho_{(\Yc,\Dc)}}$ of the integral affine structure induced by $\rho_{(\Yc,\Dc)}$ around $v_X$ is 
	\begin{align*}
		T^{\text{int}}_{\rho_{(\Yc,\Dc)}} = \begin{pmatrix}
			1 & 1 \\
			0 & 1
		\end{pmatrix}
	\end{align*}
	with respect to the basis $(\chi^{v_{C_1}^\vee},\chi^{v_{C_2}^\vee})$. 

	\label{Zaffinecomp} 
\end{prop}

\begin{proof}
Let $b_i = -(C_i \cdot H) = -(C_i^2)_X$ where $H \neq X$ is the prime component of $\Yc$ containing $C_i$. Further assume the centre of the non-toric blowup is on $C_1$. 
The integral affine structure on $\Star(\tau_{C_i})$ identifies the set $(v_{C_{i-1}}, v_{C_i}, v_X, v_{C_{i+1}})$ with 
\begin{align*}
	\left(v_0= (1,0), v_1 = (0,1), v_2 = (0,0), v_\infty = (-1, b_i)\right),
\end{align*}
while the integral affine structure on $\Star(\tau_{C_{i+1}})$ identifies the set $(v_{C_{i}}, v_{C_{i+1}}, v_X, v_{C_{i+2}})$  with 
	\begin{align*}
	\left(v_0'= (1,0), v_1' = (0,1), v_2' = (0,0), v_\infty' = (-1, b_{i+1})\right). 
\end{align*}
Thus the transition map for the tangent bundle $\Lambda$ from the chart $\Star(\tau_{C_i})$ to the chart $\Star(\tau_{C_{i+1}})$ on the intersection $\Star(\tau_{C_i}) \cap \Star(\tau_{C_{i+1}})$ is given by the matrix
\begin{align*}
	\begin{pmatrix}
		b_i & 1 \\ 
		-1 & 0
	\end{pmatrix}
\end{align*}
with respect to the basis $(v_{C_{i-1}},v_{C_i})$. 
The transition function for the cotangent bundle $\check{\Lambda}$ is given by the same matrix with respect to the basis $(\chi^{v_{C_{i-1}}^\vee},\chi^{v_{C_i}^\vee})$.
The composition of these matrices then gives the monodromy. Indeed, take a loop around $v_X$ connecting $v_{C_1}, v_{C_2}, v_{C_3}, v_{C_4}$ and $v_{C_1}$ in that order. Then the monodromy around this loop with respect to the basis $(\chi^{v_{C_1}^\vee},\chi^{v_{C_2}^\vee})$ is 
\begin{align*}
\begin{pmatrix}
		b_1 & 1 \\ 
		-1 & 0
	\end{pmatrix}
\begin{pmatrix}
		b_4 & 1 \\ 
		-1 & 0
	\end{pmatrix}
\begin{pmatrix}
		b_3 & 1 \\ 
		-1 & 0
	\end{pmatrix}
	\begin{pmatrix}
		b_2 & 1 \\ 
		-1 & 0
	\end{pmatrix}.
\end{align*}
By our assumption on $C_1$ and the geometry of the model $(\Yc,\Dc)$,  we have $b_2 = b_4 = 0$ and $b_1 + b_3 =- 1$. Thus the composition of matrices above is equal  to $\begin{pmatrix}
	1 & 1\\ 
	0 & 1
\end{pmatrix}$. This completes the proof.
\end{proof}

\begin{notation}
	Given any simply connected subset $\tau \subset \Sk(U)^\sm$, we write $\check{\Lambda}_\tau$ for the stalk of $\check{\Lambda}$ on $\Sk(U)^\sm$ at any point of $\tau$; any two such stalks are canonically identified by parallel transport. 	Given an edge/ray of $\rho$ of $\Sk(U)$, we have $\check{\Lambda}_\rho \cong \Z$.
\end{notation}

\begin{cor}
 The invariant cotangent direction is parallel to $\check{\Lambda}_{\tau_{C_1}}$. 
	\label{invariant}
\end{cor}

The singular vertices of $\Sk(U)$ which correspond to the non-toric exceptional curves which have a centre on $\overline{D}_i$ can be viewed as lying on a ray $\rho_i$ supporting the origin $v_0$ (corresponding to the toric irreducible component) and the vertices $v^i_{1},...,v^i_{k_i}$.

The same computation as in Proposition \ref{Zaffinecomp} for the monodromy of the integral affine structure around a loop containing multiple singular vertices follows readily. Indeed, suppose $v_X = v^i_{j}$ and consider a loop around $v_X=v^i_{j}, v^i_{j+1},...,v^i_{k}$. Then the monodromy   is $$\begin{pmatrix}
	1 & k+1 \\ 
	0 & 1
\end{pmatrix}$$ with respect to the coordinates $(\chi^{v_{C_1}^\vee}, \chi^{v_{C_2}^\vee})$ described in Proposition \ref{Zaffinecomp}. Thus the invariant cotangent directions for the monodromy around $v^i_{j}$ and $v^i_{j+1}$ agree on $\Star(v^i_{j})$ and $\Star(v^i_{j+1})$.  Let $S= \cup_{j=1}^{k_i} \Star(v^i_{j})$. Then with Proposition \ref{Zaffinecomp}, we have that $\check{\Lambda}_{\tau_{C_1}} \cong \Z\langle \chi^{v_{C_1}^\vee}\rangle $ and $\check{\Lambda}|_{S} \cong \Z$. By parallel transport,  $\chi^{v_{C_1}^\vee}$ can be extended trivially to all of $S$.  We can then define    $\check{\Lambda}_{\rho_i} = \check{\Lambda}_{\tau_{C_1}}$ for any singular vertex on $\rho_i$. 





\begin{notation}
	Let $(\Yc,\Dc)$ denote the snc log model of $U = Y\backslash D$ constructed in Proposition \ref{model}.  From now on, let $\Gamma$ be the set of singular vertices corresponding to non-toric irreducible components of $\Yc_k$ and let $\Sk(U)^\sm = \Sk(U) \backslash \Gamma$. In particular, $\Sk(U)^\sm$ now contains $v_0$ since the integral affine structure arising from $\rho_{(\Yc,\Dc)}$ extends across the toric vertex $v_0$. 
\end{notation}

\section{Tropical geometry}
\label{tropconst}

 Let $(Y,D)$ be a Looijenga pair. Recall the following definition from \S \ref{analper}:  \begin{align*}D^\perp := \{\alpha \in \Pic(Y) \ | \ \alpha \cdot [D_i] = 0 \ \text{for all } i\} \subseteq \Pic(Y).
	\end{align*}

\begin{defn}
	The charge $Q(Y,D)$ of a Looijenga pair $(Y,D)$ is \begin{align*}
		Q(Y,D) = 12-D^2 - r(D)
	\end{align*}
	where $r(D)$ is the number of irreducible components of $D$. 
\end{defn}

\begin{rmk}
	The charge is an invariant intrinsic to the log Calabi-Yau surface $U=Y\setminus D$ and does not depend on the compactification. When $(Y,D)$ is a Looijenga pair over $\C$,  we have the following equalities \cite{friedman2016geometry}: 
\begin{align*}
	Q(Y,D) = 2+b_2(Y) - r(D) = \chi_\text{top}(U),
\end{align*}
where $b_2$ is the second Betti number of $Y$ and $\chi_\text{top}$ denotes the topological Euler characteristic. 
\end{rmk}
\begin{lem}\cite[Lemma 2.7]{friedman2016geometry}
The charge satisfies $Q(Y,D) \geq 0 $ with equality if and only if $(Y,D)$ is toric. 	
\end{lem}
\begin{lem}\cite[Lemma 1.5]{friedman2016geometry}
The subspace $D^\perp$ is free of rank $Q(Y,D)-2+s$ where $s$ is the rank of the kernel of  
\[\bigoplus_i\Z[D_i] \longrightarrow \Pic(Y),\]
which maps $[D_i]$ to $[D_i]$.
\label{charge}
\end{lem}

\begin{proof}
Since $Y$ is a smooth rational surface, $\Pic(Y)$ is a lattice and thus the subspace $D^\perp\subseteq\Pic(Y)$ is free. The rank computation follows from the exact sequence 
	\begin{align*}
		0 \rightarrow \ker(f) \rightarrow \bigoplus \Q[D_i] \xrightarrow{f} H^2(Y,\Q) \rightarrow D^\perp \otimes_\Z \Q \rightarrow 0
	\end{align*}
\end{proof}

\begin{lem}
\label{rankDperp}
Let $(Y,D)$ be a Looijenga pair with toric model $(\bar{Y},\bar{D})$ such that the centre of the blowups happen on $k$ distinct components of $\bar{D}$. Then the rank, $s$, of the kernel of 
\[\oplus_i \mathbb{Z}[D_i]\rightarrow \Pic(Y)\]
is given by
\[ 
s=\begin{cases}
		 2,  &\text{ if } k\geq 3\\
	   0 \text{ or } 1 &\text{ if } k=1,2\\
         0 &\text{ if } k=0
	\end{cases}
\]
\end{lem}
\begin{proof}
The case $k=0$ is trivial. Let $\{\bar{D}_1,...,\bar{D}_n\}$ be the components of $\bar{D}$  in arbitrary order and $\{D_1,...,D_n\}$ be the components of $D$ such that $D_i$ is the strict transform of $\bar{D_i}$. Without loss of generality, suppose $\bar{D}_1,\bar{D}_2$ and $\bar{D}_3$ are blown up to obtain $(Y,D)$. Since $D_1,D_2$, and $D_3$ will all involve distinct exceptional curves, these are automatically linearly independent from the set $\{D_4,..,D_n\}$, which only involve the strict transforms of $\bar{D}_i$. If $\{D_4,..,D_n\}$ is a linearly dependent set, then there exists $\{c_i\}$ such that 
\[\sum_{i=4}^n c_i D_i =0\]
which implies that there exists a character $u$ such that $\langle u, v_i\rangle=c_i$ and $\langle u,v_j\rangle=0$ if $j\in \{1,2,3\}$. The only way this can be true is if $u=0$ which implies that $c_i=0$ for all $i$. Therefore, $s=0$.
\\\\
Next, suppose that only $\bar{D}_1$ is blown up to obtain $(Y,D)$. Then $D_1$ will automatically be linearly independent from $D_2,..,D_n$ and there wil be one relation between these coming from the one obtained from $\bar{D}_2,...,\bar{D}_n$. Therefore, $s=1$.\\\\
The last case is when $k=2$. Suppose $\bar{D}_1$ and $\bar{D}_2$ are blown up now. There are two cases two consider. One where the corresponding rays, $v_i$ and $v_j$, in the toric fan form $(\bar{Y},\bar{D})$ are negative of each other and also when they are not. If they are, let $u$ be an element of the dual lattice that vanishes on $v_i$, which will then also vanish on $v_j$. Then the corresponding character gives
\begin{align*} 
\chi^u&=\sum_{k=0}^n \langle u,v_k\rangle D_i\\
&=\sum_{k=2}^n \langle u,v_k\rangle D_i\\
&= 0 
\end{align*}
which shows that the remaining $n-2$ rays do not correspond to a linearly independent set of divisors. Thus, we must have that $s=1$.\\\\
The case that $v_1\ne -v_2$ and showing that $D_2,..,D_n$ are linearly independent is exactly like the case of $k=3$, therefore $s=0$.
\end{proof}
From now on, $(Y,D)$ will be a generic Looijenga pair satisfying Assumption \ref{assume} and  
 $\Sk(U) = \Sigma_{(\Yc,\Dc)}$ where $(\Yc,\Dc)$ is the model constructed in Proposition \ref{model} and $\overline{\Sk}(U) = \overline{\Sigma}_{(\Yc,\Dc)}$. We fix some notation for the rest of this chapter. 

\begin{notation}

 Recall the skeleton comes equipped with a polyhedral decomposition and denote the toric vertex by $v_0$ of $\Sk(U)$ as $O$. 
Let $\rho_i$ denote the ray in $\Sk(U)$ emanating from $O$ which supports the singular vertices $v_j^i$ of $\Sk(U)$ corresponding to non-toric blowups with  centre  on $\overline{D}_i$ as described in the previous section.   We will use the convention that $1 \in \Z \cong \check{\Lambda}_{\rho_i}$ corresponds to the integral cotangent vector which pairs positively with the tangent vector in the direction of $\rho_i$ \ie pointing away from $O$.  
	
	Along $\rho_i$ in $\Sk(U)$, there is a natural ordering of the vertices $v^i_{j}$ given by the distance from $O$. We will write $v^i_{j}< v^i_{k}$ to denote that the distance of $v^i_{k}$ from $O$ is greater than the distance of $v^i_{j}$ from $O$. Note that $v^i_{j} < v^i_{k}$ if and only if $j<k$. 
Denote the (collection of) edge(s) connecting $O$ and $v^i_{j}$ in $\rho_i$ by $\tau_j^{i}$ and let $\tau_{j,j+1}^{i}$ be the edge connecting $v_j^i$ and $v_{j+1}^{i}$ in $\rho_i$.

	Let $V(\overline{\Sk}(U))$ be the set of vertices of $\overline{\Sk}(U)$ and $E(\Sk(U))$ the set of finite edges in $\Sk(U)$. Given a vertex $v\in V(\overline{\Sk}(U))$, let $[O,v]$ denote the edge/ray connecting $O$ and $v$ and let $\check{\mathbf{e}}_v \in \check{\Lambda}_{[O,v]}$ denote the primitive cotangent vector corresponding to $1$ under the isomorphism above.

\end{notation}
\begin{defn}
	A tropical 1-cycle is a twisted singular 1-cycle on $\Sk(U) \backslash \Gamma$ with coefficients in the sheaf of integral cotangent vectors $\check{\Lambda}$ on $\Sk(U)\backslash \Gamma$.
\end{defn}


More concretely, we can describe a tropical 1-cycle as an oriented  graph $G$ with a map $f: G \longrightarrow \Sk(U)\backslash \Gamma$ and a section $\check{\xi}_e \in (f_e)^*\check{\Lambda}$ for each edge $e\subset G$.  Furthermore, at each vertex $v$ of $G$, we have the  balancing condition$$\sum_{e \ni v}\epsilon_e \check{\xi}_e  =0$$ where $\epsilon_e = 1$ if $e$ is an incoming edge and $\epsilon_e = -1$ if $e$ is an outgoing edge. 
We will depict a tropical cycle (locally) as in Figure \ref{tropcyclepic1}, with each section $\check{\xi}_e$ drawn by taking the corresponding dual tangent vector $\xi_e \in \Lambda$.

\begin{lem}[\cite{lai2023mirror}]
If a tropical cycle goes around a focus-focus singularity on $\rho$ \ie given by the matrix $\begin{pmatrix} 1 & 1 \\ 0 & 1\end{pmatrix}$ with respect to a suitable oriented basis where $\check{\Lambda}_\rho$ is the invariant direction, as shown in Figure \ref{tropcyclepic1}, then the edge emanating from the singularity carries a cotangent vector $\zeta$ parallel to $\check{\Lambda}_\rho$. Moreover, $\zeta$ is primitive if and only if the cycle carries a primitive generator of $\check{\Lambda}/\check{\Lambda}_\rho$ when it starts to surround the singularity.  
\label{prim}
\end{lem}

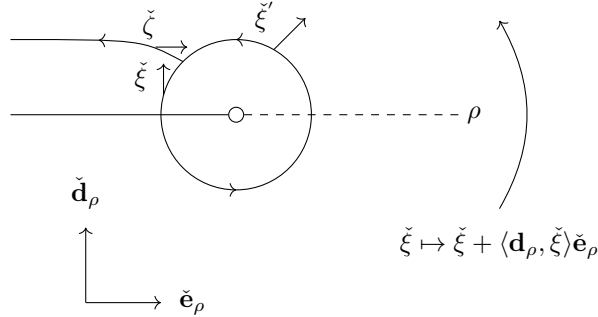
\begin{figure}[H]

\begin{center}
	\begin{tikzpicture}[circ/.style={shape=circle, inner sep=2pt, draw, node contents=}]

	\draw  node (x) at (0,0) [circ]; 
	\draw [ decoration={markings,  mark=at position 0.25 with {\arrow{>}}, mark=at position 0.75 with {\arrow{>}}},         postaction={decorate}](0,0) circle [radius=1cm]; 
 	\draw [dashed](x)-- (3,0); 
 	\draw (x) -- (-3,0); 
 	
	\node [label=right:$\check{\xi}^{'}$] at (90:1.3cm){}; 
 	\node [label=right:${\check{\xi}}$] at (-1.6,0.5){}; 
 	\node [label=right:$\check{\zeta}$] at (-1.5,1.2){}; 
 	
 	\draw [->] (140:1.4cm) -- ([xshift=12pt, yshift=0pt] 140:1.4cm); 
 	\draw [->] (60:1cm) -- ([xshift=12pt, yshift=12pt] 60:1cm); 
 	\draw [->] (165:1cm) -- ([xshift=0pt, yshift=12pt] 165:1cm); 
 	\draw  [decoration={markings, mark=at position 0.5 with {\arrow{>}}}, postaction={decorate}](135:1cm) .. coordinate [pos=.3] (a) controls (-1.25,1) .. ([xshift=-3cm, yshift=1cm] 0:0cm ); 
 	
  	\draw [->] (-2,-2.5) -- (-1,-2.5);   
 	\draw [->] (-2,-2.5) -- (-2,-1.5); 
	\node [label=above:$\check{\mathbf{d}}_{\rho}$] at (-2,-1.5){}; 
\node [label=right:$\check{\mathbf{e}}_{\rho}$] at (-1,-2.5){}; 
    \node [label=left:$\rho$] at (3.5,0){};

	\draw [<-] (3.5,1.25) [out=300, in=60] to (3.5,-1.25);

	\node [label=below:$\check{\xi} \mapsto \check{\xi}+  \langle {\mathbf{d}_{\rho}, \check{\xi}} \rangle  \check{\mathbf{e}}_{\rho}$ ] at (3.5,-1.25) {};
	
\end{tikzpicture} 
\end{center}
\caption{Tropical cycle surrounding a focus-focus singularity.}
\label{tropcyclepic1}
\end{figure}

\begin{proof}
	Let $\mathbf{e}_\rho$ be a generator of $\Lambda_\rho$ which points away from the origin $O$ and $\check{\mathbf{e}}_\rho\in \check{\Lambda}_\rho  $ the dual cotangent vector such that  $\langle \mathbf{e}_\rho, \check{\mathbf{e}}_\rho\rangle = 1$. Furthermore, let $\check{\Lambda}_\rho^\perp = \langle \mathbf{d}_\rho\rangle$ and $\check{\mathbf{d}}_\rho$ be the dual cotangent vector.
	
	As depicted in Figure \ref{tropcyclepic1}, let $\check{\xi}$ be the cotangent vector the tropical 1-cycle carries as it starts to go around the focus-focus singularity and $\check{\xi}'$ the vector after applying the monodromy transformation to $\check{\xi}$. By the balancing condition, we must have 
	\begin{align*}
		\check{\xi} + \check{\zeta} = \check{\xi}' = \check{\xi} + \langle \mathbf{d}_\rho, \check{\xi}\rangle \check{\mathbf{e}}_\rho
	\end{align*}
	which implies $\check{\zeta} = \langle \mathbf{d}_\rho, \check{\xi}\rangle \check{\mathbf{e}}_\rho$. Thus $\check{\zeta}$ is parallel to $\check{\Lambda}_\rho$ and it is a primitive generator of $\check{\Lambda}_\rho$ if and only if $\check{\xi}$ is a primitive generator of $\check{\Lambda}/\check{\Lambda}_\rho$. 
\end{proof}

\begin{rmk}
	Lemma \ref{prim} gives an explicit local description of the tropical cycles of interest to us. In particular, we will give a basis of $H_1(\Sk(U),\iota_*\check{\Lambda})$ that looks locally like Figure \ref{tropcyclepic1}. Note that the  decoration of the  loop around the focus-focus singularity is unique  once we fix the orientation of the edge emanating from the singularity, its decoration and the orientation of the loop. This is due to the balancing condition - for example, the figure below on the left satisfies the balancing condition but the one on the right is not balanced at the red cross.

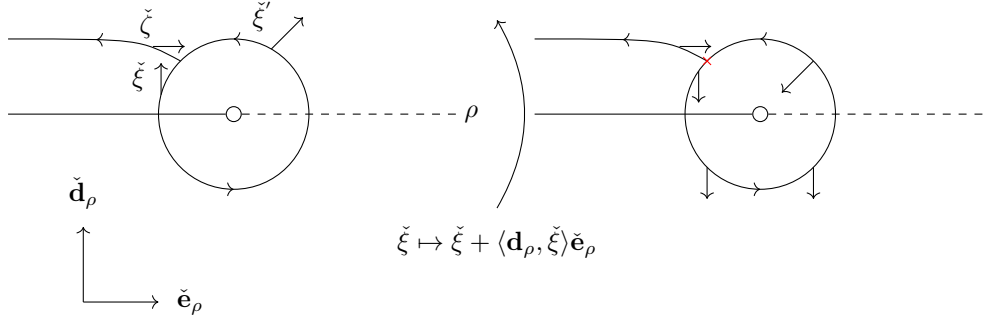
\begin{figure}[H]
	\centering
	\begin{tikzpicture}[circ/.style={shape=circle, inner sep=2pt, draw, node contents=}]
\begin{scope}[xshift = -3.5 cm]

	\draw  node (x) at (0,0) [circ]; 
	\draw [ decoration={markings,  mark=at position 0.25 with {\arrow{>}}, mark=at position 0.75 with {\arrow{>}}},         postaction={decorate}](0,0) circle [radius=1cm]; 
 	\draw [dashed](x)-- (3,0); 
 	\draw (x) -- (-3,0); 
 	
	\node [label=right:$\check{\xi}^{'}$] at (90:1.3cm){}; 
 	\node [label=right:${\check{\xi}}$] at (-1.6,0.5){}; 
 	\node [label=right:$\check{\zeta}$] at (-1.5,1.2){}; 
 	
 	\draw [->] (140:1.4cm) -- ([xshift=12pt, yshift=0pt] 140:1.4cm); 
 	\draw [->] (60:1cm) -- ([xshift=12pt, yshift=12pt] 60:1cm); 
 	\draw [->] (165:1cm) -- ([xshift=0pt, yshift=12pt] 165:1cm); 
 	\draw  [decoration={markings, mark=at position 0.5 with {\arrow{>}}}, postaction={decorate}](135:1cm) .. coordinate [pos=.3] (a) controls (-1.25,1) .. ([xshift=-3cm, yshift=1cm] 0:0cm ); 
 	
  	\draw [->] (-2,-2.5) -- (-1,-2.5);   
 	\draw [->] (-2,-2.5) -- (-2,-1.5); 
	\node [label=above:$\check{\mathbf{d}}_{\rho}$] at (-2,-1.5){}; 
\node [label=right:$\check{\mathbf{e}}_{\rho}$] at (-1,-2.5){}; 
    \node [label=left:$\rho$] at (3.5,0){};

	\draw [<-] (3.5,1.25) [out=300, in=60] to (3.5,-1.25);

	\node [label=below:$\check{\xi} \mapsto \check{\xi}+  \langle {\mathbf{d}_{\rho}, \check{\xi}} \rangle  \check{\mathbf{e}}_{\rho}$ ] at (3.5,-1.25) {};
	
\end{scope}
\begin{scope}[xshift = 3.5cm]
	\draw  node (x) at (0,0) [circ]; 
	\draw [ decoration={markings,  mark=at position 0.25 with {\arrow{>}}, mark=at position 0.75 with {\arrow{>}}},         postaction={decorate}](0,0) circle [radius=1cm]; 
 	\draw [dashed](x)-- (3,0); 
 	\draw (x) -- (-3,0); 
 	
%
 	\draw [->] (140:1.4cm) -- ([xshift=12pt, yshift=0pt] 140:1.4cm); 
 	
 	\draw [->] (-135:1cm) -- ([xshift=0pt, yshift=-12pt] (-135:1cm); 
 	\draw [->] (-45:1cm) -- ([xshift=0pt, yshift=-12pt] -45:1cm); 
 	\draw [->] (45:1cm) -- ([xshift= -12pt, yshift=-12pt] (45:1cm); 
 	\draw [->] (145:1cm) -- ([xshift=0pt, yshift=-12pt] (145:1cm); 
 	\draw  [decoration={markings, mark=at position 0.5 with {\arrow{>}}}, postaction={decorate}](135:1cm) .. coordinate [pos=.3] (a) controls (-1.25,1) .. ([xshift=-3cm, yshift=1cm] 0:0cm ); 
  \node[cross, red, inner sep=0pt, minimum size=3pt] (b) at (135:1cm) {}; 

%

%
%
\end{scope}
\end{tikzpicture} 

\end{figure}
We will discuss the case when there are multiple edges emanating from the loop in Remark \ref{multipleedge}. 
	\label{rmkorient}
\end{rmk}

We will now compute the rank of $H_1(\Sk(U),\iota_*\check{\Lambda})$. Here it is essential that the essential skeleton $\Sk(U)$ arises from the model $(\Yc,\Dc)$ from Proposition \ref{model} so that the rank of $H_1(\Sk(U),\iota_*\check{\Lambda})$ agrees with the rank of $D^\perp$. 

\begin{lem}
	We have the following equality:
	\begin{align*}
		\rank(H_1(\Sk(U)^\sm,\check{\Lambda}) = 2 \left(\sum_i k_i\right) -2 + \rank(H_0(\Sk(U)^\sm,\check{\Lambda})). 
	\end{align*}

\end{lem}
\begin{proof}
	By \cite[Proposition 2.5.4]{dimca}, we have $\chi(\Sk(U)^\sm,\check{\Lambda}) = \rank(\check{\Lambda})\cdot \chi(\Sk(U)^\sm)$. Rearranging and using the fact that $\Sk(U)^\sm$ is homeomorphic to $\R^2$ punctured at $\sum_i k_i$ distinct points, we have 
	\begin{align*}
		\rank \left(H_1(\Sk(U)^\sm, \check{\Lambda})\right) =  2\left(\sum_i k_i \right)-2  + \rank\left(H_0(\Sk(U)^\sm, \check{\Lambda})\right)
	\end{align*}
	as required. 
	\end{proof}
	
Consider the following short exact sequence of chain complexes 
\begin{align*}
	0 \rightarrow C_\bullet(\Gamma, \iota_*\check{\Lambda}) \rightarrow C_\bullet(\Sk(U),\iota_*\check{\Lambda}) \rightarrow C_\bullet(\Sk(U),\Gamma;\iota_*\check{\Lambda})\rightarrow 0 
\end{align*}
where $C_n(\Sk(U),\Gamma; \iota_*\check{\Lambda}) = C_n(\Sk(U), \iota_*\check{\Lambda}) /C_n(\Gamma, \iota_*\check{\Lambda})$. 
Taking the long exact sequence in homology gives the following exact sequence: 
\begin{align}
	\resizebox{.9\hsize}{!}{$0 \rightarrow H_1(\Sk(U), \iota_*\check{\Lambda}) \rightarrow H_1(\Sk(U),\Gamma;\iota_*\check{\Lambda}) \rightarrow H_0(\Gamma, \iota_*\check{\Lambda}) \rightarrow H_0(\Sk(U),\iota_*\check{\Lambda}) \rightarrow 0$}.
	\label{exacthomology}
\end{align}
Suppose that $\Sk(U)$ has exactly $q$ rays which have a singularity. Then by direct computation, we have that 
\begin{align*}
	\rank(H_1(\Sk(U),\Gamma; \iota_*\check{\Lambda})) = 
	\begin{cases}
		 \sum_i k_i + \sum_{i : \ k_i\neq 0}(k_i-1)  \ + (q-2),  &\text{ if } q\geq 3\\
		 \sum_i k_i + \sum_{i : \ k_i\neq 0}(k_i-1) + T_2   &\text{ if } q=1,2\\
         0 &\text{ if } q=0
	\end{cases}
\end{align*}
where $T_2 \in \{0,1\}$ is equal to the number of realised simple tropical spokes consisting of exactly two vertices. Rearranging gives 
\begin{align*}
	\rank\left(H_1(\Sk(U), \iota_* \check{\Lambda})\right)  & =\rank(H_1(\Sk(U),\Gamma; \iota_*\check{\Lambda}))  - \sum_ik_i   \\ & -\rank\left(\ker \left(H_0(\Gamma,\iota_*\check{\Lambda}\right)\rightarrow H_0(\Sk(U),\iota_*\check{\Lambda})\right).
\end{align*}

There are two cases to consider now: (i) if the singularities all lie on a line in the fan $\overline{\Sigma}$ of the toric model $\overline{Y}$, then $\rank(H_0(\Sk(U),\iota_*\check{\Lambda})) = 1$ \ie there is a global cotangent vector on $\Sk(U)$; (ii) if all the singularities do not lie on a line, then $\rank(H_0(\Sk(U),\iota_*\check{\Lambda})) = 0$. In both cases, we have 
\begin{align*}
	\rank\left(H_1(\Sk(U), \iota_* \check{\Lambda})\right)  = 
	\begin{cases}
		 (\sum_{i : \ k_i\neq 0} k_i)-2  &\text{ if } q\geq 3\\
		  \sum_{i : \ k_i\neq 0}(k_i-1)  + T_2 &\text{ if } q=1,2 \\
             0 &\text{ if } q=0
	\end{cases}
\end{align*}
Then by Lemma \ref{rankDperp}, this is equal to the rank of $D^{\perp}$. Therefore, we get the following lemma.

\begin{lem}
	We have the equality 
	\begin{align*}
		\rank(D^\perp) = \rank\left(H_1(\Sk(U), \iota_* \check{\Lambda})\right).
	\end{align*}
\end{lem}

Observe that the map $H_1(\Sk(U)^\sm,\check{\Lambda}) \rightarrow H_1(\Sk(U),\iota_*\check{\Lambda})_\text{tf}$ to the torsion-free component of $H_1(\Sk(U),\iota_*\check{\Lambda})$  is surjective and the kernel has rank $\sum_i k_i$. A generating set of the kernel is given by loops around each singularity decorated with the invariant primitive cotangent vector. 
\begin{rmk}
Instead of having a single emanating edge from a loop around a singular vertex as depicted in Figure \ref{tropcyclepic1}, consider the case where there are $k\geq 2$ emanating edges pointing away from the vertex. The figure below illustrates the case $k=3$.  Let each edge be decorated with the cotangent vector $\check{\mathbf{a}}_i$ with $\check{\mathbf{c}}$ the image of $\check{\mathbf{b}}$ under the monodromy transformation. 

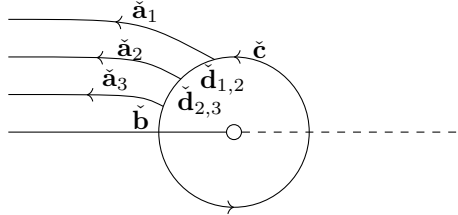
\begin{figure}[H]
\begin{center}
\begin{tikzpicture}[circ/.style={shape=circle, inner sep=2pt, draw, node contents=}]

	\draw  node (x) at (0,0) [circ]; 
	\draw [ decoration={markings,  mark=at position 0.25 with {\arrow{>}}, mark=at position 0.75 with {\arrow{>}}},         postaction={decorate}](0,0) circle [radius=1cm]; 
 	\draw [dashed](x)-- (3,0); 
 	\draw (x) -- (-3,0); 
 	
 	\node [label=right:$\check{\mathbf{a}}_1$] at (-1.6,1.6){}; 
 	\node [label=right:$\check{\mathbf{a}}_2$] at (-1.8,1.1){}; 
 	\node [label=right:$\check{\mathbf{a}}_3$] at (-2,0.7){};

 	\node [label=right:$\check{\mathbf{b}}$] at (-1.6,0.2){}; 
 	\node [label=right:$\check{\mathbf{c}}$] at (0,1.1){}; 
 	\draw  [decoration={markings, mark=at position 0.5 with {\arrow{>}}}, postaction={decorate}](135:1cm) .. coordinate [pos=.3] (a) controls (-1.25,1) .. ([xshift=-3cm, yshift=1cm] 0:0cm );

 	\draw  [decoration={markings, mark=at position 0.5 with {\arrow{>}}}, postaction={decorate}](105:1cm) .. coordinate [pos=.3] (a) controls (-1.25,1.5) .. ([xshift=-3cm, yshift=1.5cm] 0:0cm ); 
 	
 	\draw  [decoration={markings, mark=at position 0.5 with {\arrow{>}}}, postaction={decorate}](160:1cm) .. coordinate [pos=.3] (a) controls (-1.25,0.5) .. ([xshift=-3cm, yshift=0.5cm] 0:0cm ); 
 	

%
%

\node [label=right:$\check{\mathbf{d}}_{1,2}$] at (-0.7,0.7){}; 
 \node [label=right:$\check{\mathbf{d}}_{2,3}$] at (-1,0.4){}; 
\end{tikzpicture} 	
\end{center}
\end{figure}

From the exact sequence (\ref{exacthomology}), we have the morphism $$H_1(\Sk(U),\iota_*\check{\Lambda}) \rightarrow H_1(\Sk(U),\Gamma;\iota_*\check{\Lambda}).$$
The image of such a tropical cycle under this morphism will have $k$ edges emanating from the singular vertex, and each will be decorated with the invariant cotangent vector with the relation $\sum_i \check{\mathbf{a}}_i = 0$.   We also have the balancing conditions 
\begin{align*}
	\check{\mathbf{a}}_1 + \check{\mathbf{c}} &=\check{\mathbf{d}}_{1,2}\\ 
	\check{\mathbf{a}}_i + \check{\mathbf{d}}_{i-1,i} &= \check{\mathbf{d}}_{i,i+1} \text{ for }  1 < i < k\\ 
	\check{\mathbf{a}}_k +\check{\mathbf{d}}_{k-1,k}  & = \check{\mathbf{b}}.
\end{align*}
This implies $\mathbf{b} = \mathbf{c}$ and in particular $\mathbf{b}$ is parallel to the invariant cotangent direction. In particular, a tropical 1-cycle can be locally depicted as having one emanating edge. 
\label{multipleedge}
\end{rmk}

\begin{rmk}
	As observed in \cite{ruddat2021homology}, $H_1(\Sk(U), \iota_*\check{\Lambda})$ may have torsion. For example, consider the Looijenga pair $(Y,D)$ with toric model given by $\F_2$ and non-toric blowups along both fibres, whose essential skeleton is depicted below. Since the invariant cotangent directions do not span $\check{\Lambda}_O$, a simple \v{C}ech homology computation shows that $H_1(\Sk(U), \iota_*\check{\Lambda}) \cong \Z/2\Z$.  The torsion elements of $H_1(\Sk(U),\iota_*\check{\Lambda})$ have no importance in our theory and we will only work with the torsion-free part. In the work of \cite{ruddatsiebert}, the construction of the mirror family depends on a choice of splitting of $H_1(\Sk(U),\iota_*\check{\Lambda}) \rightarrow H_1(\Sk(U),\iota_*\check{\Lambda})_\text{tf}$ by the torsion subgroup $H_1(\Sk(U),\iota_*\check{\Lambda})_\text{tors}$. 
	
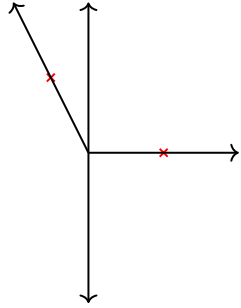
\begin{figure}[H]

\begin{center}
\tikzset{every picture/.style={line width=0.75pt}} 

\begin{tikzpicture}


\node[cross, color=red, inner sep=0pt, minimum size=3pt] (b) at (-0.5,1) {}; 
\node[cross, color=red, inner sep=0pt, minimum size=3pt] (b) at (1,0) {}; 
\draw [->] (0,0) -- (2,0);
\draw [->] (0,0) -- (0,2);
\draw [->] (0,0) -- (-1,2);
\draw [->] (0,0) -- (0,-2);

\end{tikzpicture}
\end{center}
	\end{figure}
\end{rmk}

By the rank computation above, we see that the set of simple tropical spokes and simple tropical wings generate $H_1(\Sk(U),\iota_*\check{\Lambda})_\text{tf}$.

\subsection{Tropical spokes and wings}
\label{tropspokewing}
\begin{defn}

A \textit{simple tropical spoke}, $\Gamma$, is a tropical $1$-cycle whose underlying graph circles $n$ distinct singularities where each is the first singularity (closest to the origin) on a ray and an additional $n$-valent vertex which all the edges emanating from the circled singularities meet at. Furthermore, $\Gamma$, is not an integer multiple of another tropical $1$-cycle.
\end{defn}

Denote the group generated by simple tropical spokes under addition of $1$-cycles as \underline{TropSpoke}.

\begin{defn} A \textit{simple tropical wing} is a tropical $1$-cycle that goes around two adjacent singularities on the same ray such that the edges leaving each singularity are the same and are contained in a single $2$-cell. Furthermore, we say that a simple tropical wing is \textit{positively oriented} if the edge connecting the two loops is oriented towards the $i$ singularity and decorated with the invariant cotangent direction that points away from the origin, as depicted below. 

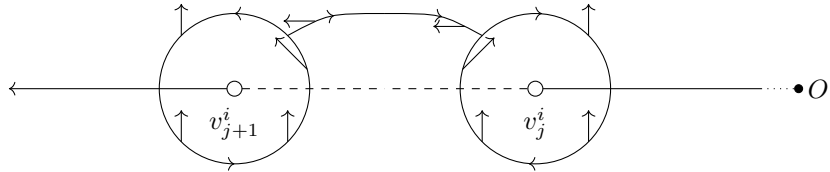
\begin{figure}[H]
    \begin{tikzpicture}[circ/.style={shape=circle, inner sep=2pt, draw, node contents=}]
	\begin{scope}[xshift = -4cm]
	\draw  [decoration={markings, mark=at position 0.5 with {\arrow{>}}}, postaction={decorate}](45:1cm) .. coordinate [pos=.3] (a) controls (1.25,1) .. ([xshift=2cm, yshift=1cm] 0:0cm); 
	\draw  node (y) at (0,0) [circ]; 
	\node [label=below:$v_{j+1}^{i}$] at (0,0){}; 
	\draw [ decoration={markings,  mark=at position 0.25 with {\arrow{>}}, mark=at position 0.75 with {\arrow{>}}},         postaction={decorate}](0,0) circle [radius=1cm]; 
 	\draw [dashed](0.1,0)-- (2,0); 
	\draw[->] (-0.1,0) -- (-3,0); 
 	
 	
 	\draw [->] (40:1.4cm) -- ([xshift=-12pt, yshift=0pt] 40:1.4cm); 
 	\draw [->] (-45:1cm) -- ([xshift=0pt, yshift=12pt] -45:1cm); 
 	\draw [->] (15:1cm) -- ([xshift=-12pt, yshift=12pt] 15:1cm); 
 	\draw [->] (135:1cm) -- ([xshift=0pt, yshift=12pt] 135:1cm); 
 	\draw [->] (-135:1cm) -- ([xshift=0pt, yshift=12pt] (-135:1cm); 
 	
 	\end{scope}
    \begin{scope}

    \draw  node (x) at (0,0) [circ]; 
	\node [label=below:$v_{j}^{i}$] at (0,0){}; 

	\draw [ decoration={markings,  mark=at position 0.25 with {\arrow{<}}, mark=at position 0.75 with {\arrow{<}}},         postaction={decorate}](0,0) circle [radius=1cm]; 
 	\draw [dashed](x) -- (-2,0); 
 	\draw (0.1,0)-- (3,0); 
 	\draw[dotted] (3,0) -- (3.5,0);
    \filldraw[black] (3.5,0) circle (1.5pt) node[anchor=west]{$O$};
 	
 	\draw [->] (138.5:1.25cm) -- ([xshift=-12pt, yshift=0pt] 138.5:1.25cm); 
 	\draw [->] (-135:1cm) -- ([xshift=0pt, yshift=12pt] -135:1cm); 
 	\draw [->] (-45:1cm) -- ([xshift=0pt, yshift=12pt] -45:1cm); 
 	\draw [->] (45:1cm) -- ([xshift=0pt, yshift=12pt] (45:1cm); 
 	\draw [->] (165:1cm) -- ([xshift=12pt, yshift=12pt] 165:1cm); 
 	\draw  [decoration={markings, mark=at position 0.5 with {\arrow{<}}}, postaction={decorate}](135:1cm) .. coordinate [pos=.3] (a) controls (-1.25,1) .. ([xshift=-2cm, yshift=1cm] 0:0cm );
    \end{scope}
    \end{tikzpicture}
\end{figure}
\end{defn}

We make the following convention for `local tropical cycles':
\begin{conv}
A counterclockwise loop around a singularity with trailing edge directed towards the origin and decorated with the invariant cotangent direction away from the origin is \textit{positively oriented}. 
\end{conv}

\begin{eg}

On the left of Figure \ref{orient}, we have a tropical cycle for the Looijenga pair arising from performing a single non-toric blowup on each boundary component of $\proj^2$. Under our correspondence, this tropical cycle corresponds to $ [E_1]+[E_2] + [E_3] -[L]$ where $E_i$ are the exceptional curves and $L$ is a representative of the pullback of the line class on $\proj^2$.  

On the right of the figure, we have a tropical cycle for the Looijenga pair with toric model $\F_1$ by performing a non-toric blowup along one of the fibres and the other two components. The tropical cycle corresponds to $-[C] +[E_2] - [E_1] - [E_3]$ where $C$ is the pullback of the $(-1)$-curve on $\F_1$ and $E_2$ is the exceptional curve with center on a fibre of $\F_1$. 

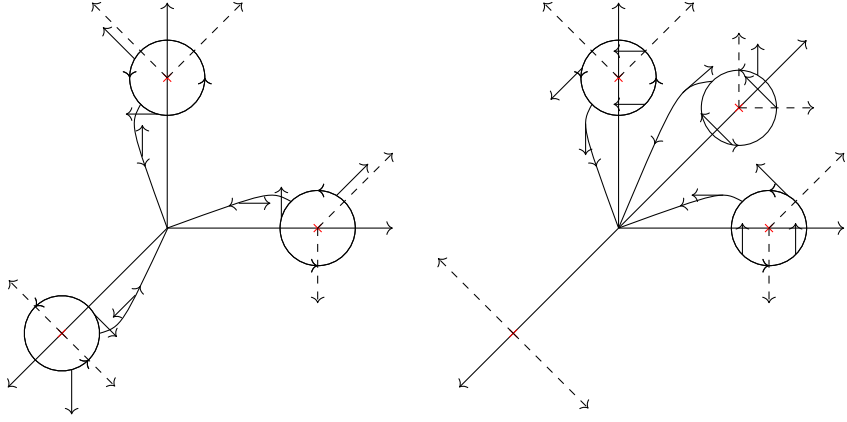
\begin{figure}[H]
\begin{center}
\begin{tikzpicture}

	\begin{scope}
			\draw[->](0,0) -- (3,0);
		\draw[->] (0,0) -- (0,3); 
		\draw[->] (0,0) -- (-2.12,-2.12);
		\node[cross, color=red, inner sep=0pt, minimum size=3pt] (b) at (2,0) {}; 
		\node[cross, color=red, inner sep=0pt, minimum size=3pt] (b) at (0,2) {}; 
		\node[cross, color=red, inner sep=0pt, minimum size=3pt] (b) at (-1.4,-1.4) {};
		

 	\begin{scope}[xshift = 2cm]
 	\draw  node at (0,0) {}; 

  	\draw[dashed,->] (0,0) -- (1,1);
  	\draw[dashed,->] (0,0) -- (0,-1);

\draw [ decoration={markings,  mark=at position 0.25 with {\arrow{>}}, mark=at position 0.75 with {\arrow{>}}},         postaction={decorate}](0,0) circle [radius=0.5cm];

	\draw [ decoration={markings,  mark=at position 0.25 with {\arrow{>}}, mark=at position 0.75 with {\arrow{>}}},         postaction={decorate}](0,0) circle [radius=0.5cm]; 

 	\draw [->] (162.5:1.1cm) -- ([xshift=12pt, yshift=0pt] 162.5:1.1cm); 
 	\draw [->] (60:0.5cm) -- ([xshift=12pt, yshift=12pt] 60:0.5cm); 
 	\draw [->] (165:0.5cm) -- ([xshift=0pt, yshift=12pt] 165:0.5cm); 
 	
      	\draw  [decoration={markings, mark=at position 0.5 with {\arrow{>}}}, postaction={decorate}](135:0.5cm) .. coordinate [pos=.3] (a) controls (-0.625,0.5) .. ([xshift=-2cm, yshift=0cm] 0:0cm );

	\end{scope}

	\begin{scope}[yshift = 2cm, rotate=90]
		\draw  node (x) at (0,0) {}; 
		\draw[dashed,->] (0,0) -- (1,1);
  		\draw[dashed,->] (0,0) -- (1,-1);
  	\draw [ decoration={markings,  mark=at position 0.25 with {\arrow{>}}, mark=at position 0.75 with {\arrow{>}}},         postaction={decorate}](0,0) circle [radius=0.5cm];

	\draw [ decoration={markings,  mark=at position 0.25 with {\arrow{>}}, mark=at position 0.75 with {\arrow{>}}},         postaction={decorate}](0,0) circle [radius=0.5cm]; 

 	\draw [->] (162.5:1.1cm) -- ([xshift=12pt, yshift=0pt] 162.5:1.1cm); 
 	\draw [->] (60:0.5cm) -- ([xshift=12pt, yshift=12pt] 60:0.5cm); 
 	\draw [->] (165:0.5cm) -- ([xshift=0pt, yshift=12pt] 165:0.5cm); 
 	
      	\draw  [decoration={markings, mark=at position 0.5 with {\arrow{>}}}, postaction={decorate}](135:0.5cm) .. coordinate [pos=.3] (a) controls (-0.625,0.5) .. ([xshift=-2cm, yshift=0cm] 0:0cm ); 
 	\end{scope}

	\begin{scope}[yshift = -1.4cm, xshift = -1.4cm, rotate=225]

	\draw[dashed,->] (0,0) -- (0,1);
	\draw[dashed,->] (0,0) -- (0,-1);
	\draw  node at (0,0) {}; 

\draw [ decoration={markings,  mark=at position 0.25 with {\arrow{>}}, mark=at position 0.75 with {\arrow{>}}},         postaction={decorate}](0,0) circle [radius=0.5cm];

	\draw [ decoration={markings,  mark=at position 0.25 with {\arrow{>}}, mark=at position 0.75 with {\arrow{>}}},         postaction={decorate}](0,0) circle [radius=0.5cm]; 

 	\draw [->] (162.5:1.1cm) -- ([xshift=12pt, yshift=0pt] 162.5:1.1cm); 
 	\draw [->] (60:0.5cm) -- ([xshift=12pt, yshift=12pt] 60:0.5cm); 
 	\draw [->] (165:0.5cm) -- ([xshift=0pt, yshift=12pt] 165:0.5cm); 
 	
      	\draw  [decoration={markings, mark=at position 0.5 with {\arrow{>}}}, postaction={decorate}](135:0.5cm) .. coordinate [pos=.3] (a) controls (-0.625,0.5) .. ([xshift=-2cm, yshift=0cm] 0:0cm );

	\end{scope}

	\end{scope}
	
	\begin{scope}[xshift = 6cm]
			\begin{scope}
			\draw[->](0,0) -- (3,0);
		\draw[->] (0,0) -- (0,3); 
		\draw[->] (0,0) -- (-2.12,-2.12);
		\draw[->] (0,0) -- (2.5,2.5);
		\node[cross, color=red, inner sep=0pt, minimum size=3pt] (b) at (1.6,1.6) {}; 

		\node[cross, color=red, inner sep=0pt, minimum size=3pt] (b) at (2,0) {}; 
		\node[cross, color=red, inner sep=0pt, minimum size=3pt] (b) at (0,2) {}; 
		\node[cross, color=red, inner sep=0pt, minimum size=3pt] (b) at (-1.4,-1.4) {};

 	\begin{scope}[xshift = 2cm]
 	\draw  node (x) at (0,0) {}; 
	\draw [ decoration={markings,  mark=at position 0.25 with {\arrow{>}}, mark=at position 0.75 with {\arrow{>}}},         postaction={decorate}](0,0) circle [radius=0.5cm]; 
  		\draw[dashed,->] (0,0) -- (1,1);
  		\draw[dashed,->] (0,0) -- (0,-1);
\draw [ decoration={markings,  mark=at position 0.25 with {\arrow{>}}, mark=at position 0.75 with {\arrow{>}}},         postaction={decorate}](0,0) circle [radius=0.5cm];
	
	\draw [->] (144.25:0.75cm) -- ([xshift=-12pt, yshift=0pt] (144.25:0.75cm); 
 	\draw [->] (-135:0.5cm) -- ([xshift=0pt, yshift=12pt] (-135:0.5cm); 
 	\draw [->] (-45:0.5cm) -- ([xshift=0pt, yshift=12pt] -45:0.5cm); 
 	\draw [->] (60:0.5cm) -- ([xshift=-12pt, yshift=12pt] 60:0.5cm); 
 	\draw  [decoration={markings, mark=at position 0.5 with {\arrow{>}}}, postaction={decorate}](135:0.5cm) .. coordinate [pos=.3] (a) controls (-0.625,0.5) .. ([xshift=-2cm, yshift=0cm] 0:0cm ); 	\end{scope}

	\begin{scope}[yshift = 2cm, rotate=90]
		\draw  node (x) at (0,0) {}; 
	\draw [ decoration={markings,  mark=at position 0.25 with {\arrow{>}}, mark=at position 0.75 with {\arrow{>}}},         postaction={decorate}](0,0) circle [radius=0.5cm]; 
	
		\draw[dashed,->] (0,0) -- (1,1);
  		\draw[dashed,->] (0,0) -- (1,-1);
   	\draw [ decoration={markings,  mark=at position 0.25 with {\arrow{>}}, mark=at position 0.75 with {\arrow{>}}},         postaction={decorate}](0,0) circle [radius=0.5cm];
	
	\draw [->] (144.25:0.75cm) -- ([xshift=-12pt, yshift=0pt] (144.25:0.75cm); 
 	\draw [->] (-135:0.5cm) -- ([xshift=0pt, yshift=12pt] (-135:0.5cm); 
 	\draw [->] (-45:0.5cm) -- ([xshift=0pt, yshift=12pt] -45:0.5cm); 
 	\draw [->] (75:0.5cm) -- ([xshift=-12pt, yshift=12pt] 75:0.5cm); 
 	\draw  [decoration={markings, mark=at position 0.5 with {\arrow{>}}}, postaction={decorate}](135:0.5cm) .. coordinate [pos=.3] (a) controls (-0.625,0.5) .. ([xshift=-2cm, yshift=0cm] 0:0cm );	\end{scope}

	\begin{scope}[yshift = 1.6cm, xshift = 1.6cm, rotate=0 ]
	\draw[dashed,->] (0,0) -- (1,0);
	\draw[dashed,->] (0,0) -- (0,1);
	
		\draw  node (x) at (0,0) {}; 
	\draw [ decoration={markings,  mark=at position 0.25 with {\arrow{>}}, mark=at position 0.75 with {\arrow{>}}},         postaction={decorate}](0,0) circle [radius=0.5cm]; 
  	
 	\draw [->] (165:0.775cm) -- ([xshift=12pt, yshift=12pt] 170:0.775cm); 
 	\draw [->] (0:0.5cm) -- ([xshift=-12pt, yshift=12pt] (0:0.5cm); 
	\draw [->] (-100:0.5cm) -- ([xshift=-12pt, yshift=12pt] (-100:0.5cm); 
 	\draw [->] (60:0.5cm) -- ([xshift=0pt, yshift=12pt] 60:0.5cm); 
 	\draw  [decoration={markings, mark=at position 0.5 with {\arrow{>}}}, postaction={decorate}](135:0.5cm) .. coordinate [pos=.3] (a) controls (-0.8,0.3) .. ([xshift=-1.6cm, yshift=-1.6cm] 0:0cm ); 
	\end{scope}

	\begin{scope}[xshift = -1.4cm, yshift = -1.4cm]
	\draw[dashed,->] (0,0) -- (-1,1);
	\draw[dashed,->] (0,0) -- (1,-1);
	\end{scope}
	\end{scope}
	\end{scope}
\end{tikzpicture}	
\end{center}

	\caption{Tropical cycles corresponding to tropical spokes.}

	\label{orient}

\end{figure}
\end{eg}

Denote the group generated by simple tropical wings under addition of $1$-cycles as \underline{TropWing}.




\begin{lem}
\label{lem:tropwing}
Let $k_i$ be the number of singularities on $\rho_i$. Then there is a group isomorphism 
	\begin{align*}
		\underline{\text{TropWing}} \longrightarrow \Z^{\sum_{k:r_k \neq 0}(r_i-1)}.
	\end{align*}
\end{lem}
\begin{proof}
This is clear from the definition of tropical wings. 
\end{proof}

\begin{lem}	
\label{primtropspinecount} 
	Suppose that $q\ge 3$ rays in $\Sk(U)$ have singularities. Then there are exactly $q-2$ simple tropical spokes. 
\end{lem}

\begin{proof}
	We have $\sum_i k_i = k$ and $ q+ \sum_{i: k_i \neq 0} (k_i-1) = k$. The condition of being a simple tropical spoke imposes $\sum_{i :k_i \neq 0}(k_i-1)$ independent relations and thus defines a codimension $\sum_{i :k_i \neq 0}(k_i-1)$ lattice of $\Z^{k-2}$. This defines a sublattice of $\Z^{k-2}$ of rank $k-2 - \sum_{i: k_i \neq 0} (k_i-1) = {q-2}$ whose generators are precisely the simple tropical spokes. 
\end{proof}

\begin{lem}
The set of simple tropical spokes and simple tropical wings form a basis for $H_1(\Sk(U),\iota_*\check{\Lambda})_\text{tf}$. 
\end{lem}

\begin{proof}
First, suppose that we are in the case where there are either $0$ or $1$ singularities, then there are no tropical spokes or wings which aligns with the fact that $H_1(\Sk(U),\iota_*\check{\Lambda})=0$. 

If there are $2$ singularities that are on the same ray, there is a single simple wing with no spokes. Likewise, if they are on different rays in which one ray is negative of the other, then there is a single simple tropical spoke with no wings. The other cases of $2$ singularities has $H_1(\Sk(U),\iota_*\check{\Lambda})=0$ and no tropical spokes or wings.

Now suppose there are $k\ge 3$ singularities spread across $q$ distinct rays. If $q\ge 3$, then by Lemma \ref{primtropspinecount}, there are $q-2$ simple tropical spokes and by Lemma \ref{lem:tropwing}, there are $\sum (k_i-1)$ simple tropical wings. In the case of $q=1$ or $q=2$, the only difference is that there are no tropical spokes.

We then to show that every $1$-cycle can be generated by simple tropical spokes and simple tropical wings. Let $\gamma$ be a tropical $1$-cycle. Suppose that it goes around singularities $\{x_{ij}\}$ with respective multiplicities $c_{ij}$. Then consider the cycle $\beta_{ij}$ that goes around singularities $x_{ij}$ and $x_{i1}$ with multiplicity $-c_{ij}$ and $c_{ij}$, respectively. Then consider $\gamma+\sum \beta_{ij}$. This will be a tropical $1$-cycle where the only singularities that it can go around are the first ones on every ray. Then we see that what remains must be a tropical spoke. 
\end{proof}

%
%

	  \subsection{Tropicalisation}
\label{tropicalisationsec}
 We now construct a homomorphism
\begin{align*}
	\Trop : D^\perp \longrightarrow H_1(\Sk(U), \iota_* \check{\Lambda})
\end{align*}
and show that it is an isomorphism. This is an analogue of the tropicalisation map in tropical and toric geometry. Note we can decompose Pic($Y$) as follows: 
\begin{align*}
	\Pic(Y) &= \Pic(\overline{Y})\oplus \bigoplus_{i,j}\Z[E_{j}^i]
\end{align*}
where $E_{j}^i$ is an exceptional divisor with centre on $D_i$ and $\Pic(\overline{Y})$ is the pullback of the Picard group of the toric model $\overline{Y}$. Furthermore, recall that  $E_j^i$ corresponds to the irreducible component $Y_j^i$.   Since $\Pic(\overline{Y})$ is generated by the $\overline{D}_i$ for $i=1,...,\ell(D)$,  we can write every $\alpha\in D^\perp$ as \begin{align*}
	\alpha = \alpha_{\textbf{t}} +\alpha_{\textbf{nt}}
\end{align*}
where $\alpha_{\textbf{t}} \in \bigoplus_i \Z[ D_i]$  and $\alpha_{\textbf{nt}} \in \bigoplus_{i,j}\Z[E_{j}^i]$. 

\begin{defn}
	A \textit{contact order} is a tuple $(\check{\lambda}_i)_{i \in I^\hb}$ with $\check{\lambda}_i \in \check{\Lambda}_{\rho_i}$. 
\end{defn}

\begin{eg}
	Let $(Y,D)$ be a toric Looijenga pair and $\beta$ a curve class on $Y$. Then the contact order associated to $\beta$ is $\left((\beta \cdot D_i)\check{\mathbf{e}}_i\right)_{i\in I^\hb}$ where $\check{\mathbf{e}}_i \in \check{\Lambda}_{\rho_i}$ is the primitive unit vector corresponding to $1$. 
\end{eg}

We also have the converse: given an element $\alpha_\mathbf{t}= \sum_i \beta_i D_i  \in \Pic(\overline{Y})$, the contact order completely determines it \cite[Theorem 2.1]{fultonintersecttoric}.  We thus define the tropicalisation $\Trop(\alpha_\mathbf{t})$  to the subset  of $\overline{\Sk}(U)$ consisting of the closure of the rays $\rho_i$ whenever $\alpha_\mathbf{t} \cdot D_i \neq 0$  equipped with the weight  $(\alpha_\mathbf{t} \cdot D_i) \check{\mathbf{e}}_i$.  
 By \cite[\S3]{fultontoric}, the divisor $\alpha_\mathbf{t}$ defines a continuous piecewise linear function $\phi_{\alpha_\mathbf{t}}$ on the support of the fan $\overline{\Sigma}$ of  $\overline{Y}$ given by $\phi_{\alpha_\mathbf{t}}(\mathbf{v}_i) = - \alpha_\mathbf{t}\cdot D_i$ and extended linearly, where  $\mathbf{v}_i$ is the primitive integral generator of the ray in $\overline{\Sigma}$ corresponding to $D_i$. Since $\phi_{\alpha_\mathbf{t}}$ is defined at the origin, we must have $\sum_i (-\alpha_\mathbf{t}\cdot D_i) \mathbf{v}_i = 0$. In particular, the tropicalisation is balanced at $O$ and thus defines an infinite tropical spoke.

We next describe the tropicalisation of a non-toric exceptional curve $E_j^i$ on $Y$.    Then the image of $(E_j^i)^\an$ under the non-archimedean SYZ fibration $\rho_{(\Yc,\Dc)}$  is the ray emanating from $v_j^{i}$ along $\rho_i$; we denote this ray by $r_{j}^{i}$. Then the tropicalisation of $\gamma E$ with $\gamma \in \Z$ is the ray $r_j^{i}$ equipped with the weight $\gamma \in \check{\Lambda}_{\rho_i}$. 
We record the following observation: 

\begin{lem}
	\label{welldef}
	Let $\alpha \in D^\perp$ and write $\alpha = \alpha_{\mathbf{t}} + \alpha_{\mathbf{nt}} = \sum_i \beta_i D_i + \sum_{i,j}\gamma_{j}^iE_{j}^i$ with $\alpha_\mathbf{t} \neq 0$. If there exists a $j$ such that $\gamma_{j}^i \neq 0$, then $\alpha_\mathbf{t}\cdot D_i \neq 0$. If $j$ is unique, then $\alpha_\mathbf{t} \cdot D_i = \gamma_{j}^i$. \end{lem}
\begin{proof}
Since $(\alpha_\mathbf{t} +\alpha_{\mathbf{nt}})\cdot D_i = 0$ and $\alpha_{\mathbf{nt}}\cdot D_i \neq 0$, we have $\alpha_{\mathbf{t}} \cdot D_i \neq 0.$
\end{proof}

Given $\alpha \in D^\perp$, we can always write $\alpha$ as a sum of elements  of $ D^\perp \cap \oplus_{r,s}\Z[E_{s}^r]$ and elements of the form $$ \sum_i \beta_i D_i + \sum_{i,j}\gamma_{j}^iE_{j}^i$$  such that for each $i$ there exists a unique $j$ such that $\gamma_{j}^i\neq 0$. Indeed, write $\alpha = \alpha_\mathbf{t} + \alpha_{\mathbf{nt}}$ and suppose there exist $j$ and $j'$ such that $\gamma^i_{j}, \gamma_{j'}^i \neq 0$. Then $\epsilon = E_{j}^i - E_{j'}^i \in D^\perp$ and we can run the same argument on $\alpha \pm \epsilon$ until we have the desired form. This leads to the following definition.   



\begin{defn}
    We call $\alpha \in D^{\perp}$ \textit{primitive} if $\alpha_{\mathbf{t}}\ne 0$ and $\alpha_{\mathbf{nt}}$ only has non-zero contributions from classes $E_1^j$. Then denote $D_S^{\perp}$ to be the subgroup of $D^{\perp}$ generated by primitive classes.
\end{defn}

\begin{defn}
Let $D^{\perp}_E$ be the subgroup of $D^{\perp}$ given by classes $\alpha$ such that $\alpha_{\textbf{nt}}=0$.
\end{defn}

\begin{lem}
$D^{\perp}_E$ is generated by elements of the form $E^i_{j+1}-E^i_{j}$.
\end{lem}

\begin{proof}
Let $\alpha\in D^{\perp}_E$. Then we can write
\[\alpha=\sum_{i,j}c_{ij} E_j^i.\]
Since $\alpha\in D^{\perp}$, we must have $\alpha\cdot D_i=0$ for all $i$. This gives that for each fixed $i$, $\sum_j c_{ij} = 0$. From here it is clear that we can write $\alpha$ as an integral linear combination of elements of the given form.
\end{proof}

\begin{lem}
We have $D^{\perp}\simeq D^{\perp}_S\oplus D^{\perp}_E$.
\end{lem}

\begin{proof}
Given $\alpha\in D^{\perp}$, we can write
\[ \alpha = \sum a_i L_i + \sum_k c_{1,k} E_1^k + \sum_{j} c_{2,j}E_2^j+\sum_{i>3,j} c_{i,j}E_i^j\]
Then consider
\[ \beta=\alpha +\sum (-c_{2,j}E_2^j+c_{2,j}E_1^j)+\sum_{i>3,j} (-c_{i,j} E_i^j+c_{i,j}E_1^j) \]
Then $\beta\in D_S^{\perp}$ and $\alpha-\beta\in D_E^{\perp}$ so we map $\alpha$ to $(\beta,\alpha-\beta)$ which gives the isomorphism.

\end{proof}

\begin{thm}(Tropical correspondence)
\label{tropcorr}
	There is an isomorphism of groups
	\begin{align*}
		T: D^\perp \longrightarrow H_1(\Sk(U),\iota_*\check{\Lambda}) _\text{tf}. 
	\end{align*}
\end{thm}

\begin{proof}
We have the following decompositions:
\[ D^{\perp}\simeq D^{\perp}_S\oplus D^{\perp}_E\]
\[H_1(\Sk(U),\iota_*\check{\Lambda}) \simeq \underline{\text{TropSpoke}}\oplus \underline{\text{TropWing}}\]

Now suppose we are given an element $\alpha\in D^{\perp}_S$ so we can write
\[ \alpha=\alpha_{\mathbf{t}}+\alpha_{\mathbf{nt}}.\]
We form isomorphisms
\[D_S^{\perp}\simeq \underline{\text{TropSpoke}} \hspace{1cm} D_E^{\perp}\simeq \underline{\text{TropWing}}\]
We know that $D^{\perp}_E$, is generated by $E_{j+1}^i-E_{j}^i$. We map this to the positively oriented tropical wing that circles around the $j+1$ and $j$-th singularity on $\rho_i$, which induces an isomorphism.\\\\
Now consider $D^{\perp}_S$ which is generated by primitive classes. Let $\beta$ be a primitive class. Then we associate to it a tropical spoke.
\begin{align*} 
\beta &= \beta_{\mathbf{t}} +\beta_{\mathbf{nt}}\\
&=\sum_i \beta_i D_i+\sum \gamma_j^iE_1^i
\end{align*}
Then we define $T(\beta)$ to be the $1$-cycle obtained by going around the first singularity on $\rho_i$ with the edge leaving the singularity labeled with $(\alpha_t\cdot D_i)\check{e}_i$ if $\alpha_t\cdot D_i\ne 0$ which is a simple tropical spoke.\\
In the other direction, given a simple tropical spoke, the vectors on the edges leaving each singularity define contact orders which defines a class $\beta_{\mathbf{t}}$. Then, we map it to the associated primitive class by adding the adequate exceptional curves as in Lemma \ref{welldef}. 
\end{proof}

\begin{rmk}
If $L = E_{j}^i- E_{k}^i$ with $k= j\pm 1$, then $L$ corresponds to a simple tropical wing with orientation equal to the sign of $j-k$. 
	Given $L \in D^\perp$, if $L = \alpha_{\mathbf{t}} + \sum_{i,j}\gamma_{j}^iE_{j}^i$ corresponds to a primitive tropical spoke, then the coefficient $\gamma_{j}^i$ gives the orientation $\text{sign}(\gamma_{j}^i)$ and multiplicity $|\gamma_j^i|$ around the corresponding singularity. 
	\label{orientmatch}
	\end{rmk}

\section{$K$-affine structures}
\label{Kaff}


 \subsection{Construction}
\label{constructKaff}
\begin{defn}
\label{Kaffdef}
Let $B$ be an integral affine manifold. Denote by $\Aff_\Z$  the sheaf of integral affine functions on $B$; this fits into the short exact sequence 
\begin{align*}
	0 \longrightarrow \R \longrightarrow \Aff_\Z \longrightarrow \check{\Lambda} \longrightarrow 0.
\end{align*}
A \textit{$K$-affine structure} on $B$ is an abelian sheaf $\Fc$ on $B$ which fits into the commutative diagram 
\begin{figure}[H]
\centering
	\begin{tikzcd}
		0 \ar[r] & K^\times \ar[r]\ar[d,"v_K"] & \Fc \ar[r]\ar[d] & \check{\Lambda} \ar[r]\ar[d, "\text{id}"] & 0 \\ 
		0 \ar[r] & \R \ar[r] & \Aff_\Z \ar[r] & \check{\Lambda} \ar[r] & 0
	\end{tikzcd}	
\end{figure}\noindent
where the rows are short exact sequences. 
\end{defn}

\begin{eg}
	Let $B= \R^n$ with the standard integral affine structure given by the embedding $\Z^n \subseteq \R^n$.  Then any $K$-affine structure on $B$ is isomorphic to $\check{\Lambda}\oplus K^\times$ since the short exact sequence above always splits. 
	\label{split}
\end{eg}

\begin{rmk}
A $K$-affine structure on $B$ is a $\GL_n(\Z)\ltimes (K^\times)^n$-torsor on $B$ such that applying $v_K^n$ to $(K^\times)^n$ gives the $\GL_n(\Z)\ltimes \R^n$-torsor associated to the integral affine structure on $B$. 
\end{rmk}
\begin{defn}
	To a $K$-affine structure $\Fc$, we can associate a monodromy representation 
\begin{align*}
		m :\pi_1(B, b) \longrightarrow \GL_n(\Z)\ltimes (K^\times)^n 
\end{align*}
for all $b \in B$ by composing the transition functions of $\Aff_K$ as we go around the loop. Post-composing with $v_K^n$ then recovers the monodromy representation associated to the integral affine structure $\Aff_\Z$. 
\label{monodromydef}
\end{defn}

As noted in \S\ref{affinoid}, an affinoid torus fibration induces an integral affine structure on the target. We review this construction below. It is sufficient to consider the case $X = \G_m^
n$ with $B = \R^n$. Then $\trop: X^\an \rightarrow B$ is an affinoid torus fibration trivially. 
 We first describe the case $X= \G_m^{n}$ and $B= \R^n$ with the affinoid torus fibration $\trop: X^\an \rightarrow B$.  Let $V\subset\R^n$ be a connected open set. In coordinates $(z_1,...,z_n)$ on $\trop^{-1}(V) \subset \G_m^{n,\an}$, every invertible analytic function $\varphi \in \Oc_{\G_m^{n,\an}}^\times(\trop^{-1}(V))$ can be written as a Laurent series 
\begin{align*}
	\varphi = \sum_{I \in \Z^n} c_Iz^I, c_I \in K
\end{align*}
such that for every $\underline{x} = (x_1,...,x_n) \in V$, $-\log|c_I| + \sum_{i=1}^n x_iI_i \rightarrow +\infty$ as $|I|\rightarrow \infty$.

\begin{defn}
Let $\varphi \in \Oc^\times_{\G_m^{n,\an}}(\trop^{-1}(V))$ and $x \in V$. By the maximum modulus principle, $|\varphi|$ is constant on the fibres of $\trop$.  Let $|\varphi(x)|$  denote the value of $|\varphi|$ on the fibre $\trop^{-1}(x)$ and define \begin{align*}
	\Val(\varphi)(x) := -\log|\varphi(x)| = \inf_I \left(-\log|c_I| + \sum_i x_iI_i\right). 
\end{align*}

\end{defn}

\begin{lem}
The map  $\Val: \trop_*(\Oc_{(\G_m^n)^\an}^\times) \rightarrow \Aff_\Z$ of sheaves on $\R^n$ satisfies the following properties: 
\begin{enumerate}[label = (\roman*)]
\item Let $a \in K^\times$. Then $\Val(a) = -\log|a|$.
	\item Let $V\subset \R^n$ be a connected open set and $\varphi\in \Oc_{({\G_m^n})^{\an}}^\times (\trop^{-1}(V))$.  There exists a unique index $I$, such that for all $x \in V$,
\begin{align*}
		\Val(\varphi) (x) = -\log|c_I| + \sum_ix_iI_i.
	\end{align*} 
	\item Given $\varphi_1,\varphi_2 \in \Oc_{({\G_m^n})^{\an}}^\times (\trop^{-1}(V))$, we have $$\Val(\varphi_1\varphi_2) = \Val(\varphi_1) + \Val(\varphi_2).$$
\end{enumerate}
\label{propval}
\end{lem}

\begin{proof}
	(i) and (iii) are immediate. 
For (ii), note that $\Val(\varphi)$ is a concave function since it is the infimum of affine functions. Thus $\Val(\varphi^{-1}) = - \Val(\varphi)$ is both convex and concave and in particular both $\Val(\varphi)$ and $\Val(\varphi^{-1})$ are affine. 
%
\end{proof}

\begin{eg}
	The integral affine structure on $\R^n$ can be then be described as the sheaf of functions of the form $\Val(\varphi)$ for $\varphi \in \trop_*(\Oc_{\G_m^{n,\an}}^\times)$.  
\end{eg}

Since $\Val(\varphi)$ is an affine function, we can make the following definition:

\begin{defn} Given $V \subseteq \R^n$ as above, let $\dVal(\varphi): V \rightarrow \check{\Lambda}$ be the derivative of $\Val(\varphi)$. This defines a morphism of sheaves $\dVal : 
\trop_*(\Oc_{\G_m^{n,\an}}^\times) \rightarrow \check{\Lambda}$. 
	
\end{defn}
More generally, suppose $X$ is an $n$-dimensional algebraic variety over $K$ and $\rho: X^\an \longrightarrow B$ an affinoid torus fibration  away from a  discriminant locus $\Gamma\subset B$ with $B^\sm = B \backslash \Gamma$.  Let $U\subset B^\sm$ be a connected open set which is homeomorphic to an open set of $\R^n$ and suppose the map $\rho$ locally looks like $\trop^{-1}(V) \subset \G_m^{n,\an}\longrightarrow V\subset B^\sm$. 
On $V$, the induced integral affine structure can then be described  as 
\begin{align*}
	\Aff_\Z(V) = \{\Val(\varphi) \ : \ \varphi \in \Oc_{X^\an}^\times (\rho^{-1}(V))\}. 
\end{align*}

Let $\Fc_1 = \ker(\Val)\subset \rho_*(\Oc_{X^\an}^\times)$; note $\Fc_1$ is precisely the sheaf of functions  which are zero in the induced integral affine structure on $B^\sm$ and is naturally an $\Oc_K^\times$-module. 
It can be alternatively described as $
	\Fc_1 = \{f \in \rho_*(\Oc_{X^\an}^\times) \ : \ |f| =1\}. 
$ 

 We then have the short exact sequence of sheaves on $B^\sm$
\begin{align*}
	0 \longrightarrow K^\times/\Oc_K^\times \longrightarrow \rho_*(\Oc_{X^\an}^\times)/\Fc_1 \xrightarrow{\dVal} \check{\Lambda} \rightarrow 0,
\end{align*}
where $K^\times/\Oc_K^\times$ embeds into $\rho_*(\Oc_{X^\an}^\times )/\Fc_1$ as constant functions.
Fix an analytic volume form on $X^\an$. More precisely, let $U$ be an analytic subspace of $X^\an$ such that $\rho^{-1}(B^\sm) \subseteq U$ and $\Omega \in \Gamma(U,\omega_X^{\otimes n})$ a top degree nowhere vanishing analytic volume form on $U$.
The analytic volume form $\Omega$ locally is given by the expression \begin{align}
		\Omega = \varphi(z_1,...,z_n) \bigwedge_{i=1}^n \frac{dz_i}{z_i},
		\label{omega}
\end{align}
where $(z_1,...,z_n)$ are local torus coordinates on $U^\an$  and $\varphi$ is an invertible function on $U$. 

\begin{defn}
	We  define $\Val(\Omega):= \Val(\varphi)$. By Lemma 2 in \cite{Kontsevich2006}, this is independent of a choice of coordinates.
\end{defn}


 As in \cite{Kontsevich2006}, we make the following assumption: 
 \begin{assumption}
 	$\Val(\Omega)$ is locally constant \ie $\dVal(\Omega) := \dVal(\varphi) =0$. 
 	\label{constantnorm}
 \end{assumption}
\begin{defn}
	Let $\Omega$ be a nowhere vanishing  analytic volume form on $X^\an$. Then define the \textit{residue} $\Res(\Omega) \in K$ of $\Omega$  to be the constant term of $\varphi$ defined in (\ref{omega}). 
\end{defn}

We make the following observation which will be needed in proving that the definition of residue is well defined.

\begin{lem}
	Let $\Omega = \varphi(z_1,...,z_n) \bigwedge_{i=1}^n \frac{dz_i}{z_i}$ with $\varphi = \sum_{I \in \Z^n}\varphi_I z^I$. Consider the volume form $$\Omega' = \Omega - c \bigwedge_{i=1}^n \frac{dz_i}{z_i}$$
	with $c\in K$. Then $\Omega'$ is exact if and only if $c = \varphi_{(0,...,0)}$. 
	\label{exact}
\end{lem}


\begin{lem}
	The residue is independent (up to sign) of the choice of local torus coordinates. 
	\label{indep}
\end{lem}

\begin{proof}
Let $(z_1,...z_n)$ and $(s_1,...,s_n)$ be two systems of local torus coordinates. Then for each $i$, we can write $$s_i = c_i z^{\mathbf{I}_i}(1+\alpha_i),$$
with $c_i \in K^\times$, $\mathbf{I}_i \in \Z^n$ and $\alpha_i = \alpha_i(z_1,...,z_n)$ such that: 
\begin{enumerate}[label = (\roman*)]
	\item the $n\times n$-matrix $ \mathbf{I} = (\mathbf{I}_1,...,\mathbf{I}_n) \in \GL_n(\Z)$ is unimodular;
	\item the function $\alpha_i$ has no degree zero term and $\Val(1+\alpha_i) = 0$.
\end{enumerate}
	It is easy to see that the scalars $c_i$ and the unimodular transformation $\mathbf{I}$ have no effect on the residue. Thus we can assume the coordinate transformation is of the form 
	\begin{align*}
		z_i \longmapsto s_i := z_i(1+\alpha_i).
	\end{align*}
Assume we can write the volume form $\Omega$ as 
\begin{align*}
	\Omega = \varphi(z_1,...,z_n) \bigwedge_{i=1}^n \frac{dz_i}{z_i} = \phi(s_1,...,s_n)\bigwedge \frac{ds_i}{s_i}
\end{align*}
with $\varphi = \sum_{I \in \Z^n}\varphi_I z^I$ and $\phi = \sum_{I \in \Z^n} \phi_I s^I$. By Lemma \ref{exact}, the volume form $\Omega - \varphi_{(0,...,0)} \bigwedge_{i=1}^n \frac{dz_i}{z_i}$ is exact. Define \begin{align*}
	u_i(z_1,\ldots,z_n) = \frac{s_i}{z_i} = 1+ \alpha_i.
\end{align*}
Since the dominant monomial is of degree 0, the (Laurent) series $\log (u_i)$ converges and the form $\frac{du_i}{u_i}$ is exact. Let $[n]$ denote the set  $\{1,...,n\}$. Then 
\begin{align*}
	\bigwedge_{i=1}^n \frac{ds_i}{s_i}  &= \bigwedge_{i=1}^n \frac{(1+\alpha_i)dz_i + z_i d\alpha_i}{z_i(1+\alpha_i)}\\ 
	&= \bigwedge_{i=1}^n \left(\frac{dz_i}{z_i} + \frac{du_j}{u_j}\right)= \sum_{I \cup J = [n], I \cap J = \emptyset}\bigwedge_{i \in I}\frac{dz_i}{z_i} \wedge \bigwedge_{j \in J}\frac{du_j}{u_j} \\ 
	&= \bigwedge_{i=1}^n \frac{dz_i}{z_i} + \sum_{\substack{I \cup J = [n], I \cap J = \emptyset \\ J \neq \emptyset}} \bigwedge_{i \in I}\frac{dz_i}{z_i} \wedge \bigwedge_{j \in J}\frac{du_j}{u_j} \\ 
	&= \bigwedge_{i=1}^n \frac{dz_i}{z_i} - \sum_{\substack{I \cup J = [n], I \cap J = \emptyset \\ J \neq \emptyset}} d\left(\log u_{j'}\bigwedge_{i \in I}\frac{dz_i}{z_i} \wedge \bigwedge_{j \in J\backslash {j'}}\frac{du_j}{u_j}\right).
\end{align*}
Thus the form $\bigwedge_{i=1}^n\frac{ds_i}{s_i} - \bigwedge_{i=1}^n \frac{dz_i}{z_i}$ is exact which implies that
\begin{align*}
	\Omega - \varphi_{(0,...,0)}\bigwedge_{i=1}^n \frac{ds_i}{s_i}
\end{align*}
is also exact. By another application of Lemma \ref{exact} using the coordinate system $(s_1,...,s_n)$, we have $\varphi_{(0,...,0)} = \phi_{(0,...,0)}$. This completes the proof. 
\end{proof}

\begin{lem}
	Let $\Omega=\varphi(z_1,...,z_n) \bigwedge_{i=1}^n \frac{dz_i}{z_i}$ be a non-vanishing analytic volume form on $X^\an$ such that $\Val(\Omega)$ is locally constant. Then $\Res(\Omega) \neq 0$. 
\end{lem} 
\begin{proof}
	Since $\Omega$ is non-vanishing, we have $\exp (-\Val(\Omega)) = \exp(-\Val(\varphi)) = |\varphi| = |\varphi_{(0,...,0)}|$ where in the last equality we use the constant norm assumption $\dVal(\Omega) = 0$. Since $\Omega$ is non-vanishing,  $\Res(\Omega) \neq 0$. 
\end{proof}

\begin{thm}\cite[Theorem 4]{Kontsevich2006}
	There exists a $K$-affine structure $\Aff_K$ on $B^\sm$ which is compatible with the integral affine structure $\Aff_\Z$ on $B^\sm$ \ie there is a diagram 
	\begin{center}
		\begin{tikzcd}
		0 \ar[r] & K^\times \ar[r]\ar[d,"v_K"] & \Aff_K \ar[r]\ar[d] & \check{\Lambda} \ar[r]\ar[d, "\text{id}"] & 0 \\ 
		0 \ar[r] & \R \ar[r] & \Aff_\Z \ar[r] & \check{\Lambda} \ar[r] & 0
	\end{tikzcd}

	\end{center}
	where the rows are exact. 
	
	\label{affkexp}
\end{thm}

\begin{proof}
We define the following multiplicative residue map
	\begin{align*}
		\Res_\Omega: \Fc_1 &\twoheadrightarrow\Oc_K^\times 
	\end{align*}
	where given an element $f \in \Fc_1$, after a choice of local torus coordinates, we can write $f = a(1+r)$ with $a \in \Oc_K^\times$ and $|r|<1$.  Then 
	\begin{align*}
		\Res_\Omega(f) = a \exp \left(\frac{\Res(\Omega\log(1+r))}{\Res(\Omega)}\right).
	\end{align*}
	Note $\Res_\Omega$ remains unchanged if we scale $\Omega$ to $c\Omega$ with $c \in K^\times$. 
	
\noindent \begin{claim}
	$\Res_\Omega$ is well defined \ie  is independent of the non-unique factorisation of $f$. 
\end{claim}
\begin{claimproof}
	Let $(z_1,...,z_n)$ and $(s_1,...,s_n)$ be two systems of local torus coordinates and suppose \begin{align*}
		a(1+r(z_1,...,z_n)) = b(1+ r'(s_1,...,s_n)).
	\end{align*}
Since $a,b \in \Oc_K^\times$, we can write $a = a'(1+\alpha)$ and $b = b'(1+\beta)$ where $a',b' \in \C^\times $ and $1+\alpha,1+\beta \in 1 + t\C[[t]]$. In particular, $\log(a(1+r))$ is well defined and given by $\log(a) + \log(1+\alpha) + \log(1+r)$ and similarly for $\log(b(1+r'))$. Then \begin{align*}
	\Res(\Omega\log(a(1+r))) = \Res(\Omega\log(b(1+r')))
\end{align*}
since $\Res$ is independent of the choice of local torus coordinates by Lemma \ref{indep}. We can write both sides of the equation as 
\begin{align*}
	&\log(a')\Res(\Omega) + \log(1+\alpha) \Res(\Omega) + \Res(\Omega\log(1+r))  \\&= \log(b')\Res(\Omega) + \log(1+\alpha) \Res(\Omega) + \Res(\Omega\log(1+r')).
\end{align*}
Dividing by $\Res(\Omega)\neq 0$ and exponentiating proves the claim. 
\end{claimproof}
		
		We return to the proof the proposition. We have a short exact sequence 
	\begin{align}
		0 \longrightarrow K^\times \longrightarrow \rho_*(\Oc_{X^\an}^\times)/\ker \Res_\Omega \longrightarrow \check{\Lambda} \longrightarrow 0
	\end{align}
	which defines a $K$-affine structure on $B^\sm$  compatible with the integral affine structure on $B^\sm$. Indeed, by construction, $\ker\Res_\Omega$ is a subsheaf of $\Fc_1$ whose elements restrict to 0 on $\Aff_\Z$ and we recover the usual integral affine structure by applying $\Val$ to $\rho_*(\Oc_{X^\an}^\times)/\ker \Res_\Omega$. 
\end{proof}

\begin{eg}
	As we noted in Example \ref{split}, the $K$-affine structure on $\R^n$ is split.   The splitting of the $K$-affine structure in Theorem \ref{affkexp} is given by the section $\check{\Lambda} \rightarrow \Aff_K$ which sends a character (where we identify $\check{\Lambda}$ with the character lattice) to the corresponding monomial in $\trop_*\Oc_{\G_m^{n,\an}}^\times$. 
	\label{AffKtorus}
\end{eg}

The following example will be very important when we begin to compute what the $K$-affine structure $\Aff_K$ from Theorem \ref{affkexp} looks like locally for log Calabi-Yau surfaces. 

\begin{eg}
Let $V \subseteq B^\sm$ be a connected open subset such that on $V$, $\rho$ locally looks like $\trop$. Let $\mu \in \Oc_K^\times$ and consider the element $\mu+ z \in \rho_*(\Oc_{X^\an}^\times)(V)$ for  $z$ a local torus coordinate on $\pi^{-1}(V)$ with $|z| < 1 $ on $V$.  Then $\Val(\mu + z) = 0$ and writing $\mu + z  = \mu\left(1+\frac{z}{\mu }\right)$ we see that $\Res_\Omega\left(1+\frac{z}{\mu }\right) = 1$. Thus $\mu +z \equiv \mu$ in $\rho_*(\Oc_{X^\an}^\times)/\ker \Res_\Omega$. \label{equiv}
\end{eg}

\begin{eg}
	Any element $f $ which can be written as a product of power series  will be an element of $\ker \Res_\Omega$. For example, 
\begin{align*}
	f = \frac{1}{1-tz}\cdot \frac{1}{1+tz^{-1}} \in \ker \Res_\Omega. 
\end{align*}

\end{eg}



\begin{notation}
	Throughout this paper, we will denote the explicit $K$-affine structure constructed in Proposition \ref{affkexp} by $\Aff_K$. 
\end{notation}

\subsection{\v{C}ech cohomology}
\label{cech}

A $K$-affine structure describes an element in $\text{Ext}^1(\check{\Lambda},K^\times) \cong H^1(B^\sm,\Lambda\otimes K^\times)$. 
Indeed, dualising and tensoring the short exact sequence defining the $K$-affine structure by $K^\times$, we have the short exact sequence 
\begin{align}
    \label{dualKaff}
    0 \rightarrow \Lambda \otimes K^\times  \rightarrow \mathcal{H}\text{om}(\Aff_K, K^\times) \rightarrow \text{Hom}(K^\times, K^\times)\rightarrow 0.
\end{align}
Taking long exact sequences,  we have \begin{align*}0 &\rightarrow H^0(B^\text{sm}, \Lambda\otimes K^\times) \rightarrow H^0(B^\text{sm},\mathcal{H}\text{om}(\Aff_K, K^\times)) \rightarrow H^0(B^\text{sm},\text{Hom}(K^\times,K^\times))  \\ &\xrightarrow{\ \ \delta\ \ } H^1(B^\text{sm}, \Lambda \otimes K^\times) \rightarrow H^1(B^\sm,\Hc\text{om}(\Aff_K,K^\times)) \longrightarrow \ldots \end{align*}
The corresponding element in $H^1(B^\sm, \Lambda\otimes K^\times)$ is the image, under the connecting homomorphism $\delta$, of the identity element in $\text{Hom}(K^\times, K^\times)$. 

\begin{rmk}
	We can also associate the cohomology class $[\Aff_\Z] \in H^1(B^\sm, \check{\Lambda}_\R)$ to the integral affine structure on $B^\sm$. This class, called the \textit{radiance obstruction class}, was introduced in \cite{goldmanhirsch}. The radiance obstruction class has been used recently in studying the notion of tropical periods \cite{tropicalperiod}, where an interpretation is given in terms of logarithmic Hodge theory. 
\end{rmk}

We will explicitly compute the $K$-affine structure for the case of log Calabi-Yau surfaces using \v{C}ech cohomology in \S\ref{localsing}. By \cite[\href{https://stacks.math.columbia.edu/tag/03F7}{Tag 03F7}]{stacks}, given an open cover $\mathcal{U}$ such that the connected components of every finite intersection are contractible, then \v{C}ech cohomology agrees with sheaf cohomology. 



We now describe how to compute the image of the  identity element in $\Hom(K^\times, K^\times) $ via a diagram chase.  We assume that the open cover $\mathcal{U}$ is refined enough such that on each open $U \in \mathcal{U}$, the local system $\Lambda|_U$ is constant.  The \v{C}ech double complex for the open cover $\mathcal{U}$ is as follows. 

\begin{figure}[H]
\fontsize{8}{9}
\centering
	\begin{tikzcd}[column sep = small]
 & 0\ar[d] & 0\ar[d] & 0\ar[d] & \\
 0 \ar[r] & \prod_{U \in \mathcal{U} } {(U,\Lambda\otimes K^\times)} \ar[r]\ar[d,"d_0"] & \prod_{U \in \mathcal{U} } (U,\mathcal{H}\text{om}(\text{Aff}_K,K^\times)) \ar[r,"\text{rest}"]\ar[d,"d_0'"] &  \prod_{U \in \mathcal{U} }(U,\text{Hom}(K^\times,K^\times)) \ar[r]\ar[d,"d_0''"] & 0 \\
 0 \ar[r] &  \prod_{U,U' \in \mathcal{U} } (U\cap U',\Lambda\otimes K^\times) \ar[r]\ar[d, "d_1"] &  \prod_{U, U' \in \mathcal{U}} \Gamma(U\cap U',\mathcal{H}\text{om}(\text{Aff}_K,K^\times)) \ar[r,]\ar[d, "d_1'"] &  \prod_{U,U'\in \mathcal{U}} \Gamma(U\cap U',\text{Hom}(K^\times,K^\times)) \ar[r]\ar[d, "d_1''"] & 0 \\
  & \dots  & \dots  & \dots & \\
	\end{tikzcd}
\end{figure}

Let $\alpha = (\text{id},,,\text{id}) \in \prod_{U}\Gamma\left(U,\Hom(K^\times,K^\times)\right)$ where $\text{id} : K^\times \rightarrow K^\times$ is the identity element in $\Hom(K^\times, K^\times)$.   A lift of the element $\text{id} \in \Gamma(U_{},\Hom(K^\times,K^\times))$ to $\Gamma(U_{},\mathcal{H}\text{om}(\Aff_K,K^\times))$ is given as follows:   for each $U$, we have $H^1(U, \Lambda \otimes K^\times) = 0$ since $U_{}$ is contractible and $\Lambda$ is a constant sheaf. In particular, the short exact sequence $$0 \longrightarrow K^\times|_{U}\longrightarrow \Aff_K|_{U}    \longrightarrow \check{\Lambda}|_{U} \longrightarrow 0$$ splits and there is a section $s_U : \Aff_K|_{U_{}}  \rightarrow K^\times$. By construction, this is a lift of $\text{id}: K^\times \rightarrow K^\times$ on $U$.

\subsection{Log Calabi-Yau surfaces}
\label{localsing}

In this section, we describe $\Aff_K$ associated to the non-archimedean SYZ fibration $\rho = \rho_{(\Yc,\Dc)}$ with $(\Yc,\Dc)$ the snc log model for $U = Y\backslash D$ from Proposition \ref{model} and compute the monodromy representation (\ref{monodromydef}) of $\Aff_K$; as a corollary, we will recover the monodromy of the integral affine structure (\ref{Zaffinecomp}). Recall the construction of $\Aff_K = \frac{\rho_*(\Oc_{U^\an}^\times)}{\ker \Res_\Omega}$ depends on a volume form $\Omega$. In the case of log Calabi-Yau surfaces, there is a distinguished volume form given by the pullback of the standard torus invariant volume form on $\overline{Y}$. It is immediate that $\Omega$ satisfies the constant norm assumption (\ref{constantnorm}). 

To compute the monodromy, we will utilise the proof of Proposition \ref{toricalong} to construct a basis of affinoid torus coordinates for each local toric chart and compute the composition of the transition functions for $\Aff_K$. 

We first introduce the setup and notation needed to carry out the computation. We fix an orientation on $D$ where the node $D_{i-1}\cap D_i$ is labelled $\infty_{i-1}$ on $D_{i-1}$ and $0_i$ on $D_i$.  This naturally induces an orientation on the skeleton $\Sk(U)$ and in particular on the boundary of each irreducible component in the special fibre.  Let $\tau_C$ be a 1-dimensional stratum  in $\Sk(U)$ corresponding to $C$ in $\Yc_k+ \Dc$. 

Let $X = Y_{j}^i$ be a non-toric irreducible component of $\Yc_k$ with boundary $\Delta_X = C_1+ C_2 + C_3 + C_4$ and denote the corresponding vertex $v_X = v_j^i$ in $\Sk(U)$. Suppose that the non-toric exceptional curve $E_X $ supported on $X$ corresponds to a non-toric exceptional curve $ E_j^i$ on $Y$ whose centre is on $\overline{D}_i$.  Without loss of generality, let $C_1$ denote the boundary divisor containing the centre of the non-toric blowup on $X$. Let  $H^{(m)}$ denote the prime component of the divisor $\Yc_k + \Dc$ containing $C_m$ which is not equal to $X$. The induced orientation on $X$ determines the divisors $X_0$ and $X_\infty$ for each $C_m$ as defined in \S \ref{constructfib}; we do not write the superscripts for $X_0$ and $X_\infty$ as it will be clear from context which rational curve we are working with.

 For each $C_m$,  we will give expressions for generators of
 
 \begin{align*}
 	\Lc_0^{(m)} &:= \Oc_{\widehat{\Yc}_{/C_m}}  (X_\infty -X_0)\\\Lc_h^{(m)}&:= \Oc_{\widehat{\Yc}_{/C_m}} (-H^{(m)} - b_hX_\infty) \\
 \end{align*} where $b_h = -(C_m\cdot H^{(m)})$ in terms of the pullback of torus coordinates on $ \G_{m,K}^2  \subset \overline{Y}$.  


 \begin{notation}

 Let  $s_{0}^{m}, s_h^{m}$ be generators of  $\Lc_0^{(m)}$ and $\Lc_h^{(m)}$ respectively and $(\Yc^{0},\Dc^{0})$ denote the Looijenga pair over $R$ with good reduction from Assumption \ref{assume} with $(\Yc^{0},\Dc^{0})\times_R K = (Y,D)$. 
Consider the fan $\overline{\Sigma}$ of the toric model $(\overline{Y},\overline{D})$ of $(Y,D)$ and suppose the divisor $\overline{D}_i$ corresponds to the ray $\overline{\rho}_i$ of $\overline{\Sigma}$; let $\mathbf{v}_i$ denote the primitive integral generator of $\overline{\rho}_i$.

Consider the pair of torus coordinates $(x,z):= \left(\chi^{\mathbf{v}_{i-1}^\vee},\chi^{\mathbf{v}_i^\vee}\right)$ on $Y$.   Suppose the zero locus of $x+\mu_j^i$ intersected with $\overline{D}_i$ is the centre of the non-toric blowup  with exceptional curve $E_j^i$. We note that the zero locus of $z$ in $Y$ contains all the non-toric exceptional curves arising from non-toric blowups with centre on $\overline{D}_i$. For this reason, we work with the following pair of coordinates: 

$$(x,y) := \left(x,\frac{z}{\prod_{j=1}^{k_i}(x+\mu_j^i)}\right).$$
By taking the formal completion, we are working away from the non-toric blow ups on $\overline{D}_i$ and thus do not need to consider them. 


\end{notation}
\subsection{Constructing the affinoid torus coordinates}
 \label{technical}
In this section, we will explicitly describe the generators of the line bundles $\Oc_{\widehat{\Yc}/C}(-X-X_
 \infty)$ in terms of the torus coordinates on  $(Y,D)$. The proceeding lemmas will carefully investigate how the birational modifications needed to construct the model $\Yc$ change the affinoid torus coordinates.

Let $(
\Yc^1,\Dc^1)$ denote the pair over $R$ given by blowing up the subvariety $\Dc_i^{0} \times_R \ \C$ in $\Yc^0$. Denote the new irreducible component by $Y'$. The divisor $X_\infty$ is given by the horizontal divisor corresponding to $D_{i+1}$.

By Proposition \ref{toricalong}, we know that $\widehat{\Yc^1}_{/C}$ is formally isomorphic to the formal toric scheme $\widehat{\Xc}^1_{/C_\mathbf{t}}$ whose fan is given in the proof of the proposition loc. cit.. 

\begin{lem}

 Under the formal isomorphism $\widehat{\Yc^1}_{/C} \cong \widehat{\Xc}^1_{/C_\mathbf{t}}$,  we have $s^1_{Y'} = \chi^{(0,1,0)} =  {y}$ is a generator of $\Oc_{\widehat{\Yc^1}_/C}(-Y'-b_{Y'}^1X_
 \infty)$ where $b_{Y'}^1 = -(C\cdot Y')_{\Yc^1}$. 
 \label{lemflopsection}
\end{lem}

\begin{proof}
   We can reduce to the case of a single non-toric blow up \ie $j=1$ and by Proposition \ref{toricalong} find a toric structure in a (formal) neigbourhood of $C$. Then $\Yc^1$ is given locally by the fan depicted in Figure \ref{fan1} with $u_0 = (1,0,0), u_Y= (0,0,1), u_{Y'} = (0,1,1)$, $u _\infty = (-1,b^1_{Y'},0)$ Thus $\chi^{u_{Y'}^\vee} = \chi^{(0,1,0)}$. 
   \begin{figure}[H]
       \centering

\tikzset{every picture/.style={line width=0.75pt}} 

\begin{tikzpicture}[x=0.75pt,y=0.75pt,yscale=-1,xscale=1]

\draw    (200,140) -- (200,80) ;
\draw    (200,140) -- (268,140) ;
\draw [shift={(270,140)}, rotate = 180] [fill={rgb, 255:red, 0; green, 0; blue, 0 }  ][line width=0.08]  [draw opacity=0] (12,-3) -- (0,0) -- (12,3) -- cycle    ;
\draw    (200,80) -- (268,80) ;
\draw [shift={(270,80)}, rotate = 180] [fill={rgb, 255:red, 0; green, 0; blue, 0 }  ][line width=0.08]  [draw opacity=0] (12,-3) -- (0,0) -- (12,3) -- cycle    ;
\draw    (141.9,120.63) -- (200,140) ;
\draw [shift={(140,120)}, rotate = 18.43] [fill={rgb, 255:red, 0; green, 0; blue, 0 }  ][line width=0.08]  [draw opacity=0] (12,-3) -- (0,0) -- (12,3) -- cycle    ;
\draw    (141.79,50.89) -- (200,80) ;
\draw [shift={(140,50)}, rotate = 26.57] [fill={rgb, 255:red, 0; green, 0; blue, 0 }  ][line width=0.08]  [draw opacity=0] (12,-3) -- (0,0) -- (12,3) -- cycle    ;

\end{tikzpicture}       \label{fan1}
       \caption{Local fan for $\Yc^1$ at height one slice}

   \end{figure}
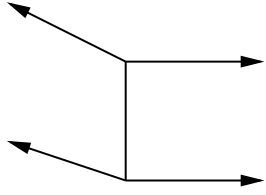
%
%
%
\end{proof}

Let $(\Yc^2,\Dc^2)$ denote the pair over obtained by applying Construction \ref{flopblow} (i) to all the exceptional curves with centre on $D_i$ \ie flop all the non-toric exceptional curves with centre on $\Dc_i \times_R\C$ to the newly produced irreducible component $Y'$ as depicted in Figure \ref{setupone}.
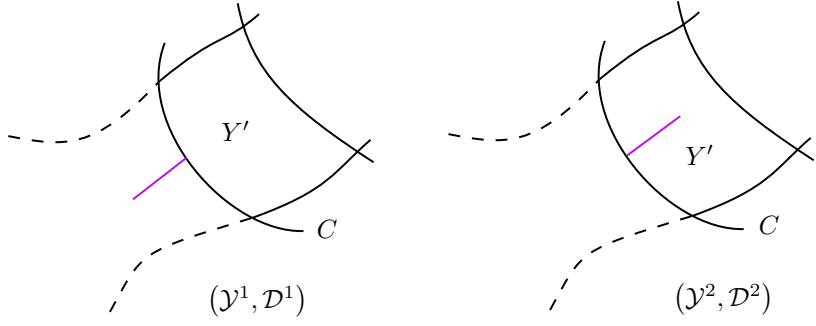
\begin{figure}[H]

\tikzset{every picture/.style={line width=0.75pt}} 

\begin{tikzpicture}[x=0.75pt,y=0.75pt,yscale=-1,xscale=1]

\draw    (198.79,72.52) .. controls (182.79,114.52) and (230.79,168.52) .. (268.79,167.52) ;
\draw [color={rgb, 255:red, 189; green, 16; blue, 224 }  ,draw opacity=1 ]   (209.79,130.52) -- (182.79,151.52) ;
\draw    (196,92) .. controls (226.79,66.52) and (227.86,74.04) .. (246.36,58.04) ;
\draw    (243.25,160.75) .. controls (282.09,146.71) and (284.09,137.46) .. (302.59,121.46) ;
\draw    (235.61,52.54) .. controls (242.11,87.04) and (261.68,105.1) .. (304.11,132.79) ;
\draw  [dash pattern={on 4.5pt off 4.5pt}]  (171.15,208.33) .. controls (188.15,177.33) and (184.25,178.75) .. (243.25,160.75) ;
\draw  [dash pattern={on 4.5pt off 4.5pt}]  (119.9,119.58) .. controls (148.68,125.1) and (163,127) .. (196,92) ;
\draw    (420.79,71.52) .. controls (404.79,113.52) and (452.79,167.52) .. (490.79,166.52) ;
\draw [color={rgb, 255:red, 189; green, 16; blue, 224 }  ,draw opacity=1 ]   (431.79,129.52) -- (458.79,109.52) ;
\draw    (418,91) .. controls (448.79,65.52) and (449.86,73.04) .. (468.36,57.04) ;
\draw    (465.25,159.75) .. controls (504.09,145.71) and (506.09,136.46) .. (524.59,120.46) ;
\draw    (457.61,51.54) .. controls (464.11,86.04) and (483.68,104.1) .. (526.11,131.79) ;
\draw  [dash pattern={on 4.5pt off 4.5pt}]  (393.15,207.33) .. controls (410.15,176.33) and (406.25,177.75) .. (465.25,159.75) ;
\draw  [dash pattern={on 4.5pt off 4.5pt}]  (341.9,118.58) .. controls (370.68,124.1) and (385,126) .. (418,91) ;

\draw (226,110.4) node [anchor=north west][inner sep=0.75pt]    {$Y'$};
\draw (460,122.4) node [anchor=north west][inner sep=0.75pt]    {$Y'$};
\draw (212,193.4) node [anchor=north west][inner sep=0.75pt]    {$ \begin{array}{l}
\left(\mathcal{Y}^{1} ,\mathcal{D}^{1}\right)\\
\end{array}$};
\draw (447,192.4) node [anchor=north west][inner sep=0.75pt]    {$ \begin{array}{l}
\left(\mathcal{Y}^{2} ,\mathcal{D}^{2}\right)\\
\end{array}$};
\draw (274,160.4) node [anchor=north west][inner sep=0.75pt]    {$C$};
\draw (497,158.4) node [anchor=north west][inner sep=0.75pt]    {$C$};

\end{tikzpicture}

\label{setupone}
\end{figure}

\begin{lem}

    Under the formal isomorphism $\widehat{\Yc^2}_{/C} \cong \widehat{\Xc}^2_{/C_\mathbf{t}}$, we have $s^2_{Y'} = \chi^{(0,-1,1)} $ is a generator of $\Oc_{\widehat{\Yc}^2/C}(-Y'-b_{Y'}^2X_\infty)$  where $b_{Y'}^2 = -(C\cdot Y')_{\Yc^2}$.
\end{lem}

\begin{proof}
    The flop is given by the local fan 

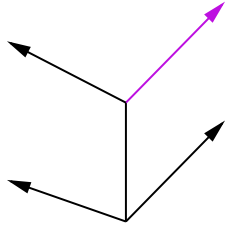
\begin{figure}[H]
    \centering

\tikzset{every picture/.style={line width=0.75pt}} 

\tikzset{every picture/.style={line width=0.75pt}} 

\begin{tikzpicture}[x=0.75pt,y=0.75pt,yscale=-1,xscale=1]

\draw    (220,171) -- (220,111) ;
\draw    (220,171) -- (268.6,121.43) ;
\draw [shift={(270,120)}, rotate = 134.43] [fill={rgb, 255:red, 0; green, 0; blue, 0 }  ][line width=0.08]  [draw opacity=0] (12,-3) -- (0,0) -- (12,3) -- cycle    ;
\draw    (161.89,150.66) -- (220,171) ;
\draw [shift={(160,150)}, rotate = 19.29] [fill={rgb, 255:red, 0; green, 0; blue, 0 }  ][line width=0.08]  [draw opacity=0] (12,-3) -- (0,0) -- (12,3) -- cycle    ;
\draw [color={rgb, 255:red, 189; green, 16; blue, 224 }  ,draw opacity=1 ]   (268.6,61.43) -- (220,111) ;
\draw [shift={(270,60)}, rotate = 134.43] [fill={rgb, 255:red, 189; green, 16; blue, 224 }  ,fill opacity=1 ][line width=0.08]  [draw opacity=0] (12,-3) -- (0,0) -- (12,3) -- cycle    ;
\draw    (161.78,80.92) -- (220,111) ;
\draw [shift={(160,80)}, rotate = 27.32] [fill={rgb, 255:red, 0; green, 0; blue, 0 }  ][line width=0.08]  [draw opacity=0] (12,-3) -- (0,0) -- (12,3) -- cycle    ;

\end{tikzpicture}
    \caption{Local fan for $\Yc^2$ at height 1 slice}
    \label{fig:enter-label}
\end{figure}\noindent
and can be described as in Proposition \ref{toricalong} by $u_0 = (1,1,0), u_Y = (0,0,1), u_{Y'} = (0,1,1)$, $u_\infty = (-1,b_{Y'}^2,0)$. 
Then the local equation for $Y'$ is given by $\chi^{u_{Y'}^\vee} = \chi^{(-1,1,0)}$.  
\end{proof}
\noindent 
Note also have the relation that 
\begin{align}
    s^2_{Y'}= \chi^{(-1, 1,0)} = \chi^{(-1,0,0)} \chi^{(0,1,0)}  = s_{X_\infty}^2 s_{Y'}^1 = \frac{s_{Y'}^1}{x+\mu_1^i}, 
    \label{chaineq}
\end{align}

\noindent since $\chi^{(1,0,0)} = x+\mu_1^i$. By induction on the number of non-toric blow ups supported on $D_i$ we have the following equality. 

\begin{cor}
    We have the relation \begin{align}
 s^2_{Y'}\prod_j (x+\mu_j^i) = s^1_{Y'}
\end{align}
\label{firstequal}
\end{cor}

We next deal with the case when we want to move all but one non-toric exceptional curve off an irreducible component. Let $(\Yc^3, \Dc^3)$ denote the pair over $R$ produced by blowing up $\Dc_i^3\times_R\C$ in $\Yc^2$. 

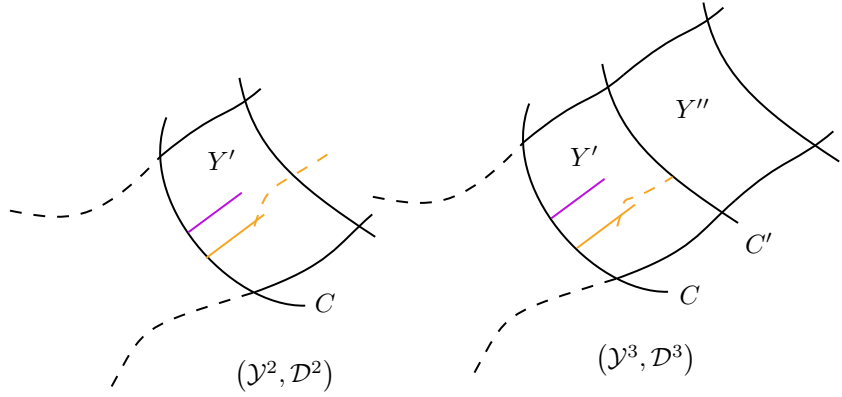
\begin{figure}[H]

\tikzset{every picture/.style={line width=0.75pt}} 

\begin{tikzpicture}[x=0.75pt,y=0.75pt,yscale=-1,xscale=1]

\draw    (178.42,88.52) .. controls (162.42,130.52) and (210.42,184.52) .. (248.42,183.52) ;
\draw [color={rgb, 255:red, 189; green, 16; blue, 224 }  ,draw opacity=1 ]   (189.42,146.52) -- (216.42,126.52) ;
\draw    (175.63,108) .. controls (206.42,82.52) and (207.49,90.04) .. (225.99,74.04) ;
\draw    (222.88,176.75) .. controls (261.72,162.71) and (263.72,153.46) .. (282.22,137.46) ;
\draw    (215.24,68.54) .. controls (221.74,103.04) and (241.31,121.1) .. (283.74,148.79) ;
\draw  [dash pattern={on 4.5pt off 4.5pt}]  (150.78,224.33) .. controls (167.78,193.33) and (163.88,194.75) .. (222.88,176.75) ;
\draw [color={rgb, 255:red, 245; green, 166; blue, 35 }  ,draw opacity=1 ]   (198.42,159.52) -- (227.79,137.52) ;
\draw [color={rgb, 255:red, 245; green, 166; blue, 35 }  ,draw opacity=1 ] [dash pattern={on 4.5pt off 4.5pt}]  (222.79,143.52) .. controls (229.79,120.52) and (229.79,126.52) .. (260.79,106.52) ;
\draw    (361.42,81.52) .. controls (345.42,123.52) and (393.42,177.52) .. (431.42,176.52) ;
\draw [color={rgb, 255:red, 189; green, 16; blue, 224 }  ,draw opacity=1 ]   (372.42,139.52) -- (399.42,119.52) ;
\draw    (358.63,101) .. controls (389.42,75.52) and (390.49,83.04) .. (408.99,67.04) ;
\draw    (405.88,169.75) .. controls (444.72,155.71) and (446.72,146.46) .. (465.22,130.46) ;
\draw    (398.24,61.54) .. controls (404.74,96.04) and (424.31,114.1) .. (466.74,141.79) ;
\draw  [dash pattern={on 4.5pt off 4.5pt}]  (333.78,217.33) .. controls (350.78,186.33) and (346.88,187.75) .. (405.88,169.75) ;
\draw  [dash pattern={on 4.5pt off 4.5pt}]  (282.53,128.58) .. controls (311.31,134.1) and (325.63,136) .. (358.63,101) ;
\draw    (408.99,67.04) .. controls (439.77,41.57) and (440.85,49.09) .. (459.35,33.09) ;
\draw    (465.22,130.46) .. controls (489.79,108.52) and (496.29,111.52) .. (514.79,95.52) ;
\draw    (449.24,30.54) .. controls (455.74,65.04) and (475.31,83.1) .. (517.74,110.79) ;
\draw [color={rgb, 255:red, 245; green, 166; blue, 35 }  ,draw opacity=1 ]   (385.42,154.52) -- (414.79,132.52) ;
\draw [color={rgb, 255:red, 245; green, 166; blue, 35 }  ,draw opacity=1 ] [dash pattern={on 4.5pt off 4.5pt}]  (405.79,143.52) .. controls (412.79,120.52) and (402.79,138.52) .. (433.79,118.52) ;
\draw  [dash pattern={on 4.5pt off 4.5pt}]  (99.53,135.58) .. controls (128.31,141.1) and (142.63,143) .. (175.63,108) ;

\draw (197.63,103.4) node [anchor=north west][inner sep=0.75pt]    {$Y'$};
\draw (204.63,209.4) node [anchor=north west][inner sep=0.75pt]    {$ \begin{array}{l}
\left(\mathcal{Y}^{2} ,\mathcal{D}^{2}\right)\\
\end{array}$};
\draw (380.63,101.4) node [anchor=north west][inner sep=0.75pt]    {$Y'$};
\draw (386.63,201.4) node [anchor=north west][inner sep=0.75pt]    {$ \begin{array}{l}
\left(\mathcal{Y}^{3} ,\mathcal{D}^{3}\right)\\
\end{array}$};
\draw (433.63,77.4) node [anchor=north west][inner sep=0.75pt]    {$Y''$};
\draw (468.74,145.19) node [anchor=north west][inner sep=0.75pt]    {$C'$};
\draw (435.74,172.19) node [anchor=north west][inner sep=0.75pt]    {$C$};
\draw (251.74,176.19) node [anchor=north west][inner sep=0.75pt]    {$C$};

\end{tikzpicture}
\label{setup3}
\end{figure}

\begin{lem}
    We have $s^3_{Y'} = s^2_{Y'}$ and $s^3_{Y''} =\frac{y}{t}$ where $s^3_{Y'} $ is a generator of $\Oc_{\widehat{\Yc^3}/C}(-Y'-b_{Y'}^3X_\infty)$, $s_{Y''}^3$ is a generator of $\Oc_{\widehat{\Yc^3}/C'}(-Y''-b_{Y''}^3X_\infty)$ and $b_{Y'}^3 = -(C\cdot Y')_{\Yc^3}$, $b_{Y''}^3 = -(C'\cdot Y'')_{\Yc^3}$. 
    \label{lemflopsection3}
\end{lem}

\begin{proof}
    The height one slice local fan for $\Yc^3$ is depicted below:

    \begin{figure}[H]
        \centering

\tikzset{every picture/.style={line width=0.75pt}} 

\begin{tikzpicture}[x=0.75pt,y=0.75pt,yscale=-1,xscale=1]

\draw    (220,260) -- (220,200) ;
\draw    (220,260) -- (278,259.03) ;
\draw [shift={(280,259)}, rotate = 179.05] [fill={rgb, 255:red, 0; green, 0; blue, 0 }  ][line width=0.08]  [draw opacity=0] (12,-3) -- (0,0) -- (12,3) -- cycle    ;
\draw    (220,200) -- (220,140) ;
\draw [color={rgb, 255:red, 245; green, 166; blue, 35 }  ,draw opacity=1 ] [dash pattern={on 4.5pt off 4.5pt}]  (220,200) -- (268.6,150.43) ;
\draw [shift={(270,149)}, rotate = 134.43] [fill={rgb, 255:red, 245; green, 166; blue, 35 }  ,fill opacity=1 ][line width=0.08]  [draw opacity=0] (12,-3) -- (0,0) -- (12,3) -- cycle    ;
\draw    (161.89,239.66) -- (220,260) ;
\draw [shift={(160,239)}, rotate = 19.29] [fill={rgb, 255:red, 0; green, 0; blue, 0 }  ][line width=0.08]  [draw opacity=0] (12,-3) -- (0,0) -- (12,3) -- cycle    ;
\draw [color={rgb, 255:red, 245; green, 166; blue, 35 }  ,draw opacity=1 ]   (278,200) -- (220,200) ;
\draw [shift={(280,200)}, rotate = 180] [fill={rgb, 255:red, 245; green, 166; blue, 35 }  ,fill opacity=1 ][line width=0.08]  [draw opacity=0] (12,-3) -- (0,0) -- (12,3) -- cycle    ;
\draw [color={rgb, 255:red, 189; green, 16; blue, 224 }  ,draw opacity=1 ]   (161.78,169.92) -- (220,200) ;
\draw [shift={(160,169)}, rotate = 27.32] [fill={rgb, 255:red, 189; green, 16; blue, 224 }  ,fill opacity=1 ][line width=0.08]  [draw opacity=0] (12,-3) -- (0,0) -- (12,3) -- cycle    ;
\draw    (268.59,91.41) -- (220,140) ;
\draw [shift={(270,90)}, rotate = 135] [fill={rgb, 255:red, 0; green, 0; blue, 0 }  ][line width=0.08]  [draw opacity=0] (12,-3) -- (0,0) -- (12,3) -- cycle    ;
\draw    (161.78,109.92) -- (220,140) ;
\draw [shift={(160,109)}, rotate = 27.32] [fill={rgb, 255:red, 0; green, 0; blue, 0 }  ][line width=0.08]  [draw opacity=0] (12,-3) -- (0,0) -- (12,3) -- cycle    ;
\draw    (220,140) -- (220,82) ;
\draw [shift={(220,80)}, rotate = 90] [fill={rgb, 255:red, 0; green, 0; blue, 0 }  ][line width=0.08]  [draw opacity=0] (12,-3) -- (0,0) -- (12,3) -- cycle    ;

\end{tikzpicture}
  \caption{Local fan for $\Yc^3$ at height one slice}
        \label{fig:enter-label}
    \end{figure}
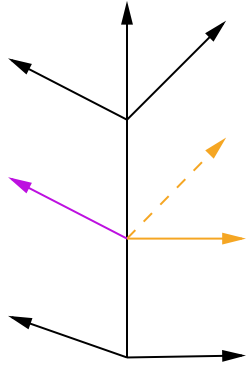 

Since the modification is away from $C$, we have that $s^3_{Y'} = s^2_{Y'} = \chi^{(-1,1,0)}$.  To compute $s_{Y''}^3$, the local toric fan at height 1 is depicted below where $u_0' = (1,1,0), u_{Y'}' = (0,1,1), u_{Y''}' = (0,2,1)$. Then $\chi^{u_{Y''}^\vee} = \chi^{(-1,1,-1)}  = \frac{s^2_{Y'}}{t} $ where the last equality follows from Equation \ref{chaineq}.

    
\end{proof}

 Let $B_j \subset Y'$ denote the strict transform of the fibre of $Y'$ which passes though $\mu_j^i$. Let $({\Yc}^4,{\Dc}^4)$ denote the  pair over $R$ produced by running Construction \ref{flopblow} (ii) to all but one $B_j$ \ie  flop all but one $B_j$, say $B_1$, to the new irreducible component $Y''$.

 \begin{figure}[H]
     \centering

\tikzset{every picture/.style={line width=0.75pt}} 

\begin{tikzpicture}[x=0.75pt,y=0.75pt,yscale=-1,xscale=1]

\draw    (458.42,85.52) .. controls (442.42,127.52) and (490.42,181.52) .. (528.42,180.52) ;
\draw [color={rgb, 255:red, 189; green, 16; blue, 224 }  ,draw opacity=1 ]   (469.42,143.52) -- (496.42,123.52) ;
\draw    (455.63,105) .. controls (486.42,79.52) and (487.49,87.04) .. (505.99,71.04) ;
\draw    (502.88,173.75) .. controls (541.72,159.71) and (543.72,150.46) .. (562.22,134.46) ;
\draw    (495.24,65.54) .. controls (501.74,100.04) and (521.31,118.1) .. (563.74,145.79) ;
\draw  [dash pattern={on 4.5pt off 4.5pt}]  (430.78,221.33) .. controls (447.78,190.33) and (443.88,191.75) .. (502.88,173.75) ;
\draw  [dash pattern={on 4.5pt off 4.5pt}]  (379.53,132.58) .. controls (408.31,138.1) and (422.63,140) .. (455.63,105) ;
\draw [color={rgb, 255:red, 245; green, 166; blue, 35 }  ,draw opacity=1 ]   (531.42,123.52) -- (558.42,103.52) ;
\draw    (505.99,71.04) .. controls (536.77,45.57) and (537.85,53.09) .. (556.35,37.09) ;
\draw    (562.22,134.46) .. controls (586.79,112.52) and (593.29,115.52) .. (611.79,99.52) ;
\draw    (546.24,34.54) .. controls (552.74,69.04) and (572.31,87.1) .. (614.74,114.79) ;
\draw    (256.42,87.52) .. controls (240.42,129.52) and (288.42,183.52) .. (326.42,182.52) ;
\draw [color={rgb, 255:red, 189; green, 16; blue, 224 }  ,draw opacity=1 ]   (267.42,145.52) -- (294.42,125.52) ;
\draw    (253.63,107) .. controls (284.42,81.52) and (285.49,89.04) .. (303.99,73.04) ;
\draw    (300.88,175.75) .. controls (339.72,161.71) and (341.72,152.46) .. (360.22,136.46) ;
\draw    (293.24,67.54) .. controls (299.74,102.04) and (319.31,120.1) .. (361.74,147.79) ;
\draw  [dash pattern={on 4.5pt off 4.5pt}]  (228.78,223.33) .. controls (245.78,192.33) and (241.88,193.75) .. (300.88,175.75) ;
\draw  [dash pattern={on 4.5pt off 4.5pt}]  (177.53,134.58) .. controls (206.31,140.1) and (220.63,142) .. (253.63,107) ;
\draw    (303.99,73.04) .. controls (334.77,47.57) and (335.85,55.09) .. (354.35,39.09) ;
\draw    (360.22,136.46) .. controls (384.79,114.52) and (391.29,117.52) .. (409.79,101.52) ;
\draw    (344.24,36.54) .. controls (350.74,71.04) and (370.31,89.1) .. (412.74,116.79) ;
\draw [color={rgb, 255:red, 245; green, 166; blue, 35 }  ,draw opacity=1 ]   (280.42,160.52) -- (309.79,138.52) ;
\draw [color={rgb, 255:red, 245; green, 166; blue, 35 }  ,draw opacity=1 ] [dash pattern={on 4.5pt off 4.5pt}]  (300.79,149.52) .. controls (307.79,126.52) and (297.79,144.52) .. (328.79,124.52) ;

\draw (497.63,135.4) node [anchor=north west][inner sep=0.75pt]    {$Y'$};
\draw (530.63,81.4) node [anchor=north west][inner sep=0.75pt]    {$Y''$};
\draw (275.63,107.4) node [anchor=north west][inner sep=0.75pt]    {$Y'$};
\draw (281.63,207.4) node [anchor=north west][inner sep=0.75pt]    {$ \begin{array}{l}
\left(\mathcal{Y}^{3} ,\mathcal{D}^{3}\right)\\
\end{array}$};
\draw (328.63,83.4) node [anchor=north west][inner sep=0.75pt]    {$Y''$};
\draw (363.74,151.19) node [anchor=north west][inner sep=0.75pt]    {$C'$};
\draw (565.74,149.19) node [anchor=north west][inner sep=0.75pt]    {$C'$};
\draw (330.74,178.19) node [anchor=north west][inner sep=0.75pt]    {$C$};
\draw (533.74,174.19) node [anchor=north west][inner sep=0.75pt]    {$C$};
\draw (476.63,204.4) node [anchor=north west][inner sep=0.75pt]    {$ \begin{array}{l}
\left(\mathcal{Y}^{4} ,\mathcal{D}^{4}\right)\\
\end{array}$};

\end{tikzpicture}
     \caption{Caption}
     \label{fig:enter-label}
 \end{figure}

 \begin{lem}
     Under the formal isomorphism $\widehat{\Yc^4}_{/C}\cong \Xc^4_{/C_\mathbf{t}}$ we have $s_{Y'}^4 = \chi^{(0,1,0)}$, $s_{Y''}^4  = \chi^{(0,1,-1)}$ where $s^4_{Y'} $ is a generator of $\Oc_{\widehat{\Yc^4}/C}(-Y'-b_{Y'}^4X_\infty)$, $s_{Y''}^4$ is a generator of $\Oc_{\widehat{\Yc^4}/C'}(-Y''-b_{Y''}^4X_\infty)$ and $b_{Y'}^4 = -(C\cdot Y')_{\Yc^4}$, $b_{Y''}^4 = -(C\cdot Y'')_{\Yc^4}$. 
 \end{lem}

 \begin{proof}
 The height one local fan for the flop  is   
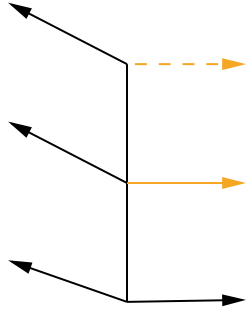
\begin{figure}[H]
    \centering

\tikzset{every picture/.style={line width=0.75pt}} 

\begin{tikzpicture}[x=0.75pt,y=0.75pt,yscale=-1,xscale=1]

\draw    (220,260) -- (220,200) ;
\draw    (220,260) -- (278,259.03) ;
\draw [shift={(280,259)}, rotate = 179.05] [fill={rgb, 255:red, 0; green, 0; blue, 0 }  ][line width=0.08]  [draw opacity=0] (12,-3) -- (0,0) -- (12,3) -- cycle    ;
\draw    (220,200) -- (220,140) ;
\draw    (161.89,239.66) -- (220,260) ;
\draw [shift={(160,239)}, rotate = 19.29] [fill={rgb, 255:red, 0; green, 0; blue, 0 }  ][line width=0.08]  [draw opacity=0] (12,-3) -- (0,0) -- (12,3) -- cycle    ;
\draw [color={rgb, 255:red, 245; green, 166; blue, 35 }  ,draw opacity=1 ]   (278,200) -- (220,200) ;
\draw [shift={(280,200)}, rotate = 180] [fill={rgb, 255:red, 245; green, 166; blue, 35 }  ,fill opacity=1 ][line width=0.08]  [draw opacity=0] (12,-3) -- (0,0) -- (12,3) -- cycle    ;
\draw [color={rgb, 255:red, 0; green, 0; blue, 0 }  ,draw opacity=1 ]   (161.78,169.92) -- (220,200) ;
\draw [shift={(160,169)}, rotate = 27.32] [fill={rgb, 255:red, 0; green, 0; blue, 0 }  ,fill opacity=1 ][line width=0.08]  [draw opacity=0] (12,-3) -- (0,0) -- (12,3) -- cycle    ;
\draw [color={rgb, 255:red, 245; green, 166; blue, 35 }  ,draw opacity=1 ] [dash pattern={on 4.5pt off 4.5pt}]  (278,140) -- (220,140) ;
\draw [shift={(280,140)}, rotate = 180] [fill={rgb, 255:red, 245; green, 166; blue, 35 }  ,fill opacity=1 ][line width=0.08]  [draw opacity=0] (12,-3) -- (0,0) -- (12,3) -- cycle    ;
\draw    (161.78,109.92) -- (220,140) ;
\draw [shift={(160,109)}, rotate = 27.32] [fill={rgb, 255:red, 0; green, 0; blue, 0 }  ][line width=0.08]  [draw opacity=0] (12,-3) -- (0,0) -- (12,3) -- cycle    ;

\end{tikzpicture} \caption{Local fan for $\Yc^4$}
    \label{fig:enter-label}
\end{figure}   




We first compute $s_{Y'}^4$. We have $u_0 = (1,0,0), u_Y = (0,0,1), u_{Y'} = (0,1,1)$ and thus $s_{Y'}^4 = \chi^{u_{Y'}^\vee}= \chi^{(0,1,0)}$. To compute $s_{Y''}^4$, we have $u_0 =(1,0,0), u_{Y'}=(0,1,1)$ and $u_{Y''} = (0,2,1)$. Thus $s^{4}_{Y''} = \chi^{u_{Y''}^\vee} = \chi^{(0,1,-1)}$.

 \end{proof}

 \noindent We then have the relations 
 \begin{align*}
 	s^3_{Y'}(x+\mu_1^i) = s^4_{Y'} \text{ and } s_{Y''}^3 (x+\mu_1^i) = s^4_{Y''}. 
 \end{align*}

 \begin{cor}
 \label{relation2}
     We have the relation
     
     \begin{align}
 	s^1_{Y'} = s^4_{Y'}\prod_{j\neq 1}(x+\mu_j^i). 
 \end{align}
     
\end{cor}

 \noindent Since $s_{Y'}^1 = y$, we can write our affinoid torus coordinates in terms of the coordinates on the generic fiber as 
 
$$s_{Y'}^4 = \frac{y}{(x+\mu_1^i)}$$ and $$
 s_{Y''}^4 = \frac{y}{t\prod_{j\neq1 }(x+\mu_j^i)}. $$

\subsection{Constructing the affinoid torus coordinates}


We can repeat the argument above to construct the sections for the other irreducible components inductively. Indeed, we have that away from $X_\infty$,  
 \begin{align*}
 	s_{X_j}^{} = \frac{y}{t^{j+1}}(x+\mu_j^{i})
 \end{align*}
 where $s_{X_j} \in \Oc_{\widehat{\Yc}/C_1}(-X_j -b_{X_j}X_\infty)$ is a generator.

We can put all of this together and now describe the sections for each 1-dimensional stratum emanating from the singular vertices of $\Sk(U)$. We fix the vertex $v_{X_j} \in \rho_i$. Recall $s_{0}^{m}, s_h^{m}$ are generators of  $\Lc_0^{(m)}$ and $\Lc_h^{(m)}$ respectively and let $s_v^{(m)}$ be a generator of the line bundle $\Oc_{\widehat{\Yc/C_m}}(-X_j -b_{X_j}X_\infty)$; note we can choose $s_v^{(m)} = s_{X_j}$ for $m=1$ and for  $m=1,3$ we have the relation $s_{v}^{(m)} s_h^{(m)} = t$. Similarly, for $m=2,4$, we have the relation $s_0^{(m)} s_v^{(m)} = t $. Recall that $\mathscr{Y}_0 = \mathscr{Y}\backslash \{c_\infty\}$ where $\mathscr{Y} = \widehat{\Yc}_{/C_m} $. Then we have: 
\begin{enumerate}[label = (\roman*)]
	\item along the formal completion of $C_1$:
\begin{align*}
	(a_1,b_1) := ( s_0^{(1)},s_h^{(1)}|_{\mathscr{Y}_0}) = \left(x, \frac{t^{j+1}}{y(x+\mu_j^{i})}\right);
\end{align*}
\item along the formal completion of $C_2$: 
\begin{align*}
	(a_2,b_2) := (s_0^{(2)}|_{\mathscr{Y}_0}, s_h^{(2)}|_{\mathscr{Y}_0}) = \left(\frac{t^{j+1}}{y}, \frac{1}{x} \right);
\end{align*}
\item along the formal completion of $C_3$: 
\begin{align*}
	(a_3,b_3) := (s_0^{(3)},s_h^{(3)}|_{\mathscr{Y}_0}) = \left(\frac{1}{x},\frac{y(1+\frac{\mu_{j+1}^{i}}{x})}{t^{j+1}}\right),
\end{align*}
\item along the formal completion of $C_4$: 
\begin{align*}
		(a_4,b_4) := (  s_0^{(4)}|_{\mathscr{Y}_0}, s_h^{(4)}|_{\mathscr{Y}_0})  = \left( \frac{y}{t^{j+1}},x\right).
\end{align*}

\end{enumerate}

\subsection{Monodromy of $K$-affine structure}

With the expressions above, we can now compute the monodromy of $\Aff_K$ for the log Calabi-Yau surface $U = Y\backslash D$. 
\begin{prop}
The monodromy of the $K$-affine structure around $v_{X_j}$ with respect to the affinoid torus coordinates $(a_2,b_2) = (s_0^{(2)},s_{h}^{(2)})$ is 
	\begin{align*}
		T_{\Aff_K} : \Aff_K(\Star(\tau_{C_2})) &\longrightarrow \Aff_K(\Star(\tau_{C_2})) \\
		a_2 &\longmapsto  a_2/\mu_j^{i}\\ 
		b_2 &\longmapsto a_2b_2.
	\end{align*}
	\label{monodromy}
\end{prop}

\begin{proof}
Let $\gamma$ be a loop around $v_{X_j}$ connecting $v_{C_1}, v_{C_2}, v_{C_3}, v_{C_4}$ and $v_{C_1}$ in that order.
As in the integral affine structure computation, we take our charts to be given by $S_m :=  \Star(\tau_{C_m})$. 

The transition map $\varphi_{m,m+1}$  for $\rho^{-1}_{(\mathcal{Y},\mathcal{D})}(S_m\cap S_{m+1})$ is
\begin{center}
\begin{tikzcd}[row sep = tiny]
{\rho^{-1}_{(\mathcal{Y},\mathcal{D})}(S_m)}                        &  & {\rho^{-1}_{(\mathcal{Y},\mathcal{D})}(S_{m+1})}         \\
\cup                                                               &  & \cup                                                    \\
{\rho^{-1}_{(\mathcal{Y},\mathcal{D})}(S_m\cap S_{m+1})} \arrow[rr, "\varphi_{i,i+1}"] &  & {\rho^{-1}_{(\mathcal{Y},\mathcal{D})}(S_m\cap S_{m+1})}\\ 
x_m \arrow[rr, maps to] & & (x_m)^{p_{m}} y_m^{-1}\\
y_m\arrow[rr, maps to] & & x_m\\
\end{tikzcd}
\end{center}
where $x_m$ is a torus coordinate corresponding to $\chi^{v_{C_{m}}^\vee}$ and $y_m$ a torus coordinate corresponding to $\chi^{v_{C_{m+1}}^\vee}$ with $p_{m+1} = -(C_{m+1}\cdot H^{(m+1)} ) $. By the geometry of the model $(\Yc,\Dc)$, we know $p_2 = p_4 = 0$ and $p_1 + p_3 = -1$. 

The volume form $\Omega = \frac{dx \wedge dz}{xz}$ when restricted to the toric charts $S_m$ is equal to 
\begin{align*}
	\Omega = -\frac{da_i \wedge db_i}{a_ib_i}. 
\end{align*}
We can the then simplify the expressions: 
\begin{align*}
	b_1 \equiv \frac{t^{j+1}}{\mu_j^{i} y} \text{ and }  b_3 \equiv \frac{y} {t^{j+1}}\text{ in } \Aff_K, 
\end{align*}
since \begin{align*}
	\Res_\Omega\left(1+ \frac{a_1}{\mu_j}\right) &= \exp\left(\frac{\Res\left(\Omega \log\left(1+\frac{a_1}{\mu_j}\right)\right)}{\Res(\Omega)}\right) \\
	&= \exp\left({\Res\left(-\frac{a_1\wedge db_1}{a_1b_1}  \mathlarger{\sum}_{m=1}^\infty\left(-\frac{a_1}{\mu_j}\right)^{m+1}\right)}\right) \\ 
	&= 1
\end{align*}
\noindent and a similar argument follows for $b_3$. 

Starting in the chart $S_2$ and following $\gamma$, we track each coordinate change  below: 
\begin{align*}
	(a_2,b_2) &\xmapsto{\varphi_{2,3}} \left(\frac{1}{b_2}, \frac{1}{a_3}\right) = \left(\frac{1}{a_3}, b_3\right) \\ 
	&\xmapsto{\varphi_{3,4}} \left(b_3,a_3b_3^{p_3}\right) = \left(a_4, \frac{a_4^{p_3}}{b_4}\right) \\ 
	&\xmapsto{\varphi_{4,1}} \left(\frac{1}{b_4}, \frac{1}{a_4b_4^{p_3}}\right) = \left(\frac{1}{a_1}, \frac{\mu_j^{i}b_1}{a_1^{p_3}}\right) \\
	&\xmapsto{\varphi_{1,2}} \left(b_1, \mu^{p_1+p_3}a_1b_1^{p_1+p_3}\right)  = \left(b_1, \frac{a_1}{\mu_j^{i}b_1}\right) = \left(\frac{a_2}{\mu_j^{i}}, a_2b_2\right).
\end{align*}

Thus the monodromy of the $K$-affine structure is given by $a_2 \mapsto a_2/\mu_j^{i} $ and $b_2 \mapsto a_2b_2$ as claimed. 
\end{proof}

By construction, the coordinates $(a_2,b_2)$ correspond to $(\chi^{v_{C_1}^\vee },\chi^{v_{C_2}^\vee})$ and thus we recover the integral affine structure computation of \S \ref{Zaffine} by taking exponents. 


%

\section{Non-archimedean periods}
\label{period}

In this section, we will define the non-archimedean period map and prove the main theorem (\ref{perequalintro}). We then apply a Torelli theorem to conclude that the $K$-affine structure determines the isomorphism type of the underlying variety. 

\begin{defn}\cite[\S7.4.2]{Kontsevich2006} 
	Let $\pi: X\longrightarrow B$ be a non-archimedean SYZ fibration. Then the smooth locus $B^\sm$ has the $K$-affine structure $[\Aff_K] $ constructed in Proposition \ref{affkexp}. The \textit{non-archimedean period map} is 
	\begin{align*}
\mathcal{P}^\an: H_1(B^\sm,\check{\Lambda}) \longrightarrow H_0(B^\sm,K^\times) = K^\times 
	\end{align*}
	given by the pairing with $[\Aff_K] \in H^1(B^\sm,\Lambda\otimes K^\times)$. 
\end{defn}

\subsection{$K$-affine structure}

Let $(Y,D)$ be a generic Looijenga pair over $K$ satisfying Assumption \ref{assume} and write $U = Y \backslash D$. Let $\Sk(U)$ be the skeleton of the model $(\Yc,\Dc)$ constructed in Proposition \ref{model}. We completely described the singular integral affine structure in \S \ref{Zaffine} and the monodromy of the $K$-affine structure $\Aff_K = \frac{\rho_*(\Oc_{U^\an})}{\ker\Res_\Omega}$ in Proposition \ref{monodromy}.  

We begin by computing the \v{C}ech cocycle of the $K$-affine structure around each non-toric vertex. We follow the notation described in \S\ref{Zaffine} and \S\ref{localsing}. 

In \S\ref{localsing}, we described the transition functions for $\rho_{(\Yc,\Dc)}(\Oc_{U^\an}^\times)$  when moving between the charts given by the stars of 1-dimensional strata. We write the transition functions for $\Aff_K$ below for moving around a singularity $v_X = v_j^i$ in $\Sk(U)$ with $\mu_X = \mu_j^i$:

\begin{table}[H]
\centering
\setlength{\tabcolsep}{10pt} 
\renewcommand{\arraystretch}{2}
	\begin{tabular}{|c|c|}
	\hline 
	$S_1 \cap S_2 $ & $(s_0^{(2)}, s_h^{(2)}) = \left({\mu_Xs_h^{(1)}}, \frac{1}{s_0^{(1)}}\right)$ \\ \hline 
	$S_2 \cap S_3 $ &$(s_0^{(3)}, s_h^{(3)}) = \left(s_h^{(2)}, \frac{1}{s_0^{(2)}}\right)$ \\\hline 
	$S_3 \cap S_4 $ &$(s_0^{(4)}, s_h^{(4)}) = \left({s_h^{(3)}}, \frac{1}{s_0^{(3)}}\right)$ \\ \hline 
$S_1 \cap S_4 $ &$(s_0^{(1)},s_h^{(1)}) = \left(s_h^{(4)}, \frac{1}{s_0^{(4)}\mu_X}\right)$ \\\hline
	\end{tabular}
	\caption{Coordinate transforms for  moving around $v_X$}
		\label{tablesections}
\end{table}

Let $\Sk(U)^{[1]}$ denote the set of 1-dimensional strata of $\Sk(U)$ and $S_\tau=\Star(\tau)$ for $\tau\in \Sk(U)^{[1]}$. We define $S_{\tau,\tau'} = S_\tau \cap S_{\tau'}.$  Then we we have an open cover $\mathcal{U}$ of $\Sk(U)^\sm$ given by  $\mathcal{U} = \{S_{\tau} : \tau \in \Sk(U)^{[1]}\}$. We consider the lift $\beta := (\sigma_\tau)_\tau \in \prod_\tau \Gamma(S_{\tau},\mathcal{H}\text{om}(\text{Aff}_K,K^\times)) $ of $\alpha:= (\text{id})_\tau  \in \prod_\tau \Gamma(S_{\tau},\Hom(K^\times,K^\times))$ as described in \S\ref{cech}.



\begin{notation}
	Given a vertex $v_X = v_j^i$ corresponding to a non-toric irreducible component $X$, let $C_m$ denote the irreducible components of the boundary of $X = Y_j^i$ and $\tau_{C_m}^{(X)}$  the corresponding one dimensional stratum.  Continuing with our previous convention, we will assume that $C_1$ is  the boundary divisor containing the centre of the non-toric blowup on $X$. When it is clear, we will write $\tau_{C_m} = \tau_{C_m}^{(X)}$. 

 \end{notation}

 To completely describe the differential, there are two cases to consider: 

\begin{enumerate}[label=(\roman*)]
	\item if $\tau =\tau_{C_{m}}^{(X)} $ and $\tau' = \tau_{C_{m'}}^{(X)}$ intersect at a common vertex $v_X$. In this case, Table \ref{tablesections} describes the transition functions for the intersection $S_{\tau,\tau'}$;
	\item if $\tau = \tau_{C_{m}}^{(X)}$ and $\tau'= \tau_{C_m}^{(X')}$ with $m=2,4$ an $v_X = v_j^i$ and $v_{X'} = v_{j\pm 1}^i$.  In this case, the transition function is given by $(s_0,s_h) \mapsto (s_0^{-1},s_h)$. 
\end{enumerate}
For case  (i), around a singular vertex  $v_X$, we have 
\begin{align*}
	d_0'(\beta)_{\tau_{C_1},\tau_{C_4}} = \frac{\sigma_{\tau_{C_1}}|_{S_{\tau_{C_1},\tau_{C_4}}}}{\sigma_{\tau_{C_4}}|_{S_{\tau_{C_1},\tau_{C_4}}}}
\end{align*}
which sends \begin{align*}
	s_h^{(1)} \mapsto \frac{\sigma_{\tau_{C_1}}(s_h^{(1)})}{\sigma_{\tau_{C_4}}\left(\frac{1}{\mu_Xs_0^{(4)}}\right)} = \mu_X, \  s_0^{(1)} \mapsto  \frac{\sigma_{\tau_{C_1}}(s_0^{(1)})}{\sigma_{\tau_{C_4}}\left({s_h^{(4)}}\right)} = 1.\end{align*}
Similarly, we have 
\begin{align*}
	d_0'(\beta)_{\tau_{C_2},\tau_{C_1}} = \frac{\sigma_{\tau_{C_2}}|_{S_{\tau_{C_2},\tau_{C_1}}}}{\sigma_{\tau_{C_1}}|_{S_{\tau_{C_2},\tau_{C_1}}}}\end{align*}
given by \begin{align*}
	s_h^{(1)} \mapsto \frac{\sigma_{\tau_{C_2}}\left(\frac{s_0^{(2)}}{\mu_X}\right)}{\sigma_{\tau_{C_1}}\left(s_h^{(1)}\right)}  = \mu_X^{-1}, \ s_0^{(1)} \mapsto \frac{\sigma_{\tau_{C_2}}\left(\frac{1}{s_h^{(2)}}\right)}{\sigma_{\tau_{C_1}}\left({s_0^{(1)}}\right)}  = 1.
\end{align*}
By the same argument,
\begin{align*}
	d_0'(\beta)_{\tau_{C_4},\tau_{C_3}} = 1 \text{ and } d_0'(\beta)_{\tau_{C_3},\tau_{C_2}} = 1. 
\end{align*}

For (ii), we apply exactly the same reasoning to conclude that if $\tau$ and $\tau'$ have adjacent vertices then $d_0'(\beta)_{\tau,\tau'} = 1$.

Let $\phi_{\lambda,c} \in \Hom(
\Aff_K,K^\times)$ denote the image of $\lambda \otimes c \in \Lambda\otimes K^\times$ under the morphism given in (\ref{dualKaff}). 
We then have the following result: \begin{cor}
 The lift $\gamma$ of $d_0'(\beta)$ to $\prod_{ \tau,\tau' } \Gamma(S_{\tau,\tau'},\Lambda\otimes K^\times)$ is given by 
\begin{align*}
	(\check{\gamma}_{\tau,\tau'})_{\tau, \tau'} = \begin{cases}
		\phi_{\mathbf{e}_h^{(1)}, \mu_X} & \text{ if } (\tau, \tau') = (\tau_{C_1}^{(X)},\tau_{C_4}^{(X)}) \\ 
		\phi_{\mathbf{e}_h^{(1)},\mu_X^{-1} } & \text{ if } (\tau, \tau') = (\tau_{C_1}^{(X)},\tau_{C_2}^{(X)})\\
		1 &\text{ otherwise.}
	\end{cases}
\end{align*}
\label{affklog}
\end{cor}

Since all finite intersections of $U_{\tau}$ are contractible and $\Lambda \otimes K^\times$ is constant on $U_{\tau}$, we have $H^i(U_{\tau_0,...,\tau_k}, \Lambda \otimes K^\times) = 0$. In particular, we have explicitly described the representative of the $K$-affine structure in $\check{H}^1(\Sk(U)^\sm,\Lambda\otimes K^\times) = H^1(\Sk(U)^\sm,\Lambda\otimes K^\times)$. 
%
\subsection{Local tropical cycles}

In this section, we will describe the pairing between a local tropical cycle \ie a single loop surrounding a focus-focus singularity with trailing edge depicted below, and the $K$-affine structure we described in the previous section.

We first describe a local tropical cycle with respect to the open cover above. For each loop around a non-toric vertex $v_X$, the tropical cycle consists of four edges between $\tau_{C_m}^{(X)}$ and $\tau_{C_{m+1}}^{(X)}$. The vertices of the edges will lie on the 1-dimensional strata $\tau_{C_m}^{(X)}$. We orient the underlying graph of the tropical cycle in the anticlockwise direction around the vertex $v_{X}$. The trailing edge can be chosen to be contained entirely in an intersection $S_{\tau,\tau'}$ 

There are two choices of decoration (up to sign) of the graph underlying the tropical cycle.  The first is given by equipping each edge with the local invariant cotangent direction, as computed in \S \ref{Zaffine}. Note that in this case, the trailing edge is decorated with the zero vector, and can be discarded.  The other decoration is depicted below where the leading edge away from the singularity is decorated with a primitive invariant cotangent direction. For $X= Y_j^i$,  by Corollary \ref{invariant}, the invariant cotangent direction is parallel to $\Lambda_{\rho_i}$.

Recall, we made the following convention for local tropical cycles in \S\ref{tropicalisationsec}:
\begin{conv}
A counterclockwise loop around a singularity with trailing edge directed towards the origin and decorated with the invariant cotangent direction away from the origin is \textit{positively oriented}. 
\end{conv}



\begin{defn}
	Let $\check{\gamma}$ be a positively oriented local tropical cycle around the vertex $v_X$. Then $\check{\gamma}$ is either a local \textit{invariant cycle} or a local \textit{twisted cycle} as depicted in Figure \ref{localtrop} on the left and right respectively. 
\end{defn}
 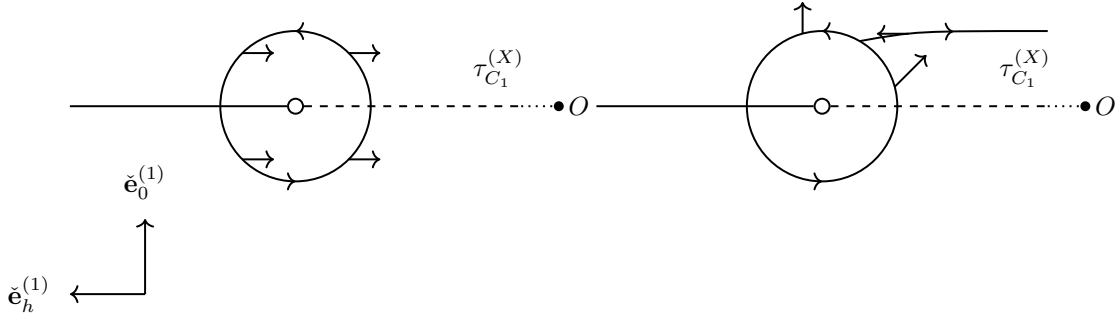
\begin{figure}[H]

	\centering

\tikzset{every picture/.style={line width=0.75pt}} 
\begin{tikzpicture}[circ/.style={shape=circle, inner sep=2pt, draw, node contents=}]
\begin{scope}[xshift = -3.5cm]
	
  	\draw [->] (-2,-2.5) -- (-3,-2.5);   
 	\draw [->] (-2,-2.5) -- (-2,-1.5); 
	\node [label=above:$\check{\mathbf{e}}_{0}^{(1)}$] at (-2,-1.5){}; 
\node [label=left:$\check{\mathbf{e}}_{h}^{(1)}$] at (-3,-2.5){};

	\draw  node (x) at (0,0) [circ]; 
	\draw [ decoration={markings,  mark=at position 0.25 with {\arrow{>}}, mark=at position 0.75 with {\arrow{>}}},         postaction={decorate}](0,0) circle [radius=1cm]; 
 	\draw [dashed](x)-- (3,0); 
 	\draw (x) -- (-3,0);
    \draw[dotted] (3,0) -- (3.5,0);
    \filldraw[black] (3.5,0) circle (1.5pt) node[anchor=west]{$O$};
 	
%
 	\draw [->] (-135:1cm) -- ([xshift=12pt, yshift=0pt] -135:1cm); 
 	\draw [->] (-45:1cm) -- ([xshift=12pt, yshift=0pt] -45:1cm); 
 	\draw [->] (45:1cm) -- ([xshift=12pt, yshift=0pt] 45:1cm); 
 	\draw [->] (135:1cm) -- ([xshift=12pt, yshift=0pt] 135:1cm); 
 	
%
    \node [label=above:$\tau_{C_1}^{(X)}$] at (2.7,0){};


%
\end{scope}
\begin{scope}[xshift = + 3.5cm]

\draw  node (x) at (0,0) [circ]; 
	\draw [ decoration={markings,  mark=at position 0.25 with {\arrow{>}}, mark=at position 0.75 with {\arrow{>}}},         postaction={decorate}](0,0) circle [radius=1cm]; 
 	\draw [dashed](x)-- (3,0); 
 	\draw (x) -- (-3,0); 
 	
 	
 	\draw [->] (40:1.5cm) -- ([xshift=-12pt, yshift=0pt] 40:1.5cm); 
 	\draw [->] (15:1cm) -- ([xshift=12pt, yshift=12pt] 15:1cm); 
 	\draw [->] (105:1cm) -- ([xshift=0pt, yshift=12pt] 105:1cm); 
 	\draw  [decoration={markings, mark=at position 0.5 with {\arrow{>}}}, postaction={decorate}](60:1cm) .. coordinate [pos=.3] (a) controls (1.25,1) .. ([xshift=3cm, yshift=1cm] 0:0cm ); 
 	
%
\node [label=above:$\tau_{C_1}^{(X)}$] at (2.7,0){};
\draw[dotted] (3,0) -- (3.5,0);
    \filldraw[black] (3.5,0) circle (1.5pt) node[anchor=west]{$O$};

%

\end{scope}
\end{tikzpicture} 
	\caption{Local tropical cycles.}
	\label{localtrop}
\end{figure}

\begin{prop}
	Let $\check{\gamma}$ be a local positively oriented tropical cycle surrounding the vertex $v_X$. Then the pairing of $\check{\gamma}$ with the \v{C}ech cocycle representing the $K$-affine structure is 
	\begin{align*}
		\langle \check{\gamma}, [\Aff_K]\rangle = \begin{cases}
			1 & \text{ if } \gamma \text{ is an invariant cycle}\\ 
			\mu_X &\text{ if } \gamma \text{ is a twisted cycle}
		\end{cases}
	\end{align*}
	where $\langle\cdot  ,\cdot \rangle : \check{\Lambda}\otimes (\Lambda\otimes K^\times) \longrightarrow K^\times$.  
	\label{localpairing}
\end{prop}

\begin{proof}

	Corollary \ref{affklog} gives a description of the representative of the $K$-affine structure on each intersection $S_{\tau,\tau'}$. In particular, the only regions $S_{\tau,\tau'}$ which will contribute non-trivially to the pairing will be $(\tau, \tau') =  (\tau_{C_1}^{(X)},\tau_{C_4}^{(X)})$ or $(\tau_{C_1}^{(X)},\tau_{C_2}^{(X)})$.
	
	If $\gamma$ is an invariant cycle, then the cotangent vector decorating it is given by the invariant direction as depicted in Figure \ref{localtrop}. By the integral affine structure computation in \S \ref{Zaffine}, the invariant cotangent direction is parallel to $\check{\mathbf{e}}_h^{(1)}$; assume $\check{\gamma}$ is equipped with the cotangent vector $\check{\mathbf{e}}_h^{(1)}$. Since the $K$-affine structure is trivial on $S_{\tau_{C_2},\tau_{C_3}}$ and $S_{\tau_{C_3},\tau_{C_4}}$, it does not contribute to the pairing. For $(\tau, \tau') =  (\tau_{C_1}^{(X)},\tau_{C_4}^{(X)})$, we have  $\gamma_{\tau,\tau'} = \phi_{{\check{\mathbf{e}}^{(1)}_h},\mu_X}$ and thus the pairing of $\check{\gamma}$ with $\gamma_{\tau,\tau'}$ is $\mu_X$. Similarly, for $(\tau,\tau') =(\tau_{C_1}^{(X)},\tau_{C_2}^{(X)})$, we have $\gamma_{\tau,\tau'} = \phi_{\check{\mathbf{e}}_h^{(1)}, \mu_X^{-1}}$ and thus the pairing of $\check{\gamma}$ with $\gamma_{\tau,\tau'}$ is $\mu_X^{-1}$. Putting this together, the pairing of the $K$-affine structure with the local invariant cycle $\check{\gamma}$ is $\mu_X \cdot \mu_X^{-1} = 1$.

	If $\gamma$ is a twisted cycle as depicted above, then the cotangent vector decorating $\check{\gamma}$ on $U_{\tau,\tau'}$ is $\check{\mathbf{e}}_0^{(1)}$ if $ (\tau, \tau')=  (\tau_{C_1}^{(X)},\tau_{C_4}^{(X)})$ and 
			$-\check{\mathbf{e}}_h^{(1)} + \check{\mathbf{e}}_0^{(1)}$ if $(\tau, \tau') = (\tau_{C_1}^{(X)},\tau_{C_2}^{(X)})$. Thus the pairing with the $K$-affine structure is $(\mu_X^{-1})^{-1} = \mu_X$ from the contribution on $S_{\tau_{C_1}^{(X)},\tau_{C_2}^{(X)}}$. 
			
	
	\end{proof}


\subsection{Non-archimedean period map}

We will now use the description of the local pairing of the $K$-affine structure with local tropical cycles to completely describe the non-archimedean period map.  

\begin{defn}
	Let $\gamma \in H_1(\Sk(U)^\sm,\check{\Lambda})$. We say $\gamma$ is 
	\begin{enumerate}
		\item an \textit{invariant cycle} if it is given by a loop around a singularity decorated with an invariant cotangent vector; 
		\item an \textit{winged tropical cycle} if under the tropical correspondence \ref{tropcorr}, $\gamma$ corresponds to a simple tropical wing; 
		\item a \textit{spoked tropical cycle} if under the tropical correspondence \ref{tropcorr}, $\gamma$ corresponds to a simple tropical spoke. 
	\end{enumerate}
\end{defn}

Note it is enough to compute the non-archimedean period map for winged tropical cycles and spokes. 

\begin{thm}
\label{naper}
Let 
\begin{align*}
	\Sk(U)^\text{sing} \coloneqq \{v_X : X \text{ is a non-toric irreducible component of } \Yc_k\}
\end{align*}
be the set of singular vertices in $\Sk(U)$. Then the non-archimedean period map $$\mathcal{P}^\an : H_1(\Sk(U)^\sm, \check{\Lambda}) \longrightarrow K^\times,$$ giving by the pairing with $[\Aff_K]$ is as follows: 
	\begin{enumerate}
		\item if $\check{\gamma}$ is an invariant cycle, then $\mathcal{P}^\an=1$;
		\item if $\check{\gamma}$ is a winged tropical cycle which surrounds $v_X$ with positive orientation and $v_{X'}$ with negative orientation, then $\mathcal{P}^\an(\check{\gamma}) = \frac{\mu_{X}}{\mu_{X'}}$;
		\item if $\check{\gamma}$ is a spoked tropical cycle which consists of a $\epsilon_X$-oriented loop around $v_X$ with  $X\in S \subseteq \Sk(U)^\text{sing}$,  then \begin{align*}
			\mathcal{P}^\an(\check{\gamma}) = \prod_{X \in S }\left(\mu_X\right)^{\epsilon_X}. 
		\end{align*}
	\end{enumerate}
\end{thm}

\begin{proof}

This follows immediately from Proposition \ref{localpairing} where we described the pairing for local tropical cycles and noting that the emanating edges  do not contribute to the pairing as they are decorated with the invariant cotangent vector. 
%
\end{proof}

This implies the following corollary:

\begin{cor}
	The non-archimedean period map $\mathcal{P}^\an$ factors through $H_1(\Sk(U),\iota_*\check{\Lambda})$. 
\end{cor}

\subsection{Comparison between algebraic periods and non-archimedean periods} 

We now complete the proof of Theorem \ref{perequalintro}, which says that the non-archimedean period map agrees with the algebraic period map for generic log Calabi-Yau surfaces.


We recall the definition of the algebraic period map 
\begin{align*}
	\mathcal{P}: D^\perp &\longrightarrow \Pic^0(D) \stackrel{\ref{gm}}\cong \G_m \\ 
	L &\longmapsto L\rvert_D.
\end{align*}
The line bundle $L|_D $ is of multidegree $(0,\dots,0)$ and we can write the divisor $\ell_D$ associated to $L|_D$ as $\ell_D = \sum_{i=1}^n\sum_{j=1}^{r_i} (q_{ij} - p_{ij})$ with $q_{ij}, p_{ij} \in D_i^{\circ}$. By the isomorphism $\psi: \Pic^0(D)\cong \G_m$ constructed in Lemma \ref{gm}, we can rewrite this  as $$\ell_D = \sum_{i\in I} \epsilon_{i}(q_{j_i}^i - p_{j_i}^i)$$ with $q_{j_i}^i, p_{j_i}^i \in D_i^{\circ}$, where $q_{j_i}^i$ is the centre of a non-toric blowup on $D_i$ corresponding to the non-toric exceptional curve  $E_{j_i}^i$,  $I \subset \{1,..., \ell(D)\}$ and $\epsilon_i \in \Z$. 

Suppose $\ell_D \equiv \ell_D' := \sum_{i\in I} \epsilon_{i}(q_{j_i}^i - \bar{p}_{j_i}^i)$. By Lemma \ref{gm} we have $\prod_i{z_i(p_{j_i}^i)^{\epsilon_i}}=\prod_i{z_i(\bar{p}_{j_i}^i)^{\epsilon_i}} $ where $z_i$ is an affine coordinate on $D_i$. Thus for each $i$ we can fix a $p_i \in D_i^\circ$ which is not a centre of a non-toric blowup on $D_i$ and write $\ell_D = \sum_{i \in I} \epsilon_i(q_{j_i}^i - p_i)$.  From now on, we choose $z_i$ such that it is compatible with the choice of coordinates we made in \S \ref{localsing} \ie $z_i(q_j^i) = \mu_j^i$  and we have $z_i(p_i) = 1$. 


We can now prove the main result of this thesis.

\begin{thm}
	The non-archimedean period map $\mathcal{P}^\an: H_1(\Sk(U),i_*\Lambda) \longrightarrow K^\times$ agrees with the algebraic period map $\mathcal{P}: D^\perp \longrightarrow K^\times$ \ie the following diagram commutes.
	\begin{figure}[H]
	\begin{center}
		\begin{tikzcd}
			H_1(\Sk(U),i_*\Lambda)_{\text{tf}}\ar[rr,"\mathcal{P}^\an"]\ar[rd,"\sim"]\ar[rd] & & K^\times \\ 
			&D^\perp \ar[ur,"\mathcal{P}"] & 
		\end{tikzcd}
	\end{center}
	\end{figure}
	\label{perequal}
\end{thm}

\begin{proof}
Let $\check{\gamma}$ be a spoked tropical cycle and $L_{\check{\gamma}}$ the corresponding element in $D^\perp$. Then by the description above, we can assume $L_{\check{\gamma}}|_D = \Oc_D( \sum_{i \in I} \epsilon_{i}(q_{j_i}^i-p_{i}))$.  We then have \begin{align*}
	\mathcal{P}(L_{\check{\gamma}}) = \prod_{i \in I} \left(\frac{z_{i}(q_{j_i}^i)}{z_{i}(p_{i})}\right)^{\epsilon_{i}} =  \prod_i (\mu_{j_i}^i)^{\epsilon_{i}}.  
\end{align*}
By Remark \ref{orientmatch}, we have that $\check{\gamma}$ has a $\epsilon_i$-oriented local twisted tropical cycle around $v_{j_i}^i$. Thus $\mathcal{P}^\an(\check{\gamma}) = \mathcal{P}(L_{\check{\gamma}})$.

Let $\check{\gamma}$ be a winged tropical cycle and $L_{\check{\gamma}}$ the corresponding element in $D^\perp$. We can assume  $L_{\check{\gamma}}|_D = q_{j}^i- q_{k}^i$ and thus 
\begin{align*}
	\mathcal{P}(L_{\check{\gamma}}) = \frac{z_i(q_{j}^i)}{z_i(q_{k}^i)} = \frac{\mu_{{j}}^i}{\mu_{{k}}^i},
\end{align*}
which agrees with $\mathcal{P}^\an(\check{\gamma})$ since $\check{\gamma}$ has a positively oriented local twisted cycle around $v_j^i$ and a negatively oriented  local twisted cycle around $v_k^i$ by Remark \ref{orientmatch}.  Having checked the diagram on generators, we conclude the diagram commutes. This completes the proof of the theorem.

%

\end{proof}

The following result realises the motivating theme of this paper: that the essential skeleton and a $K$-affine structure on it remember a great deal of geometric information. 
\begin{cor} Let $(Y,D)$ be a generic Loojenga pair satisfying Assumption \ref{assume}. Then the $K$-affine structure $\Aff_K$ on $\Sk(U)$ determines the isomorphism type of $U$.
	\label{isocorr}
\end{cor}

\begin{proof}
	By Theorem \ref{perequal}, the non-archimedean period map computes the period point $D^\perp \rightarrow K^\times$. This determines an element of the moduli stack of Looijenga pairs with a marking of $\Pic(Y)$ by \cite[Proposition 2.9 + \S6]{gross_hacking_keel_2015}. In particular, it determines the isomorphism type of $U= Y\backslash D$. 
\end{proof}
\subsection{Analytic continuation to the non-generic case}

In this section, we seek to remove the assumption of genericity on the Looijenga pair $(Y,D$).

\begin{defn}
	Let $\pi: (\mathbf{Y},\mathbf{D}) \rightarrow S$ be a family of Looijenga pairs over a reduced connected space $S$. A \textit{marking} of the family $\pi$ is an isomorphism $R^2\pi_*\underline{\Z} \cong \underline{\Upsilon}$ where $\Upsilon$ is a lattice isomorphic to $H^2(\mathbf{Y}_s,\Z)$ for some $s \in S$ and $\underline{\Upsilon}$ is the constant local system on $S$ with fibre $\Upsilon$. A marking  exists if and only if the monodromy homomorphism of the family is trivial. Such families will be referred to as \textit{marked}. 
\end{defn}

We have the following proposition concerning uniqueness of analytic continuation. 

\begin{prop}[\cite{bassat} Prop. 4.2]
	Let $X$ be a smooth connected strictly $K$-analytic space and $Z$ a non-empty affinoid subdomain of $X$. Then the restriction map $\Oc(X) \rightarrow \Oc(Z)$ is injective. 
	\label{analyticcont}
\end{prop}

Let $(\overline{Y},\overline{D})$ be a toric Looijenga pair. Let $N> 0$ and  $(k_1,...,k_{\ell(D)}) \in \Z^{\ell(D)}_{\geq 0}$ be a partition of $N > 0$ \ie $\sum_i k_i = N$. Let $\G_{m,\circ}^n \subset \G_m^n$ denote the elements $(x_i) \in \G_m^n$ such that $|x_i|_K\leq 1$.   We can then define the marked family of Looijenga pairs $$\pi : (\mathbf{Y},\mathbf{D}) \rightarrow \prod_{i=1}^{\ell(D)} \G_{m,\circ}^{k_i} =\G_{m,\circ}^N$$ such that the fibre of $(p_{ij}) \in \prod_{i=1}^{\ell(D)}\prod_{j=1}^{k_i}\G_{m,\circ}$ is given by the Looijenga pair arising from non-toric blowups of $(\overline{Y},\overline{D})$ at $p_{ij} \in \overline{D}_i^\circ$. Fixing a basis for $\Upsilon$, we have an isomorphism $\Hom(\Upsilon,K^\times) \cong (K^\times)^r$ where $r = \rank(\Upsilon)$. Define the following closed subvariety of $\G_{m,\circ}^N$:

\begin{align*}
	\Delta := \{ (x_{ij}) \in \prod_{i=1}^{\ell(D)}
	\prod_{j=1}^{k_i}\G_{m,\circ} \ : \ \overline{x}_{ij} = \overline{x}_{ik} \} 
\end{align*}
where $\overline{x}_{ij}$ denotes the reduction of $x_{ij}$ in $\Oc_K/(t) = \C$.  
This set is precisely the points $(p_{ij})$ for which the specialisations coincide; this also includes the case of infinitely near points \ie  when we perform a non-toric blowup at the same point more than once and thus the generic fiber is no longer a generic Looijenga pair.

The non-archimedean period map is defined on the generic locus of this family \ie away from $\Delta$. The map can be extended naturally to all of $\G_{m,\circ}^N$ by the formula given in Theorem \ref{naper}.  We thus have a map 
\begin{align*}
	\G_{m,\circ}^N \longrightarrow (K^\times)^r 
\end{align*}
which we can  analytify to produce 
\begin{align*}
	\nu: (\G_{m,\circ}^N)^\an \longrightarrow (K^\times)^r. 
\end{align*}
The composition with projection onto each factor of $(K^\times)^r$ is analytic and thus $\nu$ is analytic. The non-archimedean period map has been defined on $(\G_{m,\circ}^N)^\an\backslash \Delta^\an$ and by Proposition \ref{analyticcont}, any extension to $(\G_{m,\circ}^N)^\an$ is unique. Thus the non-archimedean period can be uniquely extended to cover the non-generic case. As a consequence, we have the following theorem:

\begin{thm}
	Let $(\Yc,\Dc)$ be a Looijenga pair over $R$ with good reduction such that the combinatorial type of the generic fibre and special fibre agree. Then the non-archimedean period map for $U= (\Yc\backslash \Dc)_K$ agrees with the algebraic period map for $U$. 
    \label{perequalfinal}
\end{thm}

\begin{proof}
	Given a generalised exceptional curve $C_1+...+C_{k-1}+ E$ arising from blowing up infinitely near points on $D$, we have $C_i \in D^\perp$ for $1\leq i\leq k-1$. Since $C_i$ is disjoint from $D$, we have that $\mathcal{P}(C_i) = 1$. This agrees with the non-archimedean period map and thus the proof is complete. 
\end{proof}

\bibliographystyle{alpha}
\bibliography{bib}

\end{document}